\documentclass[12pt]{article}
\usepackage{graphicx, fullpage}
\usepackage{amsmath,amssymb,verbatim,amsthm,bbm}
\usepackage{natbib}
\setcitestyle{authoryear,round}
\usepackage{xcolor}
\usepackage{caption}
\usepackage{dsfont}
\usepackage{subcaption}
\usepackage{hyperref}
\usepackage{listings}
\usepackage{todonotes}

\newcommand\independent{\protect\mathpalette{\protect\independenT}{\perp}}
\def\independenT#1#2{\mathrel{\rlap{$#1#2$}\mkern2mu{#1#2}}}

\usepackage{amsfonts} %% <- also included by amssymb
\DeclareMathSymbol{\shortminus}{\mathbin}{AMSa}{"39}

\newtheorem{theorem}{Theorem}[section]

\newtheorem{defin}{Definition}[section]
\newtheorem{lemma}{Lemma}[section]
\newtheorem{example}{Example}[section]
\newtheorem{corollary}{Corollary}[section]

\newtheorem{remark}{Remark}[section]

\newcommand{\Ex}{\mathrm{E}}
\newcommand{\cov}{\mathop{\rm {\mathbb C}ov}\nolimits}%

\newcommand{\C}{\mathds{C}}

\def\m{\mathcal}
\def\mb{\mathbb}
\def\mr{\mathrm}

\definecolor{cobalt}{rgb}{0.0, 0.28, 0.67}

\newcommand{\bX}{\mathbf{X}}

\title{Conditionally specified 
graphical modeling of stationary multivariate time series}
\author{{Anirban Bhattacharya$^1$, Jan Johannes$^2$ and
    Suhasini Subba Rao$^1$}}
\date{$^1$ Texas A\&M University, TX-77845, USA. \\%
  $^2$ Universit\"at Heidelberg, D-69120, Germany \\[2ex]%
  August 19, 2025}
\begin{document}

\maketitle

\begin{abstract}
Graphical models are ubiquitous for summarizing conditional relations
in multivariate data. In many applications involving multivariate time
series, it is of interest to learn an interaction graph that treats
each individual time series as nodes of the graph, with the presence
of an edge between two nodes signifying conditional dependence given
the others. Typically, the partial covariance is used as a measure of
conditional dependence. However, in many applications, the outcomes
may not be Gaussian and/or could be a mixture of different
outcomes. For such time series using the partial covariance as a
measure of conditional dependence may be restrictive. In this article,
we propose a broad class of time series models which are specifically
designed to succinctly encode process-wide conditional independence in
its parameters. For each univariate component in the time series, we
model its conditional distribution with a distribution from the
exponential family. We develop a notion of process-wide compatibility
under which such conditional specifications can be stitched together
to form a well-defined strictly stationary multivariate time
series. We call this construction a conditionally exponential
stationary graphical model ({\it CEStGM}). A central quantity
underlying CEStGM is a positive kernel which we call the interaction
kernel. Spectral properties of such positive kernel operators
constitute a core technical foundation of this work. We establish
process-wide local and global Markov properties of CEStGM exploiting a
Hammersley-Clifford type decomposition of the interaction
kernel. Further, we study various probabilistic properties of CEStGM
and show that it is geometrically mixing. An approximate Gibbs sampler
is also developed to simulate sample paths of CEStGM.

\end{abstract}

\section{Introduction}

%Graphical models use a graph-based representation to depict conditional dependence between variables in a  
%multivariate random vector. 
Markov random fields \citep{p:bro-64,p:cliff-ham-71,p:grimmett-73,p:bes-74}, also called (undirected) graphical models \citep{b:lau-96},
are  widely popular %a flexible method 
for flexibly modeling high dimensional probability distributions. The resulting distribution readily leads to a graph-based representation which depicts the conditional dependence between variables in a  
multivariate random vector. Typically, %the 
these conditional distributions are from the exponential family, which allow for mixtures of different variable types \cite{p:lau-wer-89,p:yan-14}. 
Distributions based on the exponential family satisfy the classical Markov properties (pairwise, local and global).
%further the graph has the desirable property that it is closed under marginalization, \cite{p:ric-spi-02}. {\color{blue}Should we keep the last comment? Okay with removing it -- AB}   
There exists extensive research in undirected 
graphical models
where the main 
focus has been to model independent, identically distributed multivariate random vectors;  see \cite{b:lau-96,b:kol-09} for book-level treatments. However, in %many 
diverse fields such as medicine, finance, 
psychology and the biological sciences, these vectors are observed over time, and the assumption of temporal independence is unlikely to hold. 
%In %this case, 
%such cases, 
%Consequently, the aforementioned graphical modeling approaches 
%would only model contemporaneous conditional dependence, thus omitting important conditional temporal dependencies. 
\cite{p:kol-10} and \cite[Sections 2.5 \& 3.1]{p:has-20} both consider graphical modeling of the  contemporaneous conditional dependence of multivariate  mixed time series data. 
They are motivated by several real examples, including
the gene expression data of a common fruit fly observed over its life time and a momentary assessment state study from an individual diagnosed with major depression, where their momentary state was recorded daily over a period of time. Clearly, an approach based on contemporaneous conditional dependence omits important conditional temporal dependencies that are likely to exist in these time series.
% which can lead to a spurious graph. 
%which can lead to a misleading graph. %[write applications here]
In such applications, it may be more meaningful to report a `process-wide' graph which treats
the entire time series, corresponding to each variable, as nodes of the graph.

%A natural extension of undirected graphs to multivariate time series $\{X_{t};t\in \mathbb{Z}\}$ 
%would require one to model the joint density of %the infinite dimensional process $\{X_{t};t\in \mathbb{Z}\}$. Unfortunately, this is not possible as it is well known that a joint density of a stationary distribution over $\mathbb{Z}$ {\ccb with respect to the product measure} {\color{red}I think we could just say product measure to allow for the point measures?}
%{\ccb does not} exist [cite result, Jan is onto this]. 

Motivated by the compelling range of applications, 
we propose a general multivariate stationary time series framework that preserves the pertinent features of an undirected graph and seamlessly allows modeling high-dimensional time series; potentially consisting of a combination of proportions, positive, categorical and
count variables.
%a straightforward manner.
The limited body of work on undirected graphical models for time series has so far
focused on the class of 
Gaussian linear time series; pioneered by \cite{p:dah-00}, where he builds
%{\color{blue}To the best of our knowledge, there does not exist a general approach for building undirected graphs for multivariate time series.}
%Therefore, 
%it is of interest to ask, how much of the classical graphical model framework for multivariate random vectors can we transfer to the multivariate time series setting? An important contribution that helps, in part, to answer this question is 
%\cite{p:dah-00}. 
%We take inspiration from \cite{p:dah-00}
%An important contribution that helps, in part, to answer this question is 
%\cite{p:dah-00} 
an undirected graph for multivariate stationary time series in terms of its partial correlation by placing an edge between two nodes if there exists a non-zero partial correlation. %between the nodes.
% . He proposes to build an undirected graph where each node corresponds to a component in the multivariate time series and an edge is placed between two nodes if there exists a non-zero partial correlation between the nodes. 
\cite{p:dah-00} shows that the graph satisfies the classical Markov properties in a process-wide partial correlation sense. As
the non-zero partial correlations are encoded in the inverse spectral density matrix, 
classical spectral techniques can be used to estimate the graph (c.f. \citep{p:eic-03,p:eic-08,jung2015graphical,p:tank-15,p:fie-18,p:tug-22,p:dey-22, p:fas-24,
p:bas-24, p:kra-25}). See \cite{p:gather-02} and \cite[Sections 7 \& 8]{p:tank-15} for some representative applications.
However, the underlying time series must be Gaussian to translate partial non-correlation into process-wide conditional independence between components of the time series. In many applications, the time series is not Gaussian and a pre-processing step is often used in an attempt  to `Gaussianize' 
the data c.f. \cite{p:tug-25}. 

In light of these considerations, we
define a new class of multivariate, mixed, time series models %for modelling  
where the parameters encode a `process-wide' conditional independence graph. 
Motivated by \cite{p:bes-74,p:yan-15}, we begin with a {\it conditional specification} where the conditional distributions at each node and at each time point are modeled via appropriate exponential families. 
%In addition to our overarching motivation, such a conditional specification also allows modeling of mixed, {\color{blue}high dimensional}, multivariate time series in a straightforward manner. 
A key contribution of this article is to delineate simple verifiable conditions on the parameters of these distributions to ensure the existence of a valid stationary multivariate stochastic process with these pre-specified node-wise conditionals. We call this {\it process-wide compatibility}, extending the notion of {\it compatibility} of a finite collection of conditional distributions \citep{p:bes-74,p:hob-98,p:arn-02}.  
In Section \ref{sec:cstgm}
we give an explicit construction of the multivariate stochastic process, which we refer to as {\it CEStGM} (conditionally exponential stationary graphical model), by specifying its finite-dimensional projections and verifying Kolmogorov's consistency theorem \citep{billingsley2017probability}.
The nature of the conditional specification prevents 
%the application of probabilistic Markov  process techniques
the verification of Doeblin-type conditions on a Markov transition kernel
to show existence of a stationary process, c.f. \citep{b:dou-14,b:douc2018markov,p:dou-19,p:tru-21,p:dou-neu-tru-23}. Instead, the crucial ingredient
%  based on verifying whether a time series has a stationary solution based on its Markov transition kernel cannot be used.}
in the construction of CEStGM is a quantity that we dub the {\it interaction kernel}, which is a positive kernel (though importantly not a Markov transition kernel) involving the parameters appearing inside the conditional specifications. 
The proposed framework includes the multivariate Gaussian time series, as defined by \cite{p:dah-00}, as a special case.  We additionally provide a number of explicit non-Gaussian constructions, including mixed data types. In Section \ref{sec:graph},
we arrive at our original motivation of showing that sparsity in the model parameters encode a notion of process-wide conditional independence,
which we define in terms of the sigma-algebras generated by the respective processes.
In Section \ref{sec:prob_cestgm}
we show the CEStGM process is ergodic and geometrically
$\beta$-mixing under mild conditions, and develop a Gibbs sampling based approximate sampling algorithm to simulate sample paths of the CEStGM process.
In Section \ref{sec:d_nn} we give an extension to higher order temporal dependence. All proofs can be found in the supplementary material.

The core technical foundation of this work largely builds on the spectral theory of linear integral operators with strictly positive kernel functions \citep{schaefer1974banach,boelkins1998spectral}. Curiously, if such an operator is {\it compact}, then its dominant eigenfunction is strictly positive due to the celebrated Krein--Rutman theorem \citep{krein1948linear,krein1962linear}, which generalizes the Perron--Frobenius theory for positive matrices to abstract Banach spaces. Appealing to this theory, the dominant eigenfunctions of the interaction kernel and its adjoint are rendered pivotal to the CEStGM formulation. To corroborate compactness,
one could verify a stronger, but more amenable, Hilbert--Schmidt condition on the interaction kernel; this is what we apply in all the examples.  Proving the aforementioned probabilistic properties of CEStGM such as $\beta$-mixing %, mentioned above, 
organically necessitates manipulations with powers of the interaction kernel in an operator theoretic sense, which interestingly shares a high-level resemblance with iterated Markov kernels. Specifically, we establish and crucially exploit a {\it power-iteration} flavored result for powers of compact {\it non} self-adjoint operators, which may be of independent interest. The local and global Markov properties are derived using a {\it Hammersley--Clifford type factorization} of the interaction kernel. En route, we prove a general result about process-wide conditional independence of $\beta$-mixing multivariate time series. As far as we are aware, the probabilistic tools we develop in this article are new and may be of independent interest.

% {\color{blue}Edit}
% The dominant eigenfunctions of this interaction kernel play a crucial 
% {\color{blue}fundamental}
% role in our construction, which are guaranteed to be positive due to the celebrated Krein--Rutman theorem \citep{krein1962linear}. 

% {\color{blue}In Section \ref{sec:prob_cestgm}}
% we show the CEStGM process is {\color{blue} ergodic and geometrically}
% $\beta$-mixing under mild conditions, and develop a Gibbs sampling based approximate sampling algorithm to simulate sample paths of the CEStGM process. Both these results crucially exploit a {\it power-iteration} flavored result for powers of compact non self-adjoint operators, which may be of independent interest. 

% Using the $\beta$-mixing result, we arrive at our original motivation of showing that sparsity in the model parameters encode a  notion of process-wide conditional independence, which we define for in terms of the sigma-algebras generated by the respective processes. {\ccb Equipped with this notion of process-wide conditional independence and exploiting a Hammersley--Clifford type factorization of the interaction kernel}, we establish local and global Markov properties of CEStGM in Section \ref{sec:graph}. {\ccb En route, we prove a general result about process-wide conditional independence of $\beta$-mixing multivariate time series that may be of independent interest.} 

%{\cred (there is a lot to say here (for example, the HC decomposition of $R$), so I am looking for suggestions as to how much details we should provide. alternatively, we can merge the two paragraphs.)}

\emph{Broader literature on count and other non-continuous response time series.} Over the past forty years, 
several approaches for modeling
`non-Gaussian'  time series have been proposed. 
They include
integer-valued autoregressive (INAR) models 
\citep{p:mck-85,p:alz-al-90,p:fra-sub-95,p:joe-96,p:lat-97, b:wei-18}, and  
observation driven models that use a GLM approach to model the conditional distribution of the time series given the past (\citep{p:fei-81,
p:zeg-qaq-88, p:ben-rig-sta-03,
p:dav-03,p:fok-04,p:lat-06,p:fok-09,p:neu-11,p:deb-tru-24,p:tru-24} and the recent review in \cite{p:dav-21}).
More recently, copula and latent model approaches have been proposed \cite{p:tru-23}, \citep{p:bed-02,p:bea-15,p:mag-22}, \citep{p:jia-23, p:kon-lun-24,p:kim-due-fis-pip-24}. Each model has its own advantage, but 
%Despite their clear advantages 
only a few of the above models are easily scalable to high-dimensions and, 
as far as we are aware,
none encode process-wide conditional independence.  

\subsection*{Preliminary groundwork} %{\cred (More apt section title? -- AB)}\textcolor{blue}{Agreed, but what?}
We begin by introducing key notation and summarizing two distinct strands of research in the theory
and modeling of (undirected) conditional independence/partial correlation graphs.
We blend these ideas to define CEStGM in Section \ref{sec:cstgm}.
  %{\cred (Starting from the second line above, the material can perhaps also go inside the next paragraph, since we are only talking about CI graph, and not the PC graph here. -- AB)}{\color{blue}Agreed}

% Define a graph $G=(V,E)$ with vertex set $V = [p] :\,= \{1, \ldots, p\}$ and edge set $E \subset V\times V$
% A graph $G=(V,E)$ consist of the vertices
% $V = \{1,\ldots,p\}$ and the edge set $E \subset V\times
% V$, the graph is undirected so if $(a,b)\in V$ then $(b,a)\in V$.
% Given $X = (X^{(1)},\ldots,X^{(p)})$ and $X_{V\backslash
% \{a,b\}} = (X^{(c)};c\{1,\ldots,p\}\backslash\{a,b\})$. 
% The conditional independence graph is defined by the following
% rule: $(a,b)\notin E$ if and only if $X^{(a)}\independent X^{(b)}|X_{V\backslash
%   \{a,b\}}$.
%Depending on the context, $X^{(a)}\independent X^{(b)}|X_{V\backslash
%\{a,b\}}$ can mean conditionally independent or no partial correlation between 
%$X^{(a)}$ and $X^{(b)}$ given $X_{V\backslash
%\{a,b\}}$. 

We start with the formal definition of conditional
independence. Suppose
$(\Omega,\mathcal{G},P)$ is a probability space and $\mathcal{G}_{1}$,
$\mathcal{G}_{2}$ and $\mathcal{G}_{3}$ are sub-sigma algebras of
$\mathcal{G}$. We define the set of functions
$ [\mathcal{G}_{i}]^{+}  = \{ \textrm{all positive, bounded }
  \mathcal{G}_{i}\textrm{-measurable functions}\}$. 
     Then  $\mathcal{G}_{1}\independent \mathcal{G}_{2}|\mathcal{G}_{3}$
iff for  all $m_{1}\in [\mathcal{G}_{1}]^{+}$ and $m_{2}\in [\mathcal{G}_{2}]^{+}$
\begin{eqnarray*}
\Ex[m_{1}m_{2}|\mathcal{G}_{3}] =
  \Ex[m_{1}|\mathcal{G}_{3}]\Ex[m_{2}|\mathcal{G}_{3}],
\end{eqnarray*}  
almost surely (see Theorem 2.2.1 in \cite{b:flo-mou-rol-90}). If
$\mathcal{G}_{i} = \sigma(Y_{i})$,  where $Y_{i}$ is a finite
collection of random variables then we often use the notation
 $Y_{1}\independent Y_{2}|Y_{3}$. For finite random vectors with strictly positive densities, this is
  typically stated in terms of joint densities (see Proposition 2.2.1,
  \cite{b:lau-20}) but we require this more
  general definition later on in the paper to define process-wide versions of conditional independence. %{\color{blue}Is this okay?}

\paragraph*{Graphical models using the conditional exponential family}
We first define a conditional independence graph.
%{\color{blue}Probably bring this definition before the subheading}
Let $X = (X^{(1)},\ldots,X^{(p)})$ be a collection of $p$ scalar
 random variables, with
 $X^{(a)} \in \m X^{(a)}$. For $a, b \in [p] :\,= \{1, \ldots, p\}$, define $X_{V\backslash
  \{a,b\}} = (X^{(c)} \,; \, c \in [p]\backslash\{a,b\})$. The {\it conditional independence graph} associated with $X$ is an undirected graph $G=(V,E)$ with vertex set $V = [p]$ and edge set $E \subseteq V\times
V$ with the following rule: $(a,b)\notin E$ if and only if $X^{(a)}\independent X^{(b)}|X_{V\backslash
  \{a,b\}}$.
  
For %$X = (X^{(1)},\ldots,X^{(p)})^{\top}$ and 
$a \in [p]$, let $X_{-a} = (X^{(c)} \,:\, c \ne a)^\top$ denote the random vector with the entry $X^{(a)}$ removed from $X$. 
\cite{p:bes-74} first proposed using distributions from the
exponential family as a method for modeling conditional interactions $X^{(a)} \mid X_{-a}$, and obtained sufficient conditions for these collection of conditional distributions across $a \in [p]$ to be {\it compatible}, that is, define a valid unique joint probability distribution. 
%between 
%different random variables. 
Subsequently, it has been used to model
mixed Binary/Gaussian random vectors in  
\cite{p:lau-wer-89} and more recently in the high dimensional set-up
in, for example,  \cite{p:wai-03}, \cite{p:rav-10}, 
  \cite{p:jal-11}, \cite{p:yan-14}, \cite{p:yan-15}, \cite{p:foy-15}, \cite{p:lev-17},
  \cite{b:wai-19}. We mention that the early statistical literature on conditional interactions was motivated by problems in statistical mechanics  
%starting with the Ising  and Potts  model, 
\citep{b:geo-11}. 
% The basic premise is as follows. The
% conditional distribution is from the exponential family which can be
% written in the   following natural exponential form
The basic premise is to assume that each  conditional distribution is from the {\it exponential family}, written in the following natural exponential form,
  \begin{eqnarray}
  \label{eq:natural}  
    p_\theta(x) \propto \exp(s(x) \, \theta_1+x \, \theta_{2}+c(x)) %\quad x\in \mathcal{X}
  \end{eqnarray}  
where $(s(\cdot),x)$ are sufficient statistics of the distribution (and we allow for the case that $s(x)=x$), and the support of $p_\theta(\cdot)$ does not depend on $\theta =(\theta_1, \theta_2)^\top$.
%{\cred (Avoided using $\m X$ in the above display, since we use it for the product space later on -- AB)}
% We define the random vector $X =
% (X^{(1)},\ldots,X^{(p)})^{\top}$ and the random vector with the entry
% $X^{(a)}$ removed as $X_{-(a)}$. 
The exponential family (\ref{eq:natural}) forms the basis of modeling
the conditional distribution of $X^{(a)}$ given $X_{-a}$, where 
$\theta_{2}$ is replaced with a linear sum of the conditioning set $X_{-a}$. In
particular, for each $a\in [p]$, let
 \begin{eqnarray}\label{eq:cond_exp}
    p_a(X^{(a)}|X_{-a}) \propto
   \exp\left (\theta_a s(X^{(a)})+X^{(a)}\sum_{b \ne a} \phi^{(a,b)}X^{(b)} +c(X^{(a)})\right), \quad X^{(a)}\in \mathcal{X}^{(a)},
 \end{eqnarray}
be the conditional density of $X^{(a)} \mid X_{-a}$ with respect to a sigma-finite measure $\mu_j$ on $\m X_j$. The question of {\it compatibility} of such conditional distributions has a long history; see for example, \cite{p:bro-64,p:bes-74,p:hob-98}. In the present context, we paraphrase \cite{p:yan-15}[Theorem 2.1] for conditions implying a compatible joint distribution of $X$ on the product space $\m X = \otimes_{j=1}^p \m X^{(j)}$, starting from the collection of conditionals in \eqref{eq:cond_exp}.
Let $\theta = (\theta_1, \ldots, \theta_p)^\top$, and $\phi$ denote the collection of coefficients $\phi^{(a, b)}$ for $a \ne b \in [p] \times [p]$. 
%A sufficient condition for the above collection of conditional distributions to define a valid joint distribution on the product space $\m X = \otimes_{j=1}^p \m X^{(j)}$ 
If $\phi^{(a, b)} = \phi^{(b, a)}$ for all $a \ne b$, and 
\begin{align}\label{eq:Apt}
A(\phi,\theta) :\,= \int_{\m X}\exp\left(\sum_{a=1}^{p}\big\{\theta_{a} s(x^{(a)})+c(x^{(a)})\big\}
  + \mathop{\sum\sum}_{a \ne b} \frac{1}{2} \phi^{(a,b)}x^{(a)}x^{(b)}\right)\prod_{j=1}^{p} \, \mu_j(d x_j) < \infty, 
\end{align}
then the joint distribution of $X$ is well-defined and is given by the {\it Gibbs distribution}
\begin{align}
\label{eq:jointexponential}  
 p(x^{(1)},\ldots,x^{(p)}) %= A(\phi,\theta)^{-1}
 \propto \exp\left(\sum_{a=1}^{p}
 \big\{\theta_{a} s(x^{(a)})+c(x^{(a)})\big\}
  +  \mathop{\sum\sum}_{a \ne b} \frac{1}{2}\phi^{(a,b)}x^{(a)}x^{(b)}\right).
\end{align}
% where $\phi^{(a,b)} = \phi^{(b,a)}$ and $\mathcal{N}_{a}\subset
% \{1,\ldots,p\}\backslash\{a\}$ is the neighborhood set of $a$. Let
% \begin{eqnarray*}
%   A(\phi,\theta) = \int_{}\exp\left(\sum_{a=1}^{p}\theta_{a} s(x^{(a)})
%   + \sum_{a=1}^{p}\sum_{b\in
%   \mathcal{N}_{a}}\frac{1}{2}\phi^{(a,b)}x^{(a)}x^{(b)}\right)\prod_{j=1}^{p}dx_{j}.
% \end{eqnarray*}  
% If $A(\phi,\theta)<\infty$, then the joint
% distribution of $X=(X^{(1)},\ldots,X^{(p)})$ is given by the Gibbs distribution
% \begin{eqnarray}
% \label{eq:jointexponential}  
%  p(x^{(1)},\ldots,x^{(p)}) = A(\phi,\theta)^{-1}\exp\left(\sum_{a=1}^{p}\theta_{a} s(x^{(a)})
%   + \sum_{a=1}^{p}\sum_{b\in
%   \mathcal{N}_{a}}\frac{1}{2}\phi^{(a,b)}x^{(a)}x^{(b)}\right).
% \end{eqnarray}
What makes this model attractive, besides the ease in which different distributions can ``slot'' together, is that the conditional
independence graph it describes is determined by the non-zero coefficients $\phi^{(a, b)}$. Specifically, for $a \in [p]$, define $\m N_a = \{b \in [p] \,:\, \phi^{(a, b)} \ne 0\}$ to be the neighborhood set of node $a$. Inspecting \eqref{eq:cond_exp}, it is apparent that $X^{(a)}\independent X^{(c)}|X_{V\backslash\{a,c\}}$ iff $c\notin \mathcal{N}_{a}$. By symmetry of the $\phi^{(a,c)}$s, the neighborhood relation is symmetric, i.e., $c \in \m N_a$ iff $a \in \m N_c$. Thus, the undirected edge set $E$ for the conditional independence graph consists of all such  
neighbors $(a, c)$. 
%{\cred (I deferred the definition of the neighborhood set to after equation (4). -- AB)}
% In particular, $X^{(a)}\independent X^{(c)}|X_{V\backslash\{a,c\}}$ iff
% $c\notin \mathcal{N}_{a}$. Thus the neighbourhood sets $\{\mathcal{N}_{a}\}$
% immediately yield the edge set $E$. 
%Further, since
%$p(x^{(1)},\ldots,x^{(p)})$ is strictly positive %on $\m X$,
%the random vector $X$ satisfies the global Markov property and admits
%a Hammersley--Clifford factorisation. {\cred %(Probably need a citation here? -- AB)}{\color{blue} The book by Lauritzen, 1996, Theorem 3.9. But I think he requires that the density is continuous. So we should probably paraphrase the above}

\paragraph*{Graphical models for multivariate time series} While the
above contributions are model based, \cite{p:dah-00} proposes
a process-wide conditional graph for multivariate time series
based on the partial correlation. 
%From now on we refer to this approach as StGGM  %(stationary time series Gaussian graphical models).
Given the stationary multivariate time series $\{X_{t} = (X_{t}^{(1)},\ldots,X_{t}^{(p)}\}$,
the vertex set consists of each component in the series
(rather than individual random variables) and
the edge set is determined by the partial covariance between time series. 
Let $\mathcal{H} = \overline{sp}(X_{t}^{(c)};t\in \mathbb{Z},c\in
\{1,\ldots,p\})$ and $\mathcal{H}_{-(a,b)} =
\overline{sp}(X_{t}^{(c)};t\in \mathbb{Z},c\in
\{1,\ldots,p\}\backslash\{a,b\})$. Then the partial covariance between
time series $\{X_{t}^{(a)}\}$ and $\{X_{t}^{(b)}\}$ is defined as 
$\rho_{h}^{(a,b)|\shortminus\{a,b\}}=\cov[X_{t}^{(a)} - P_{\mathcal{H}_{-(a,b)}}(X_{t}^{(a)}),
X_{t+h}^{(b)} - P_{\mathcal{H}_{-(a,b)}}(X_{t+h}^{(b)})]$. The partial correlation graph is defined under the rule that
 $(a,b) \in E$  iff
  $\{\rho_{h}^{(a,b)|\shortminus\{a,b\}}\neq 0 \textrm{ for some  }h \in \mathbb{Z}\}$.
  \cite{p:dah-00} shows that the condition
  $\{\rho_{h}^{(a,b)|\shortminus\{a,b\}}=0 \textrm{ for all
  }h \in \mathbb{Z}\}$ is equivalent to
  $[f(\omega)^{-1}]_{a,b}=0$ for all $\omega\in [0,2\pi]$, where
  $f(\omega) = \sum_{h\in \mathbb{Z}}C(h)\exp(ih\omega)$ is the spectral density matrix of the time series 
  $\{X_{t}\}_{t\in \mathbb{Z}}$ with 
  $\{C(h)\}_{h\in \mathbb{Z}}$ as the autocovariance function
  $C(h) = \cov[X_{0},X_{h}]$. 
 Whereas most graphical model approaches consider the conditional independence between random variables at the individual random variable level, this 
 approach is unique in that the partial covariance is a process-wide measure of conditional dependence for Gaussian processes.
% (where two stochastic processes are removed in the conditioning set). 
%  \cite{p:bas-sub-23} connect the
%  process-wide partial correlation graph to the individual-level conditional graph. The connection between the individual level 
%  conditional graph 
%  to the process level graph forms the basis of the proposed modelling approach [NEED TO IMPROVE].

  %{\cred (If the previous paragraph looks okay, we %probably need to expand this paragraph a little bit %more; define the quantities such as sp and spectral %density matrix, cite more literature including your %paper with Sumanta etc. -- AB)}

\section{Conditionally specified multivariate time series} \label{sec:cstgm}
% Our aim is to 
%   blend the above approaches. In particular, use a
% model based approach for modelling process-wide conditional dependence for
% multivariate time series that is non-Gaussian. 
Our aim is to blend the above approaches to propose a probabilistic framework towards modeling process-wide conditional dependence for multivariate non-Gaussian time series.
To motivate our
approach we start with the  vector autoregressive model
VAR$(1)$. The inverse spectral density matrix of a 
VAR$(1)$ has a simple form, from which it is straightforward to deduce the 
partial correlation graph defined in \cite{p:dah-00}.
Moreover, we show below that a VAR$(1)$
model with iid Gaussian innovations can also be expressed in
a conditional exponential form. Using this we can explicitly connect the
partial correlation graph of a multivariate time series with the conditional
independence graph from an exponential distribution approach. We show how this
can be used as a stepping stone for defining 
%modelling the 
a process-wide
graphical model for general non-Gaussian time series.

Consider the VAR$(1)$ model 
\begin{eqnarray}
\label{eq:XtVAR1}  
X_t =   \left(
\begin{array}{c}
X_{t}^{(1)} \\
X_{t}^{(2)} \\
X_{t}^{(3)} 
\end{array}
\right) = 
\left(
\begin{array}{ccc}
\alpha_1 & \beta_1 &  0 \\
 0 & \beta_2 & 0  \\
0 & \beta_3 & \gamma_3 \\
\end{array}
\right) \left(
\begin{array}{c}
X_{t-1}^{(1)} \\
X_{t-1}^{(2)} \\
X_{t-1}^{(3)} \\
\end{array}
\right) +\varepsilon_{t} = AX_{t-1}+\varepsilon_{t}
\end{eqnarray}
where $\{\varepsilon_{t}\}_{t \in \mb Z}$ are iid Gaussian random vectors with
$\varepsilon_{t} \sim N(0,I_{3})$. The inverse spectral density is 
%{\cred (Is there a specific reason you write it as a product of two matrices? -- AB)}{\color{blue}No reason besides being awful at multiplying matrices without errors :)}
\begin{align*}
  f(\omega)^{-1} 
  &= \left(
  \begin{array}{ccc}
    |1-\alpha_{1}e^{i\omega}|^{2} &
                 -\beta_{1}e^{-i\omega} (1-\alpha_{1}e^{i\omega}) & 0\\
 -\beta_{1}e^{i\omega} (1-\alpha_{1}e^{-i\omega}) & \beta_{1}^{2}+\beta_{3}^{2}+|1-\beta_{2}e^{-i\omega}|^{2} & -\beta_{3}e^{-i\omega}(1-\gamma_{3}e^{i\omega})  \\
0 & -\beta_{3}e^{i\omega}(1-\gamma_{3}e^{-i\omega}) & |1-\gamma_{3}e^{i\omega}|^{2} \\
\end{array}
  \right).
\end{align*}
Since $[f(\omega)^{-1}]_{(1,3)}=[f(\omega)^{-1}]_{(3,1)} =0$  and all
the remaining entries are non-zero, the partial correlation graph 
has vertex set $V=\{1,2,3\}$ and edge set $E=\{(1,2),(2,3)\}$.
We now examine the same model from the perspective of conditional distributions.
For each $(a,t)$ with $a \in \{1, 2, 3\}$ and $t \in \mb Z$, we define the sigma-algebra
% $\mathcal{H}_{(a,t)} = \sigma(X_{\tau}^{(b)};\tau \in \mathbb{R},b\in
% \{1,\ldots,3\},(b,\tau)\neq (a,t))$. 
$\mathcal{H}_{(a,t)} = \sigma(X_{\tau}^{(b)}; \, (b,\tau)\neq (a,t))$.
Since $\{X_{t}\}$ is a
Gaussian stochastic process, the conditional distributions
$X_{t}^{(1)}|\mathcal{H}_{(1,t)}\sim \mathcal{N}(\mu_{1,t},|\alpha|^{-1})$,
$X_{t}^{(2)}|\mathcal{H}_{(2,t)}\sim \mathcal{N}(\mu_{2,t},|\beta|^{-1})$, and 
$X_{t}^{(3)}|\mathcal{H}_{(3,t)}\sim
\mathcal{N}(\mu_{3,t},|\gamma|^{-1})$ are all univariate Gaussian with parameters
  \begin{align}\label{eq:gaussian3}
  \begin{aligned}
\mu_{1,t} &=
   \frac{\alpha_{1}}{|\alpha|}X_{t-1}^{(1)}+\frac{\alpha_{1}}{|\alpha|}X_{t+1}^{(1)}
  -\frac{\alpha_{1}\beta_{1}}{|\alpha|}X_{t}^{(2)}+
     \frac{\beta_{1}}{|\alpha|}X_{t-1}^{(2)} \\
\mu_{2,t} &=
      \frac{\beta_{2}}{|\beta|}X_{t-1}^{(2)} +\frac{\beta_{2}}{|\beta|}X_{t+1}^{(2)} +
    \frac{\beta_{1}}{|\beta|}X_{t+1}^{(1)} +
     \frac{\beta_{3}}{|\beta|}X_{t+1}^{(3)}-
    \frac{\alpha_{1}\beta_{1}}{|\beta|}X_{t}^{(1)}-
    \frac{\gamma_{3}\beta_{3}}{|\beta|}X_{t}^{(3)}  \\
\mu_{3,t} &=
  \frac{\gamma_{3}}{|\gamma|}X_{t-1}^{(3)}+\frac{\gamma_{3}}{|\gamma|}X_{t+1}^{(3)}-
     \frac{\gamma_3\beta_{3}}{|\gamma|}X_{t}^{(2)}+
               \frac{\beta_{3}}{|\gamma|}X_{t-1}^{(2)}
               \end{aligned}
       \end{align}  
where  $|\alpha| = 1+\alpha_1^{2}$, 
  $|\beta| = 1+\beta_{1}^{2}+\beta_{2}^{2}+\beta_{3}^{2}$, $|\gamma| =
  1+\gamma_{3}^{2}$. Examining 
  %coefficients 
  the expressions
  of $\mu_{1,t}$,
  $\mu_{2,t}$ and $\mu_{3,t}$ show that the time series $\{X_{t}^{(3)}\}$
  does not appear in the conditional specification of $X_{t}^{(1)}$ and
  $\{X_{t}^{(1)}\}$
  does not appear in the conditional specification of $X_{t}^{(3)}$,
  whereas $\{X_{t}^{(2)}\}$ appears in both the conditional
  specification of  $X_{t}^{(1)}$  and $X_{t}^{(3)}$. This suggests
  in general 
  that appearance of a time series $\{X_{t}^{(a)}\}$ (regardless of lag) in the
  conditional distribution of $X_{t}^{(b)}$ implies that process-wide the
  two time series  $\{X_{t}^{(a)}\}$ and  $\{X_{t}^{(b)}\}$ are
  conditionally dependent. We make this notion precise by defining
  a conditionally-specified nearest-neighbor multivariate time
  series which includes the VAR$(1)$ model as a special case. In the Gaussian VAR case, $\{X_t\}$, by virtue of being a multivariate Gaussian process, is a well-defined stochastic process. However, more generally, the {\it existence} of such a stochastic process is a question in itself, which we address below.
  %{\cred (Probably a silly comment stemming from my lack of knowledge about standard practice in the literature, but do you think there is a possibility the reader confuses between the role of $t$ in $X_t^{(a)}$ (a specific time point) versus $\{X_t^{(a)}\}$ (the entire time series)? Should we use, say, $X_m$ instead when we talk about a specific $t$? Also, can we define $X^{(a)} = \{X_t^{(a)}\}_{t \in \mb Z}$?)} {\color{blue}Yes definitely, should we use bold ${\bf X}^{(a)}$? to drive home that point?}

We introduce some general notation first. Let $X_t = (X_{t}^{(1)},\ldots,X_{t}^{(p)})^{\top}$ for $t \in \mb Z$ be a multivariate time series, where $X_{t}^{(a)}\in \mathcal{X}^{(a)} \subseteq \mb R$. Let us also define $\bX^{(a)} = \{X_t^{(a)}\}_{t \in \mb Z}$ to be the univariate time series corresponding to $a \in [p]$. For any $a \in [p]$ and $t \in \mb Z$, define the sigma-algebra 
\begin{equation}\label{eq:H_at}
\mathcal{H}_{(a, t)} = \sigma\left(X_\tau^{(b)}\,:b\, \in [p], \tau \in \mathbb{Z}, (b, \tau) \neq (a, t) \right),
\end{equation}
that is, the sigma-algebra created by $\{X_\tau^{(a)}\}_{\tau \ne t}$ and the collection of time series $\{\bX^{(b)}\}$ for $b \ne a$.  Our goal is to {\it define} a stochastic process $\{X_t\}_{t \in \mb Z}$ such that the conditional distribution of any $X_t^{(a)}$ given $\m H_{(a, t)}$ has a natural exponential family form as in \eqref{eq:natural} - \eqref{eq:cond_exp}, %given by 
 \begin{align}
 \begin{aligned}\label{eq:conditionalspecification}  
  p_{a}(x_{t}^{(a)}|\mathcal{H}_{(a,t)})  &\propto
  \exp\left(\theta^{(a)}s_a(x_{t}^{(a)}) 
             +x_{t}^{(a)}\Theta_{a}(\mathcal{H}_{(a,t)}) + c_a(x_t^{(a)}) \right), \quad
                                                              x_{t}^{(a)}\in
  \mathcal{X}^{(a)},\\[1ex]
  \textrm{ where }  \Theta_{a}(\mathcal{H}_{(a,t)}) &=
  \sum_{b\ne a}\Phi_{0}^{(a,b)}x_{t}^{(b)} + 
             \sum_{b \in [p]} \left[\Phi_{-1}^{(a,b)}x_{t-1}^{(b)}+
                                \Phi_{1}^{(a, b)}x_{t+1}^{(b)}\right],
% \textrm{ where }  \Theta_{a}(\mathcal{H}_{(a,t)}) &=&
%   \sum_{b\in\{a,\mathcal{N}_{a}\}}[\widetilde{\Phi}_{0}^{(a,b)}x_{t}^{(b)}+
%              \Phi_{1}^{(a,b)}x_{t-1}^{(b)}+
%                                 \Phi_{1}^{(b,a)}x_{t+1}^{(b)}], \nonumber
\end{aligned}                                
\end{align}
where $\Phi_\ell = (\Phi_\ell^{(a, b)})$ are $p \times p$ matrices for
$\ell \in \{-1, 0, 1\}$, with $\Phi_0$ having zero on the
diagonals. The $(a, b)$th entry of $\Phi_0$ specifies the coefficient
of
$X_t^{(b)}$ inside $\Theta_a^{(t)} :\,= \Theta_a(\m H_{(a, t)})$ (for $a \ne b)$, while those for $\Phi_{-1}$ and $\Phi_1$ respectively specify the coefficients of $X_{t-1}^{(b)}$ and $X_{t+1}^{(b)}$ inside $\Theta_a^{(t)}$, for any $t$. We have restricted to a $1$-nearest (time) neighbor specification in \eqref{eq:conditionalspecification} above (that is, $X_t^{(a)}$ only possibly interacts with $X_\tau^{(b)}$ for $\tau \in \{t-1, t, t+1\}$), along with the conditionals having a {\it natural} exponential family form, to make the subsequent development transparent. In Section \ref{sec:d_nn}, we show that the above construction generalizes to a $d$-nearest (time) neighbor
specification and {\it general} distributions from the exponential family. 
%{\cred (As before, I have postponed the graph neighborhood discussion here. Since there are two types of neighborhoods at play, one re time and one re graph, spacing them out.)}
\begin{example}[VAR(1) revisited]
Consider the VAR(1) model from \eqref{eq:XtVAR1}. Since the conditional distributions $X_t^{(a)} \mid \m H_{(a, t)}$ are all Gaussian, with the conditional means given in \eqref{eq:gaussian3},  is satisfied with $s_a(x) = x^2/2$ for $a \in \{1, 2, 3\}$ and the matrices
\begin{align*}
\Phi_0 = 
\begin{pmatrix}
0 & -\alpha_1 \beta_1  & 0 \\
-\alpha_1\beta_1 & 0 & -\beta_3 \gamma_3 \\
0 & -\beta_3 \gamma_3 & 0
\end{pmatrix}, \
\Phi_{-1} = 
\begin{pmatrix}
\alpha_1 & \beta_1 & 0 \\
0 & \beta_2 & 0 \\
0 & \beta_3 & \gamma_3
\end{pmatrix}, \
\Phi_1 = 
\begin{pmatrix}
\alpha_1 & 0 & 0 \\
\beta_1 & \beta_2 & \beta_3 \\
0 & 0 & \gamma_3
\end{pmatrix}. 
\end{align*}
The structure of these matrices reveal some interesting features. First, $\Phi_0$ is symmetric, implying the coefficient of $X_t^{(b)}$ in $\Theta_a^{(t)}$ is the same as that of $X_t^{(a)}$ in $\Theta_b^{(t)}$, for any $a \ne b$. In other words, at any time point $t$, pairwise interactions between the nodes are symmetric. Next, $\Phi_{-1} = \Phi_1^\top$, implying the coefficient of $X_{t-1}^{(b)}$ in $\Theta_t^{(a)}$ is the same as the coefficient of $X_{t+1}^{(a)}$ in $\Theta_t^{(b)}$ (which further equals the coefficient of $X_t^{(a)}$ in $\Theta_{t-1}^{(b)}$, by time-homogeneity). Thus, across neighboring time points, the coefficients of interaction between pairs of nodes are also symmetric. We show below that both these symmetries are necessary to ensure a compatible joint distribution exists. Finally, observe that the entries $(1, 3)$ and $(3, 1)$ of all three matrices equal zero. This has implications towards process-wide conditional independencies of $\bX^{(1)}$ and $\bX^{(3)}$ given $\bX^{(2)}$; we discuss this in details in Section \ref{sec:graph}. 
%{\cred plan here is to write out these matrices corresponding to eq. (6), and show the symmetries and sparsity at play. }
\end{example}

Before proceeding to the stochastic process construction, we discuss conditions for the conditional specifications in \eqref{eq:conditionalspecification} to define a joint distribution over a finite time horizon $t = 1, \ldots, n$. To deal with the two boundary cases ($t = 1$ and $t = n$), we adopt common practice to assume a {\it reflective boundary}, i.e., $X_1$ only interacts with $X_2$, and $X_n$ only with $X_{n-1}$. This requires a small modification in the definition of $\Theta_a(\m H_{(a, t)})$ as below. 
\begin{theorem}\label{theorem:distributionedge}
Fix a positive integer $n > 1$, and consider the collection of random variables $\{X_t^{(a)}\}$ for $t \in [n], a \in [p]$. Assume the conditional distributions of $X_t^{(a)} \mid \m H_{(a, t)}$ are given by the first line of \eqref{eq:conditionalspecification}, with
% For any $1 < t < n$, assume the conditional specification in \eqref{eq:conditionalspecification}. For $t \in \{1, n\}$, assume a {\it reflective boundary}, that is, 
\begin{align*}
\Theta_a(\m H_{(a, t)}) = 
\begin{cases}
\sum_{b \ne a} \Phi_0^{(a, b)} X_t^{(b)} + \sum_{b \in [p]} \left[\Phi_{-1}^{(a, b)} X_{t-1}^{(b)} + \Phi_1^{(a, b)} X_{t+1}^{(b)} \right], & 1 < t < n \\
\sum_{b \ne a} \Phi_0^{(a, b)} X_1^{(b)} + \sum_{b \in [p]} \Phi_1^{(a, b)} X_{2}^{(b)}, &  t = 1 \\
\sum_{b \ne a} \Phi_0^{(a, b)} X_n^{(b)} + \sum_{b \in [p]} \Phi_{-1}^{(a, b)} X_{n-1}^{(b)}, &  t = n. \\
\end{cases}
\end{align*}
% Suppose for each $a\in \{1,\ldots,\}$ joint distribution of
%   $p_{a}(x_{t}^{(a)}|\mathcal{H}_{(a,t)})$ is defined by
%   (\ref{eq:conditionalspecification}).
%   For a given $n$, we define the constant of integration as
%   \begin{eqnarray*}
%    c_{n}(x_0,x_{n+1})= \int \exp\left(
%   \sum_{t=1}^{n}[\theta^{\top}{\bf
%       s}(x_{t})+x_{t}^{\top}\Phi_{0}x_{t}]+\sum_{t=1}^{n+1}
%       x_{t-1}^{\top}\Phi_1 x_{t}\right)\prod_{\tau=1}^{n}dx_{\tau}
%   \end{eqnarray*}
Suppose $\Phi_0$ is symmetric, and $\Phi_1 = \Phi_{-1}^\top$. Then,
the joint distribution of $(X_1, \ldots, X_n)$ defined on the $n$-fold
product of $\m X=\times_{a=1}^{p}\mathcal{X}^{(a)}$ is given by 
{\small \begin{align*}
p^{(n)}(x_1, \ldots, x_n) = [c^{(n)}]^{-1} \, \exp\left(
  \sum_{t=1}^{n}\big[\theta^{\top}{\bf
      s}(x_{t}) + {\bf 1}_p^\top {\bf c}(x_t) + \frac{1}{2} \, x_{t}^{\top}\Phi_{0}x_{t} \big] + \frac{1}{2} \sum_{t=2}^n \left[x_t^\top \Phi_{-1} x_{t-1} + x_{t-1}^\top \Phi_1 x_t \right] \right),
\end{align*}}
provided the constant of integration $c^{(n)}$ is finite. In the above display, 
$\theta = (\theta_{1},\ldots,\theta_{p})^\top$, ${\bf 1}_p$ is a $p$-vector of ones, and for $x = (x_1, \ldots, x_p)^\top$, 
${\bf s}(x) :\,= (s_{1}(x_1),\ldots,s_{p}(x_p))^{\top}$ \& ${\bf c}(x) :\,= (c_{1}(x_1),\ldots,c_{p}(x_p))^{\top}$. Also, $x_t = (x_t^{(1)}, \ldots, x_t^{(p)})^\top \in \m X$ for all $t \in [n]$. 
\end{theorem}
% However, the conditional specification does not necessarily guarantee
% that a joint distribution exists. In the following theorem we give
% conditions under which for a given $n$, the joint distribution of
% $X_{m+1},\ldots,X_{m+n}$ conditioned on the edge $(X_{m},X_{m+n+1})$ exists.
% \begin{theorem}\label{theorem:distributionedge}
%   Suppose for each $a\in \{1,\ldots,\}$ joint distribution of
%   $p_{a}(x_{t}^{(a)}|\mathcal{H}_{(a,t)})$ is defined by
%   (\ref{eq:conditionalspecification}).
%   For a given $n$, we define the constant of integration as
%   \begin{eqnarray*}
%    c_{n}(x_0,x_{n+1})= \int \exp\left(
%   \sum_{t=1}^{n}[\theta^{\top}{\bf
%       s}(x_{t})+x_{t}^{\top}\Phi_{0}x_{t}]+\sum_{t=1}^{n+1}
%       x_{t-1}^{\top}\Phi_1 x_{t}\right)\prod_{\tau=1}^{n}dx_{\tau}
%   \end{eqnarray*}
%   where $\theta^{\top} = (\theta_{1},\ldots,\theta_{p})$ and
%   ${\bf s}(x) = (s_{1}(x),\ldots,s_{p}(x))^{\top}$.
% If  $c_{n}(x_0,x_{n+1})<\infty$ then the joint distribution of $X_{m+1},\ldots,X_{m+n}$
%   conditioned on $X_{m},X_{m+n+1}$ is
% \begin{eqnarray}
% p(x_{1},\ldots,x_{n}|x_{0},x_{n+1}) 
%   = c_{n}(x_{0},x_{n+1})^{-1}\exp\left(
%   \sum_{t=1}^{n}[\theta^{\top}{\bf
%       s}(x_{t})+x_{t}^{\top}\Phi_{0}x_{t}]+\sum_{t=1}^{n+1}
%       x_{t-1}^{\top}\Phi_1 x_{t}\right).
% \end{eqnarray}
% \end{theorem}
\begin{proof}
The proof is an application of \cite{p:hob-98}, Theorem 1
(with origins in \cite{p:bro-64} and \cite{p:bes-74}). 
\end{proof}  
While one obtains a valid joint distribution for any $n$ under the conditions of Theorem \ref{theorem:distributionedge}, the joint distributions $p^{(n)}$ thus constructed lack {\it internal consistency}, for example, $p^{(n)}(x_1, \ldots, x_n) \ne \int p^{(n+1)}(x_1, \ldots, x_{n+1}) dx_{n+1}$, thereby rendering them unsuitable for process modeling. Moreover, even for a given $n$, the induced marginal distributions of $X_t$, $t = 1, \ldots, n$, change with $t$, implying a lack of stationarity. We present concrete illustrations of these features in the univariate Gaussian case in Section \ref{subsec:univ_g_reflect} of the Supplement.

\subsection{The construction of Conditional Exponential Stationary
  Graphical Model}

We now return to our main desiderata of constructing a strictly stationary stochastic process $\{X_t\}_{t \in \mb Z}$ with node-wise conditionals as in \eqref{eq:conditionalspecification}, where from now onwards we additionally  impose the baseline constraints that $\Phi_0$ is symmetric, and $\Phi_{-1} = \Phi_1^\top$, motivated by Theorem \ref{theorem:distributionedge}. We recall that a time series
$\{X_{t};t\in \mathbb{Z}\}$ is called strictly stationary if for all
$m,\tau_{1},\ldots,\tau_{n}\in \mathbb{Z}$ the joint distribution of
$(X_{m+\tau_{1}}, X_{m+\tau_{2}},\ldots, X_{m+\tau_{n}})$ does not
depend on $m$. To motivate our construction, we revisit the joint distribution $p^{(n)}$ obtained in Theorem \ref{theorem:distributionedge}. Using $\Phi_{-1} = \Phi_1^\top$, write 
\begin{align*}
p^{(n)}(x_1, \ldots, x_n) \, \propto \, \exp\left(
  \sum_{t=1}^{n}[\theta^{\top}{\bf
      s}(x_{t})+ {\bf 1}_p^\top {\bf c}(x_t) + \frac{1}{2}x_t^{\top}\Phi_{0}x_{t}]+\sum_{t=2}^{n}
    x_{t-1}^{\top}\Phi_{1}x_{t}\right). 
\end{align*}
For $x, y \in \m X$, define  
\begin{align}\label{eq:GH_def}
G(x) =  \exp\left(\theta^{\top}{\bf
      s}(x)+ {\bf 1}_p^\top {\bf c}(x) + \frac{1}{2} \, x^{\top}\Phi_{0}x\right), \quad
H(x,y) = \exp\left(
      x^{\top}\Phi_1 y\right).%, \, R(x, y) =  G(x)^{1/2} H(x, y) G(y)^{1/2}. 
\end{align}
We can then write 
\begin{align}
p^{(n)}(x_1, \ldots, x_n) & \, \propto \, G(x_{1})\left[\prod_{t=2}^{n}G(x_{t})H(x_{t-1},x_{t})\right] \nonumber \\
& = G(x_{1})^{1/2}\left[\prod_{t=2}^{n}G(x_{t-1})^{1/2}H(x_{t-1},x_{t})G(x_{t})^{1/2}\right]G(x_{n})^{1/2} \nonumber \\
& = G(x_{1})^{1/2}\left[\prod_{t=2}^{n} R(x_{t-1}, x_t)\right] G(x_n)^{1/2}, \label{eq:pn_newform} 
\end{align}
where the {\it interaction kernel} $R(\cdot, \cdot)$ is defined as 
\begin{align}\label{eq:int_ker}
R(x, y) = G(x)^{1/2} \, H(x, y) \, G(y)^{1/2}, \quad x, y \in \m X.
\end{align}
The interaction kernel plays a central role in the development below. Observe that $R$ is always {\it strictly positive}, which carries major implications. In the univariate case ($p = 1$), is clear that $R(x, y) = R(y, x)$. However, for $p > 1$, $R$ is generally not symmetric, since $H(y, x) = \exp \big(y^\top \Phi_1 x \big) = \exp \big(x^\top \Phi_1^\top y) \ne H(x, y)$, since $\Phi_1$ is not necessarily symmetric. Moreover, $R$ is not a Markov transition kernel since $R(x,y)$ is not a conditional density (in $y$) given the past state $x$.
%$\int R(x, y) \mu(dy)$ is generally different from one, and thus $R$ is not a Markov transition kernel. 
Accordingly, our subsequent analysis requires tools from the spectral theory of positive (not necessarily Markovian) linear operators.
%{\color{blue}I am a bit confused by this, I thought (?) one could say that $\Phi_{1}$ is not necessarily symmetric, I guess it means that $\Phi_{1}^{\top} = \Phi_{1} = \Phi_{-1}$ so it is the same as what you said.} {\cred Modified.}\textcolor{blue}{By $\Phi$ you mean $\Phi_{-1}$?}

We first recall some essential facts about compact operators on Hilbert spaces and their spectrum; see e.g. \cite{b:con-90}[Chapter 2] for a standard textbook treatment, and also Section \ref{sec:FA_rev} of the Supplement for a self-contained summary. Let $\mu = \otimes_{j=1}^p \mu_j$ denote the product measure on $\m
X$, and let $L^2(\m X, \mu) :\, = \{f : \m X \to \C \,:\, \int |f|^2 d
\mu < \infty\}$ denote the {\it Hilbert space} of square-integrable
functions, with the usual $L^2$ inner product $\langle f, g \rangle =
\int f \overline{g} d\mu$ and $L^2$ norm $\|f\| = (\int |f|^2 d\mu)^{1/2}$.   
Using the kernel $R$, define the {\it linear integral operator} $T : L^2(\m X, \mu) \to L^2(\m X, \mu)$ as 
\begin{align}\label{eq:int_oper_main}
T(f)[y] = \int_{\m X} R(x, y) f (x) \mu(dx), \quad y \in \m X. 
\end{align}
If $\max \left\{\int R(x, y) \mu(dx), \int R(x, y) \mu(dy) \right\}
\le c$ $\mu$-a.e., then $T$ is a bounded linear operator with
$\|T\|_{\rm op} \le c$ where $\|\cdot\|_{\rm op}$
denotes the operator norm
induced by the Hilbert space norm $\|\cdot\|$. The {\it adjoint operator} of $T$, denoted $T^*$, is given by,
\begin{align}\label{eq:int_adj_main}
T^*(f)[y] = \int_{\m X} R(y,x) f(x) \mu(dx), \quad y \in \m X. 
\end{align}
When $p = 1$, one has $T = T^*$ using symmetry of $R$, implying the operator $T$ is {\it self-adjoint}. However, this is generally not the case. 

% We recall some essential facts about the spectrum of compact operators on Hilbert spaces; see e.g. \cite{b:con-90}[Chapter 2]. See also Section \ref{sec:FA_rev} of the Supplement for a self-contained summary. 
A bounded linear map is called {\it compact} if it maps the unit ball to a set with compact closure. A function $f \in L^2(\m X, \mu)$ is an {\it eigenfunction} of the integral operator $T$ if there exists $\lambda \in \C$ such that $T(f) = \lambda f$; any such $\lambda$ is called an {\it eigenvalue} of $T$. If $T$ is compact, then it has countably many eigenvalues, and the {\it spectral radius} $r(T)$ is the supremum of $|\lambda|$ over all eigenvalues $\lambda$. Importantly, $r(T) = r(T^*)$. 

\begin{defin}[Strongly positive operator]\label{def:spo}
A linear integral operator $T$ as in \eqref{eq:int_oper_main} is called strongly positive if $R(x, y) > 0$ a.e. $\mu \otimes \mu$. It is called positive if $R(x, y) \ge 0$ a.e. $\mu \otimes \mu$ and $R \ne 0$. These definitions are consistent with the more general definitions in \cite[Chapter 2]{boelkins1998spectral} in a Banach space setting. 
\end{defin}
A strongly positive (resp. positive) operator is an infinite-dimensional generalization of a matrix with all strictly positive (resp. nonnegative) entries. Let $K = \{f \in L^2(\mathcal{X}, \mu) \,:\, f \ge 0 \text{ a.e. } \mu\}$ denote the {\it cone} of non-negative functions in $L^2(\mathcal{X}, \mu)$. It is straightforward to see that if $f \in K$, then for a positive operator $T$, one has $T(f) \in K$, which implies $T(K) \subseteq K$. Moreover, if $f \ne 0 \in K$, and $T$ is strongly positive, then $T(f) > 0$ a.e. $\mu$. 

Note that the operator corresponding to the interaction kernel $R$ in \eqref{eq:int_ker} is strongly positive. We now state a version of the {\it Krein--Rutman} theorem \citep{krein1962linear} pertaining to the dominant eigenvalue and corresponding eigenfunction of a {\it strongly positive integral operator}, which is a key ingredient of our process construction. 
%{\cred (TODO; will need to pack information compactly.)}
\begin{theorem}[Adapted from Theorem 6.6 (Chapter V) of \cite{schaefer1974banach}]\label{thm:KR_without_int_main}
Let $(\m X, \m A, \mu)$ be a sigma-finite measure space, and $T: L^2(\m X, \mu) \to L^2(\m X, \mu)$ be a strongly positive integral operator as in Definition \ref{def:spo}.
%an integral operator as in \eqref{eq:int_oper_main} with the kernel $R(x, y) > 0$ for all $(x, y) \in \mathcal{X} \times \mathcal{X}$. 
Additionally, suppose $T$ is compact. 
%Suppose (i) $T$ is compact, and (ii) $T$ is irreducible, i.e., $S \in \m A$ and $\mu(S) > 0, \mu(\m X \setminus S) > 0$ implies $\int_{\m X \setminus S} \int_S R(x, y) \mu(dx) \mu(dy) > 0$.
% \begin{itemize}
% \item [(i)] Some power of $T$ is compact. 

% \item[(ii)] $S \in \m A$ and $\mu(S) > 0, \mu(\m X \setminus S) > 0$ implies 
% \begin{align}\label{eq:irred}
%     \int_{\m X \setminus S} \int_S R(x, y) \mu(dx) \mu(dy) > 0. 
% \end{align}
% \end{itemize}
Then, (a) the spectral radius $r(T)$ is positive, (b) the spectral radius $r(T)$ is an eigenvalue of $T$, and has a unique 
eigenfunction $v \in L^2(\m X, \mu)$ with $\|v\| = 1$ satisfying $v > 0 \text{ a.e. } \mu$, and (c) every other eigenvalue $\lambda$ of $T$ satisfies $|\lambda| < r(T)$.  
% \begin{itemize}
% \item [(a)] The spectral radius $r(T)$ is positive. 

% \item [(b)] The spectral radius $r(T)$ is an eigenvalue of $T$, and has a unique normalized eigenfunction $v \in L^2(\m X, \mu)$ with $\|v\| = 1$ satisfying $v > 0 \text{ a.e. } \mu$. 

% \item [(c)] $r(T)$ is a simple eigenvalue, i.e., every other eigenvalue $\lambda$ of $T$ has $|\lambda| < r(T)$.
% \end{itemize}
\end{theorem}
A result in the spirit of Theorem \ref{thm:KR_without_int_main} was first proved for positive matrices by \cite{perron1907grundlagen}, and subsequently for non-negative irreducible matrices by \cite{frobenius1912matrizen}. The first infinite-dimensional version of this result was proved by \cite{jentzsch1912integralgleichungen} for positive integral operators on $L^2[0,1]$. \cite{krein1948linear,krein1962linear} generalized the Perron--Frobenius--Jentzsch theory to abstract Banach spaces; see the introduction section of \cite{boelkins1998spectral} for an excellent review of this topic. For more details on the specific version of the theorem we use here, please refer to Theorem \ref{thm:KR_without_int} and Remark \ref{rem:schaefer_vs} in the Supplement. 

\begin{remark}\label{rem:only_positive}
In Theorem \ref{thm:KR_without_int_main}, $r(T)$ is a {\it simple} eigenvalue, i.e., its eigenspace has dimension one (and therefore is spanned by $v$); see also \cite{victory1982linear}[Theorem 4(i)]. Moreover, $v$ (upto scaling) is the only a.e. positive eigenfunction of $T$, i.e., if $\widetilde{v}$ is an a.e. positive eigenfunction of $T$, then there exists a positive constant $c$ such that $\widetilde{v} = c v$. See \ref{subsec:only_positive} in the Supplement for a proof of this fact. %{\cred (Find reference, or point to a proof. A placeholder for now.)}
\end{remark}
%For matrices with all entries positive, Perron's theorem gives that the spectral radius is a {\it simple eigenvalue}\footnote{That is, the spcetral radius is an eigenvalue, and any other eigenvalue is strictly smaller in magnitude} with 
%{\cred (Will give the version of Krein--Rutman as in the slides next, then proceed to the theorem.)}
Let $T$ satisfy the assumptions of Theorem \ref{thm:KR_without_int_main}, and denote $r = r(T) > 0$. Let $v$ ($\|v\| = 1$) be the unique positive eigenfunction of unit length such that $Tv = r v$, that is, 
\begin{align*}
\int_{\m X} R(x, y) v(x) \mu (dx) = r \, v(y), \quad y \in \m X. 
\end{align*}
From \eqref{eq:int_adj_main}, it is evident that $T^*$ also satisfies the assumptions of Theorem \ref{thm:KR_without_int_main}. Further, recall that $r(T^*) = r(T)$. Let $v^*$ ($\|v^*\| = 1$) denote the corresponding positive eigenfunction. Define $w = v^*/\langle v, v^*\rangle$, so that $w > 0$ a.e., $\langle v, w \rangle = 1$ and $T^* w = r w$. The eigenfunction pair $(v, w)$ plays an important role below. If $T$ is self-adjoint, then $w = v^* = v$, however more generally $w \ne v$. With these ingredients, we state our main result concerning the existence and construction of a multivariate stationary stochastic process with node-wise conditionals as in 
\eqref{eq:conditionalspecification}. We refer to this stochastic process as {\it Conditionally Exponential Stationary Graphical Model} (CEStGM). 
\begin{theorem}[CEStGM construction]\label{theorem:one_stationarity}
Suppose the kernel $R: \m X \times \m X \to (0, \infty)$ and the corresponding integral operator $T: L^2(\m X, \mu) \to L^2(\m X, \mu)$ are defined as in \eqref{eq:int_ker} and \eqref{eq:int_oper_main}, respectively. Suppose $T$ is a compact operator. Then, there exists a unique strictly stationary stochastic process $\{X_t\}_{t \in \mb Z}$ with $X_t = (X_t^{(1)}, \ldots, X_t^{(p)})$ and $X_t^{(a)} \in \m X^{(a)}$, such that the conditional distribution of any $X_t^{(a)}$ given $\m H_{(a, t)}$ is given by \eqref{eq:conditionalspecification}. %We refer to this stochastic process as Conditionally Exponential Stationary Graphical Model (CEStGM). 

Let $r :\, = r(T) = r(T^*)$ be the spectral radius of $T$ and $T^*$. Let $v$ and $w$ be a.e. positive eigenfunctions of $T$ and $T^*$ corresponding to $r$, appropriately scaled so that $\langle v, w\rangle = 1$. Then, for any $n \in \mb N$ and shift $m \in \mb Z$, the joint density function of $(X_{m},\ldots,X_{m+n-1})$ (with respect to the $n$-fold product of $\mu$) is given by 
\begin{eqnarray}\label{eq:joint_cestgm}
%\label{eq:pnpd}  
p_{[1:n]}(x_{1},\ldots,x_{n})
  =  \frac{1}{r^{n-1}} \, 
  v(x_{1}) \left[\prod_{t=2}^{n} R(x_{t-1},x_{t}) \right] w(x_n), \quad \textrm{} x_t \in \m X, \, \textrm{ for } t \in [n]. 
\end{eqnarray}
In particular, for all $t\in \mathbb{Z}$,
the marginal density function of $X_{t}$ is $p_{1}(x) = v(x)w(x)$.
\end{theorem}
Compared to the joint density $p^{(n)}$ in \eqref{eq:pn_newform} arising from the reflective boundary condition in Theorem \ref{theorem:distributionedge}, the only difference in the expression for $p_{[1:n]}$ involves the presence of the eigenfunctions $v$ and $w$ at the two ends in place of $G^{1/2}$.
Since $v$ and $w$ are a.e. positive, $p_{[1:n]}$ is strictly positive for any $n$. For $n = 1$, that $p_1$ is a density follows from $\langle v, w \rangle = 1$; for $n > 1$ the same follows from iteratively integrating and using that $v$ and $w$ are eigenfunctions of $T$ and $T^*$. 
Augmenting the eigenfunctions $v$ and $w$ at the two boundaries is the core idea behind the validity of the CEStGM process.
% Let $p_{[1:n]}$ %{\color{blue}change $\pi$ to $p$?}
% be the joint probability distribution of $(X_m, \ldots, X_{m+n-1})$ specified in \eqref{eq:joint_cestgm}. 
We show in Section \ref{sec:pf_one_stationarity} how to construct all finite-dimensional distributions from these $p_{[1:n]}$s, and verify that these finite-dimensional probability measures satisfy Kolmogorov's consistency conditions \cite[Chapter 4]{billingsley2017probability}, thus leading to a valid stochastic process $\{X_t\}_{t \in \mb Z}$. As an illustration, consider 
\begin{align*}
p_{[1:4]}(x_1, \ldots, x_4) = \frac{1}{r^3} \, 
  v(x_{1}) \left[R(x_1, x_2) \, R(x_2, x_3) \, R(x_3, x_4)\right] w(x_4).
\end{align*}
Since $\int R(x_1, x_2) v(x_1) d\mu(x_1) = r v(x_2)$, we get $\int p_{[1:4]}(x_1, \ldots, x_4) \, d\mu(x_4) = p_{[1:3]}(x_1, \ldots, x_3)$. Similarly, using $\int R(x_3, x_4) w(x_4) d\mu(x_4) = r w(x_3)$, we obtain $\int p_{[1:4]}(x_1, \ldots, x_4) \, d\mu(x_4) = p_{[1:3]}(x_1, \ldots, x_3)$. These key identities form the basis of the proof. The (strict) stationarity of CEStGM follows directly from construction, since the joint distribution of $(X_m, \ldots, X_{m+n-1})$ is invariant to the shift $m$. Finally, the uniqueness follows from Remark \ref{rem:only_positive} which precludes the existence of other positive eigenfunctions of $T$ (and $T^*$). %{\cred (need last line?)}

We make some remarks about various aspects of the above theorem. 

\begin{remark}[ Process-wide Compatibility, Compactness, \& Hilbert--Schmidt condition]\label{rem:compact_vfy}
The compactness condition on $T$ in Theorem \ref{theorem:one_stationarity} can be viewed as a generalization of the integrability condition (i.e. finiteness of $c^{(n)}$) in Theorem \ref{theorem:distributionedge}.
We show in Example \ref{ex:AR1} that for Gaussian autoregressive models the compactness of the operator $T$ is an if and only if condition to ensure stationarity of the process. In general, we conjecture that compactness %of the operator $T$ 
is a necessary condition to ensure stationarity of the underlying distribution. Accordingly, we dub the compactness condition together with the baseline constraints, $\Phi_0$ is symmetric and $\Phi_{-1} = \Phi_1^\top$, as {\it process-wide compatibility} conditions. This extends the more usual notion of compatible conditionals \citep{p:bes-74,p:hob-98,p:arn-02}.

Directly verifying that $T$ is compact is usually quite difficult. However, if $R(\cdot, \cdot) \in L^2(\mathcal{X} \times \mathcal{X}, \mu \otimes \mu)$, so that 
\begin{align}\label{eq:HS_defn}
    \|R(\cdot, \cdot)\|^2  :\,= \int_{\m X \times \m X} R(x, y)^2 \mu(dx)\mu(dy) < \infty,
\end{align}
then the operator $T$ is Hilbert--Schmidt, which implies $T$ is compact \cite[Chapter 2]{b:con-90}. For all our examples below, we verify the Hilbert--Schmidt (HS) condition to ensure compactness. Moreover, the HS condition is crucially used to establish key probabilistic properties of CEStGM including $\beta$-mixing; refer to Section \ref{sec:prob_cestgm}. 

Verifying the HS condition imposes certain restrictions on the parameters of the constituent exponential families; however, these conditions are no more stringent than verifying the joint distribution 
$p^{(n=2)}$ arising from the reflective boundary condition in Theorem \ref{theorem:distributionedge} is integrable. 

% in the finite-dimensional case. Indeed, consider the joint distribution $p^{(n)}$ arising from the reflective boundary condition in Theorem \ref{theorem:distributionedge}. For $n = 2$, $p^{(2)}(x_1, x_2) \propto G(x_1) H(x_1, x_2) G(x_2)$, and if $p^{(2)}$ is integrable, then $R(\cdot, \cdot)$ (with parameters $\Phi_0$ and $\Phi_1/2$) satisfies the HS condition. 

%{\color{red}This is perfect! Is this the same as %$G(x_1) H(x_1, x_2)^{2} G(x_2)$ being integrable?}
\end{remark}

\begin{remark}[Characteristics of interaction kernel and
  eigenvectors]\label{rem:eigencharacteristics}
  The eigenfunctions $v$ and $w$ can be analytically calculated only in special cases; see Example \ref{ex:AR1} below for one such instance. More generally, one requires numerical procedures to approximate them; see Theorem \ref{theorem:simulation} where we describe an MCMC-based strategy to this end. We can, however, provide general qualitative estimates as follows. 
  We have $v(y) = r^{-1} \, \int R(x, y) v(x) dx  = r^{-1}
  \, \langle R_y, v\rangle \le r^{-1} \, \|R_y\|$, where $R_y(\cdot) = R(\cdot, y)$. This provides a decay rate on $v(\cdot)$ provided $\|R_y\|$ is finite, which holds under \eqref{eq:HS_defn}. 
  In addition, if $\sup_y \|R_y\|$ is finite, then $v$ is a
  bounded function (a.e.). The same holds true for $w$.
 % and hence, $w$. {\color{blue}why is $v^{*}$ discussed, can we not go straight to $w$ here?} {\cred (we certainly can. we used $\|v\| = 1$ above, so I mentioned $v^*$ since it is also norm one. But I agree it may distract. Let's just say `holds true for $w$'.)}
  %Further, $\sup_{y}\int R(x,y)^{2}dx<\infty$. {\cred (placement okay? I tried to add some context to the remark at the beginning.)}
\end{remark}

\begin{example}[Univariate case]\label{ex:univar}
Consider the case $p = 1$ so that \eqref{eq:conditionalspecification}
assumes the simpler form $p(x_t \mid \m H_t) \propto \exp \{\theta \,
s(x_t) + \phi x_t (x_{t-1} + x_{t+1}) +c(x_t)\}$, where $\theta, \phi$ are real-valued parameters. From \eqref{eq:GH_def}, we have $G(x) = \exp\{\theta \, s(x)+c(x_t)\}$, and $H(x, y) = \exp(\phi x y)$. %{\cred revisit} 
\\[1ex]
(i) {\it Conditional Gaussian.} We have $s(x) = -x^2/2$ for $x \in \mb
R$. The HS condition is satisfied for  $\phi \in \mb
R$ and $\theta>|2\phi|$ (see Example \ref{ex:AR1} for the special case
$\theta = 1+\phi^{2}$). \\
(ii) {\it Conditional Binary.} We have $s(x) = x$ for $x \in \{0, 1\}$. The HS condition is satisfied for any $\theta, \phi \in \mb R$. \\
(iii) {\it Conditional Exponential.} We have $s(x) = -x$ for $x \in [0, \infty)$. The HS condition is satisfied for any $\theta > 0$ and $\phi \leq 0$. \\
(iv) {\it Conditional Poisson.} We have $s(x) = x$ and $c(x) = -\log
x!$ for $x \in \{0, 1, \ldots\}$. The HS condition is satisfied for
any $\theta \in \mb R$ and $\phi \leq  0$.
\end{example}

\begin{example}[Gaussian AR(1) in light of Theorem \ref{theorem:one_stationarity}]\label{ex:AR1}
Consider the (univariate) Gaussian AR$(1)$ process
  $X_{t} = \phi X_{t-1} + \varepsilon_{t}$,
  where $\varepsilon_{t} \overset{i.i.d.} \sim N(0,1)$.  If $|\phi|<1$, using the
 Cholesky decomposition it is well known
  that the joint distribution of $X_{0},\ldots,X_{n+1}$ is
  \begin{eqnarray}\label{eq:ar1_density}
    p(x_0,\ldots,x_{n+1}) 
    \propto\exp\left[-\frac{1}{2} \left\{x_{0}^{2}+x_{n+1}^{2}+
    \sum_{t=1}^{n}(1+\phi^{2})x_{t}^{2}-2\phi\sum_{t=1}^{n+1}x_{t}x_{t-1}\right\}
    \right].
  \end{eqnarray}
We now explain why the coefficient of $x_0^{2}$ and $x_{n+1}^{2}$
differs from $x_{t}^{2}$ (for $1\leq t\leq n$) in the above display using Theorem
\ref{theorem:one_stationarity}. 
For the Gaussian AR$(1)$ model (where $\phi$ can be arbitrary), the conditional specification is 
\begin{align*}
X_t \mid \m H_t \sim N\left( \frac{\phi}{1+\phi^2} (X_{t-1} + X_{t+1}), \frac{1}{1+\phi^2} \right). 
\end{align*}
The conditional interaction kernel is then given by 
\begin{eqnarray}\label{eq:kernel_AR1}
  R(x,y) = \exp\left(-\frac{1}{4}(1+\phi^{2})x^{2}\right)\exp(\phi xy)
  \exp\left(-\frac{1}{4}(1+\phi^{2})y^{2}\right).
\end{eqnarray}
If $|\phi|\neq 1$, then it can be shown $\int R(x,y)^{2}dx
  dy<\infty$, implying $T$ is a Hilbert--Schmidt operator.
  By explicit calculation (see Section \ref{subsec:eigencalc} in the Supplement) we can show that 
  the positive eigenvector (upto scaling) corresponding to the maximal eigenvalue of $T$ (and $T^*$, by symmetry of $R$) is
  \begin{eqnarray*}
v(x) = v^*(x) &\propto& 
         \left\{
       \begin{array}{cc}  
         \exp(-\frac{1}{2}x^{2}) & |\phi|<1 \\
         \exp(-\frac{\phi^2}{2}x^{2}) & |\phi|>1
        \end{array} 
         \right.
  \end{eqnarray*}
Then, from Theorem \ref{theorem:one_stationarity}, we have
\begin{eqnarray*}   
    && p(x_0,\ldots,x_{n+1}) \propto \\
     &&   \left\{
   \begin{array}{cc}     
   \exp\big(-\frac{1}{2}(x_{0}^{2}+x_{n+1}^{2}+
    \sum_{t=1}^{n}(1+\phi^{2})x_{t}^{2}-2\phi\sum_{t=1}^{n+1}x_{t}x_{t-1})
     \big) & |\phi|<1 \\
     \exp\big(-\frac{1}{2}(\phi^{2}(x_{0}^{2}+ x_{n+1}^{2})+
    \sum_{t=1}^{n}(1+\phi^{2})x_{t}^{2}-2\phi\sum_{t=1}^{n+1}x_{t}x_{t-1})
     \big) & |\phi|>1\\
   \end{array}
    \right.
\end{eqnarray*}
The joint density recovered by Theorem \ref{theorem:one_stationarity} in the case $|\phi| < 1$ coincides with \eqref{eq:ar1_density}. The case $|\phi| > 1$ recovers the stationary non-causal solution to $X_t = \phi X_{t-1} + \varepsilon_t$. In the case of the random walk with $|\phi|=1$  (the nonstationary
  case), it can be shown that $T$ is not a compact operator on $L_{2}(\mathbb{R})$; see Section \ref{subsec:not_compact} in the Supplement for a proof. 
\end{example}

%{\cred (Do we need a multivariate example here?)}

\begin{example}[Multivariate construction]
We now define a trivariate time series
$\{(X_{t}^{(1)},X_{t}^{(2)},X_{t}^{(3)})\}_{t \in \mb Z}$, whose conditional specifications
follow univariate Poisson distributions 
\begin{align*}
p(x_{t}^{(1)}|\mathcal{H}_{t}^{(1)}) &=\exp(x_{t}^{(1)}(\theta^{(1)} +
  \Phi_{-1}^{(1,1)}x_{t-1}^{(1)}+ \Phi_{1}^{(1,1)}x_{t+1}^{(1)}+\Phi_{1}^{(1,2)}x_{t+1}^{(2)})
              -\log x_{t}^{(1)}!) \\
p(x_{t}^{(2)}|\mathcal{H}_{t}^{(2)}) &=\exp(x_{t}^{(2)}(\theta^{(2)} +
                                       \Phi^{(2,2)}_{-1}x_{t-1}^{(2)}+
                                      \Phi^{(2,2)}_{1} x_{t+1}^{(2)}+
                                    \Phi^{(2,1)}_{-1}x_{t-1}^{(1)}+\Phi^{(2,3)}_{1}x_{t+1}^{(3)})
                                       -\log x_{t}^{(2)}!) \\
  p(x_{t}^{(3)}|\mathcal{H}_{t}^{(3)}) &=\exp(x_{t}^{(3)}(\theta^{(3)}
                                         + 
  \phi^{(3,3)}_{-1}x_{t-1}^{(3)}+\Phi_{1}^{(3,3)}x_{t+1}^{(3)}+\Phi^{(3,2)}_{-1}x_{t-1}^{(2)})
                                       -\log x_{t}^{(3)}!) 
\end{align*}
In the above example, the interaction matrices are $3\times
3$-dimensional matrices where $\Phi_0 =0$, $\Phi_{-1}=\Phi_{1}^{\top}$ with 
\begin{align*}
\Phi_{-1} = 
\begin{pmatrix}
  \Phi_{-1}^{(1,1)} & 0 & 0 \\
  \Phi_{-1}^{(2,1)} & \Phi_{-1}^{(2,2)} & 0 \\
  0 & \Phi_{-1}^{(3,2)} & \Phi_{-1}^{(3,3)}  \\
\end{pmatrix} \textrm{ and }
\Phi_{1} = \begin{pmatrix}
  \Phi_{1}^{(1,1)} & \Phi_{1}^{(1,2)} & 0 \\
   0 & \Phi_{1}^{(2,2)} & \Phi_{1}^{(2,3)} \\
  0 & 0 & \Phi_{1}^{(3,3)}  \\
\end{pmatrix}
\end{align*}
This gives the interaction kernel $R(x,y) =
H(x)^{1/2}G(x,y)H(y)^{1/2}$ where 
\begin{align*}
  H(x) &= \exp(\theta^{(1)} x^{(1)} + \theta^{(2)} x^{(2)} +
         \theta^{(3)}x^{(3)} - \log x^{(1)}! - \log x^{(2)}! - \log x^{(3)}!) \\
  G(x,y) &=\exp(x^{(1)}\Phi_{1}^{(1,1)}y^{(1)} +
 x^{(2)}\Phi_{1}^{(2,2)}y^{(2)}+x^{(3)}\Phi_{1}^{(3,3)}y^{(3)}+x^{(1)}\Phi_{1}^{(1,2)}y^{(2)}
           + x^{(2)}\Phi_{1}^{(2,3)}y^{(3)}).
\end{align*}
If all the entries of $\Phi_{1}$ are negative, then
$\|R(\cdot,\cdot)\|_{}<\infty$ and the HS-condition \eqref{eq:HS_defn} is satisfied.
Consequently, $\{X_{t}\}_{t\in \mathbb{Z}}$ is a strictly stationary
time series by Theorem \ref{theorem:one_stationarity}. We note that the requirement that the entries in $\Phi_{1}$
are negative is standard for conditionally specified Poisson
distributions; see, for example, \cite{p:zeg-qaq-88}, Section 2.2(iii)
and \cite{p:yan-15}, Section 2.4.
\end{example}

We now establish that under the conditions of Theorem
\ref{theorem:one_stationarity}, $\{X_t\}_{t \in \mb Z}$ is a
homogeneous Markov process. Recall that $\{X_t\}_{t \in \mb Z}$ is a
Markov process if for each $t \in \mb Z$, $X_t \mid X_{t-1}, X_{t-2},
\ldots \overset{d} = X_t \mid X_{t-1}$; and the Markov process is
homogeneous if the distribution of $X_t \mid X_{t-1} = x$ is the same
for all $t$. 
%{\cred(Need suggestions regarding how to fit in the story
%   of the left transition here. Define $Y_t = X_{-t}$?)} {\color{blue}I,
% personally, would not do this. I think it would distract a bit from the
% core points. We seem to have an overload of results in this paper.}
\begin{corollary}\label{cor:Markovian}
Suppose the conditions in Theorem \ref{theorem:one_stationarity}
hold. Then $\{X_{t}\}_{t \in \mb Z}$ is a homogenuous
  Markov process with right transition kernel
  $p_{t \,|\, t-1}(x_{t}|x_{t-1})
  =r^{-1}R(x_{t-1},x_{t}) \, w(x_{t})/w(x_{t-1})$  and left
  transition kernel $p_{t-1 \,|\, t}(x_{t-1}|x_{t})
  =r^{-1}R(x_{t-1},x_{t}) \, v(x_{t-1})/v(x_{t})$.
\end{corollary}  
\begin{proof}
The proof immediately follows from Theorem \ref{theorem:one_stationarity}.
\end{proof}

% {\color{blue}
% \begin{remark}
% One thing we may want to mention is that when we change the reflective boundary condition to $p_{n|n-1}(x_{n}|x_{n-1}) = r^{-1}R(x_{n-1},x_{n})w(x_{n})/w(x_{n-1})$ (for the boundary at $n$) and 
% $p_{1|2}(x_{1}|x_{2}) = r^{-1}R(x_{1},x_{2})v(x_{1})/v(x_{2})$
% (for the boundary at $n=1$) then we get the stationary distribution which satisfies the interval consistent property.
% \end{remark}
% }

We know from Theorem \ref{theorem:one_stationarity} that the marginal density of any $X_t$ is $p_1(x) = v(x) w(x)$. Suppose, $X_{t-1} \sim p_1$ and $X_t \mid X_{t-1} = x_{t-1} \sim p_{t \,|\,t-1}(\cdot \mid x_{t-1})$. Then, the marginal density of $X_t$ is 
\begin{align*}
& \int_{\m X} p_{t \,|\, t-1}(x_t \mid x_{t-1}) \, p_1(x_{t-1}) d\mu(x_{t-1}) 
= \int_{\m X} r^{-1} R(x_{t-1}, x_t) \, w(x_t)/w(x_{t-1}) \, v(x_{t-1}) w(x_{t-1}) d\mu(x_{t-1}) \\
&= w(x_t) \int_{\m X} r^{-1} R(x_{t-1}, x_t) \, v(x_{t-1}) d\mu(x_{t-1}) = w(x_t) v(x_t) = p_1(x_t),
\end{align*}
implying the marginal distribution is preserved under the right
transition kernel defined above.
A similar story holds for the left transition kernel.

\paragraph{Generalization to any distribution in the natural exponential family} In order to reduce cumbersome notation up until now our focus was on 
a specific subfamily of the exponential family. We conclude this section by defining the one-Markov CEStGM specification for the 
general exponential family. We recall that any distribution in the exponential family has the form
\begin{eqnarray}
\label{eq:generalexponential}
f(x) \propto \exp\left(\sum_{j=1}^{K}\theta_{j}s_{j}(x)+c(x)\right)  = \exp\left({\bf s}(x)^{\top}{\boldsymbol \theta} +c(x)\right)
\end{eqnarray}  
where ${\bf s}(x) = (s_{1}(x), s_{2}(x),\ldots, s_{K}(x))^{\top}$
($\{s_{j}(x)\}_{j=1}^{K}$ are sufficient statistics)
and ${\boldsymbol \theta}= (\theta_{1},\ldots,\theta_K)^{\top}$. Note that (\ref{eq:generalexponential})
includes (\ref{eq:natural}) with 
$s_{1}(x) = s(x)$ and $s_{2}(x)=x$. Thus the conditional specification we give below includes (\ref{eq:conditionalspecification}) as a special case.

Let $X_{t} = (X_{t}^{(1)},\ldots,X_{t}^{(p)})$, where 
$X_{t}^{(a)}\in \mathcal{X}^{(a)}\subseteq \mathbb{R}$. 
We suppose that the conditional distribution of $X_{t}^{(a)}$ given 
$\mathcal{H}_{t}^{(a)}$ has the form 
\begin{eqnarray}
\label{eq:fcondgeneral}
  p(x_{t}^{(a)}|\mathcal{H}_{(a,t)}) &\propto&
       \exp\left({\bf s}^{(a)}(x_{t}^{(a)})^{\top}{\boldsymbol \Theta}_{a}(\mathcal{H}_{a,t})+c^{(a)}(x_{t}^{(a)})\right) \quad x_{t}^{(a)}\in \mathcal{X}^{(a)}
\end{eqnarray}
where ${\bf s}^{(a)}(x^{(a)}) = (s^{(a)}_{1}(x^{(a)}),\ldots,
 s^{(a)}_{K_{a}}(x^{(a)}))^{\top}$ and 
 ${\boldsymbol \Theta}_{a}(\mathcal{H}_{a,t})$ is a $K_a$-dimensional vector
\begin{align}
{\boldsymbol \Theta}_{a}(\mathcal{H}_{a,t}) = 
{\boldsymbol \theta}^{(a)} +\sum_{b\in [p]\backslash\{a\}}\Phi^{(a,b)}_{0}{\bf s}^{(b)}(x_{t}^{(b)})+
  \sum_{b\in [p]}
   \left[\Phi^{(a,b)}_{-1}{\bf s}^{(b)}(x_{t-1}^{(b)}) +
 \Phi^{(a,b)}_{1}{\bf s}^{(b)}(x_{t+1}^{(b)})   \right]
 \label{eq:1lag}
\end{align}
 with ${\boldsymbol \theta}^{(a)} =
(\theta_{1}^{(a)},\ldots,\theta_{K_a}^{(a)})^{\top}$ and
$\Phi_{0}^{(a,b)}$, $\Phi_{-1}^{(a,b)}$ and $\Phi_{1}^{(a,b)}$ are 
$(K_{a}\times K_{b})$-dimensional matrices.

For the purpose of defining the joint distribution we 
define the $p\times p$ block matrices
\begin{eqnarray*}
 \boldsymbol{\Psi}_{\ell} = \left(\Phi^{(a,b)}_{\ell};1\leq a,b \leq p\right) \textrm{ for }\ell\in \{-1,0,1\}.
 \end{eqnarray*}
Noting that in order to ensure that the conditional distributions form a compatable joint distribution we require that 
$\Phi^{(a,a)}_0=0$ for $a\in [p]$,  $(\Phi^{(a,b)}_0)^{\top} = \Phi^{(b,a)}_0$
and  $(\boldsymbol{\Psi}_{1})^{\top} = \boldsymbol{\Psi}_{-1}$ (or equivalently, 
$(\Phi^{(a,b)}_{-1})^{\top} = \Phi^{(b,a)}_{1})$).

To obtain the distribution of the stochastic process we  define the corresponding interaction kernel 
$R(x,y) = G(x)^{1/2}H(x,y)G(y)^{1/2}$ where 
\begin{eqnarray*}
 G(x) 
  = \exp\left({\boldsymbol \theta}^{\top}
      {\bf s}(x)+
      \frac{1}{2} 
      {\bf s}(x_{t})^{\top}{\Psi}_0 {\bf s}(x_{t})
+  {\boldsymbol 1}_{p}^{\top}{\bf c}(x)\right)
\textrm{ and }
H(x,y) 
  =  \exp\left({\bf s}(x)^{\top}\Psi_{1} {\bf s}(y)\right)
\end{eqnarray*}
where ${\bf c}(x) = (c^{(1)}(x^{(1)}),\ldots,c^{(p)}(x^{(p)}))^{\top}$,
${\boldsymbol 1} = (1,\ldots,1)^{\top}$,
${\bf s}(x)^{\top} = \textrm{vec}({\bf
   s}^{(1)}(x^{(1)}),\ldots, {\bf s}^{(p)}(x^{(p)}))$,  ${\boldsymbol
   \theta}^{\top} = \textrm{vec}({\boldsymbol \theta}^{(1)},\ldots, {\boldsymbol
   \theta}^{(p)})$. Using this interaction kernel $R$, Theorem \ref{theorem:one_stationarity} holds.

Below we give some examples. 
Additional examples, including mixed data types, are given in Section \ref{sec:examples} in the Supplement.

\begin{example}\label{exam:beta1}[Univariate conditional beta]
Suppose that $\{X_{t}\}$ is a univariate time series whose conditional
distribution follows a beta-distribution
  \begin{eqnarray*}
   \log p(x_{t}|\mathcal{H}_{t})    &\propto&
  [\alpha + \psi_{1}s_{1}(x_{t-1}) 
  +\psi_{1}s_{1}(x_{t+1})+ 
  \phi_{1}s_{2}(x_{t-1})+
    \phi_{2}s_{2}(x_{t+1})-1]s_1(x_{t})
   \nonumber\\
  &&  + [\beta + \psi_{2}s_{2}(x_{t-1}) 
  +\psi_{2}s_{2}(x_{t+1})+  
  \phi_{2}s_{1}(x_{t-1})+
    \phi_{1}s_{1}(x_{t+1})-1]s_{2}(x_{t}) \quad x\in (0,1)
    \end{eqnarray*}
    where $s_{1}(x) = \log x$ and $s_{2}(x) = \log (1-x)$. In this case
\begin{eqnarray*}
  {\boldsymbol \Psi}_{0} = 0,\
  {\boldsymbol \Psi}_{-1} =\left(
  \begin{array}{cc}
    \psi_{1} & \phi_{1} \\
    \phi_{2} & \psi_{2} \\
  \end{array}  
  \right), \ {\boldsymbol \Psi}_{1} =\left(
  \begin{array}{cc}
    \psi_{1} & \phi_{2} \\
    \phi_{1} & \psi_{2} \\
  \end{array}  
  \right),
\end{eqnarray*}
and ${\boldsymbol \theta} = (\alpha-1,\beta-1)$. Thus
\begin{eqnarray*}
  G(x) = \exp({\boldsymbol \theta}^{\top}{\bf s}(x)) \quad
  H(x,y) = \exp({\bf s}(x)^{\top}\Psi_{1}{\bf s}(y)).
\end{eqnarray*}
The exponent in $H(x,y)$ is negative if
$\psi_{1}\leq 0$, $\psi_{2}\leq 0$,
$\phi_{1}\leq 0$ and $\phi_2\leq 0$ (since $s_{i}(x)$ and $s_{j}(y)$ are
negative). Thus in the case that 
$\alpha>0$, $\beta>0$ (standard conditions for a beta-distribution)
and   $\psi_{1},\psi_{2},\phi_{1},\phi_{2}\leq 0$ the kernel $R(x,y)^{2} \leq 
G(x)G(y)$, and satisfies the HS-condition. 

Since conditionally $X_{t}$ follows a beta-distribution with shape parameters
%\begin{eqnarray*}
%\Ex[X_{t}|\mathcal{H}_{t}] 
%= \frac{\alpha(x_{t-1},x_{t+1})}{\alpha(x_{t-1},x_{t+1})+\beta(x_{t-1},x_{t+1})}
%\end{eqnarray*}
%where
\begin{eqnarray*}
\alpha(x_{t-1},x_{t+1}) &=& \alpha -1 
+\psi_{1}(s_{1}(x_{t-1})+s_{1}(x_{t+1}))+
\phi_{1}(s_{2}(x_{t-1})+
    s_{2}(x_{t+1}))\\
    \beta(x_{t-1},x_{t+1}) &=& \beta -1 
+\psi_{2}(s_{2}(x_{t-1})+s_{2}(x_{t+1}))+
\phi_{2}(s_{1}(x_{t-1})+
    s_{1}(x_{t+1})),
\end{eqnarray*}
then $X_{t}$ is likely to be closer to
one if 
$\alpha(x_{t-1},x_{t+1})\gg \beta(x_{t-1},x_{t+1})$
(conversely it is likely to closer to zero if 
$\alpha(x_{t-1},x_{t+1})\ll \beta(x_{t-1},x_{t+1})$). 
However, if $x_{t-1}$ and $x_{t+1}$ are close to one, then $s_{1}(x_{t-1})$ and $s_{1}(x_{t+1})$ will be close to zero and $s_{2}(x_{t-1})$ and $s_{2}(x_{t+1})$ will be 
extremely negative. Therefore to increase the probability of 
$X_{t}$ being close to one (when $x_{t-1}$ and $x_{t+1}$ are close to one), we require that 
$\phi_1< \psi_1$ and $\phi_2< \psi_2$. For example, if we set 
$\psi_{1}=\psi_{2}=0$, this condition is satisfied and 
 \begin{eqnarray*}
   \log p(x_{t}|\mathcal{H}_{t})\propto
  [\alpha + 
  \phi_{1}s_{2}(x_{t-1})+
    \phi_{2}s_{2}(x_{t+1})-1]s_1(x_{t})
    + [\beta + 
  \phi_{2}s_{1}(x_{t-1})+
    \phi_{1}s_{1}(x_{t+1})-1]s_{2}(x_{t}).
    \end{eqnarray*}
The above specification induces positive dependence between neighbouring observations.    
\end{example}

\begin{example}[Bivariate conditional beta]
  We now generalize the univariate specification in Example
  \ref{exam:beta1} to the multivariate case. We define the conditional
  distributions of a bivariate time series $\{(X_{t}^{(a)},X_{t}^{(b)})\}_{t}$ as
\begin{align*}
   \log p(x_{t}^{(a)}|\mathcal{H}_{(a,t)}) 
   &\propto
  [\alpha^{(a)} - 1 + \phi_{1,2,1}^{(a,a)}s_{2}(x_{t-1}^{(a)})+
    \phi_{1,1,2}^{(a,a)}s_{2}(x_{t+1}^{(a)})+\phi_{0,1,2}^{(a,b)}s_{2}(x_{t}^{(b)})]s_{1}(x_{t}^{(a)})
   \nonumber\\
  &  + [\beta^{(a)} - 1 + \phi_{1,1,2}^{(a,a)}s_{1}(x_{t-1}^{(a)})+
    \phi_{1,2,1}^{(a,a)}s_{1}(x_{t+1}^{(a)})+\phi_{0,2,1}^{(a,b)}s_{1}(x_{t}^{(b)})]s_{2}(x_{t}^{(a)})
              \end{align*}
  \begin{align*}         
      \log p(x_{t}^{(b)}|\mathcal{H}_{(b,t)}) 
   &\propto
  [\alpha^{(b)} - 1 + \phi_{1,2,1}^{(b,b)}s_{2}(x_{t-1}^{(b)})+
    \phi_{1,1,2}^{(b,b)}s_{2}(x_{t+1}^{(b)})+\phi_{0,2,1}^{(a,b)}s_{2}(x_{t}^{(a)})]s_{1}(x_{t}^{(b)})
    \\
  &  + [\beta^{(b)} - 1 + \phi_{1,1,2}^{(b,b)}s_{1}(x_{t-1}^{(b)})+
    \phi_{1,2,1}^{(b,b)}s_{1}(x_{t+1}^{(b)})+\phi_{0,1,2}^{(a,b)}s_{1}(x_{t}^{(a)})]s_{2}(x_{t}^{(b)}).
  \end{align*}
In this example, ${\boldsymbol \theta} =
(\alpha^{(a)}-1,\beta^{(a)}-1, \alpha^{(b)}-1,\beta^{(b)}-1)$,
${\bf s}(x) =
(s_{1}(x^{(a)}),s_{2}(x^{(a)}),s_{1}(x^{(b)}),s_{2}(x^{(b)}))^{\top}$,
the block matrices are defined as 
\begin{eqnarray*}
  {\boldsymbol \Psi}_{0}^{} =
  \left(
  \begin{array}{cc}
    0  & \Phi_{0}^{(b,a)} \\
   \Phi_{0}^{(a,b)} & 0 \\ 
  \end{array}  
  \right),  \ {\boldsymbol \Psi}_{1}^{} =
  \left(
  \begin{array}{cc}
    \Phi_{1}^{(a,a)}  & 0 \\
   0 & \Phi_{1}^{(b,b)} \\ 
  \end{array}  
  \right) \textrm{ and }
  {\boldsymbol \Psi}_{-1}^{} =
  \left(
  \begin{array}{cc}
    (\Phi_{1}^{(a,a)})^{\top}  & 0 \\
   0 & (\Phi_{1}^{(b,b)})^{\top} \\ 
  \end{array}  
  \right) 
\end{eqnarray*}
where
\begin{eqnarray*}
   \Phi_{0}^{(a,b)} =
  \left(
  \begin{array}{cc}
    0  & \phi_{0,1,2}^{(a,b)} \\
   \phi_{0,2,1}^{(a,b)} & 0 \\ 
  \end{array}  
  \right),  \Phi_{1}^{(a,a)} =
  \left(
  \begin{array}{cc}
    0  & \phi_{1,1,2}^{(a,a)} \\
   \phi_{1,2,1}^{(a,a)}   & 0 \\ 
  \end{array}  
  \right) \textrm{ and }
   \Phi_{1}^{(b,b)} =\left(
  \begin{array}{cc}
    0  & \phi_{1,1,2}^{(b,b)} \\
   \phi_{1,2,1}^{(b,b)}   & 0 \\ 
  \end{array}  
  \right).
\end{eqnarray*}  
Thus
\begin{eqnarray*}
  G(x) = \exp\left({\boldsymbol \theta}^{\top}{\bf s}(x)  +\frac{1}{2}{\bf
  s}(x)^{\top}\Psi_{0}{\bf s}(x)\right)\quad
  H(x,y) = \exp\left({\bf s}(x)^{\top}\Psi_{1}{\bf s}(y)\right).
\end{eqnarray*}
Analogous to Example \ref{exam:beta1}, if all the entries in
$\Psi_{0}$ and $\Psi_{1}$ are zero or negative then
$R(x,y)^{2}\leq \exp({\boldsymbol \theta}^{\top}{\bf s}(x) +
{\boldsymbol \theta}^{\top}{\bf s}(y))$ and $R(x,y)$ satisfies the
HS-condition if
$\alpha^{(a)},\alpha^{(b)},\beta^{(a)},\beta^{(b)}>0$. 
\end{example}

\subsection{Markov properties for CEStGM}\label{sec:graph}

In this section, we return to the original motivation behind CEStGM and
show that sparsity in the parameters of CEStGM encode process-wide conditional
independence between the time series. We first define the notion of conditional independence for multivariate stochastic processes. 
This definition is analogous to the notion of
zero partial correlation for a multivariate Gaussian time series given in
\cite{p:dah-00}. Clearly, an approach
based on partial correlations would be unsuitable for measuring
conditional dependence between non-Gaussian time series. Instead, to
define conditional independence between stochastic processes we adapt
the approach described in 
\cite[Section 6.4B]{b:flo-mou-rol-90} and Definition 2.1 and Theorem 2.1 in \cite{p:eic-12}. These authors define the notation of Granger
non-causality between stochastic processes through their corresponding
sigma-algebras.

\begin{defin}\label{defin:CondIndependence}[Process-wide conditional independence]
  Suppose $\{(X_{t}^{C},X_{t}^{D},X_{t}^{E}) : t\in \mathbb{Z}\}$ is a
  multivariate stochastic process defined on
  $(\Omega,\mathcal{F},P)$ (noting that $X_{t}^{C}$, $X_{t}^{D}$ and
  $X_{t}^{E}$ are not necessarily univariate random vectors). We
  define the sub-sigma algebras $\mathcal{F}^{C} =
  \sigma(X_{t}^{C};t\in \mathbb{Z})$, $\mathcal{F}^{D} =
  \sigma(X_{t}^{D};t\in \mathbb{Z})$ and $\mathcal{F}^{E} =
  \sigma(X_{t}^{E};t\in \mathbb{Z})$.
  %We define the sets
  %   \begin{eqnarray*}
  %[\mathcal{F}^{C}]^{+}  = \{ \textrm{all positive }
 % \mathcal{F}^{C}\textrm{-measurable bounded positive functions}\},
%\end{eqnarray*} 
% and similarly for $[\mathcal{F}^{D}]^{+}$.
  Then the processes
  $\{X_{t}^{(C)} : t\in \mathbb{Z}\}$ and  $\{X_{t}^{(D)} : t\in
  \mathbb{Z}\}$ are conditionally independent given
   $\{X_{t}^{(E)} : t\in \mathbb{Z}\}$ if 
$\mathcal{F}^{C}\independent
\mathcal{F}^{D}|\mathcal{F}^{E}$.
\end{defin}  

It is worth bearing in mind that the above Definition  \ref{defin:CondIndependence} includes the classical notion of independence between
stochastic processes. Specifically, the two
stochastic processes
${\bf X}^{(a)} = \{X_{t}^{(a)}:t\in \mathbb{Z}\}$ and 
${\bf X}^{(b)} =\{X_{t}^{(b)}:t\in\mathbb{Z}\}$
are independent if $\sigma({\bf X}^{(a)})\independent \sigma({\bf
  X}^{(b)})$ which is equivalent to the more classical definition that
for all  $k_1,k_2>0$ and $s_{1},\ldots,s_{k_1}\in \mathbb{Z}$, $t_{1},\ldots,t_{k_2}\in \mathbb{Z}$ one has
%\begin{eqnarray*}
$\sigma(X_{s_{1}}^{(a)},\ldots,X_{s_{k_1}}^{(a)})\independent
  \sigma(X_{t_{1}}^{(b)},\ldots,X_{t_{k_2}}^{(b)})$.
%\end{eqnarray*}
Moreover, Definition \ref{defin:CondIndependence} includes, as a special case 
conditional, independence between multivariate Gaussian time series
defined in terms of partial correlation given in
\cite{p:dah-00}. 
% Further, from 
%  Definitions \ref{defin:Markov} and \ref{defin:CondIndependence} we can immediately define pairwise, local and global Markov properties for stochastic processes (analogous to the partial correlation properties given in Section 3, \cite{p:dah-00}). 

%As a comparison, Granger
%non-causality between stochastic processes given in \cite{p:eic-12} is
%based on the past. The process $b$ does not Granger cause process $a$ if 
%$\mathcal{F}^{(a)}(t+1)\independent \mathcal{F}^{(b)}(t)|\mathcal{X}^{V\backslash\{a\}}(t)$ where 
%$\mathcal{F}^{(a)}(t+1)=\sigma(X^{(a)}_{t+1},X^{(a)}_{t},\ldots)$,
%$\mathcal{F}^{(b)}(t)=\sigma(X^{(b)}_{t},X^{(b)}_{t-1},\ldots)$ and
%$\mathcal{F}^{V\backslash\{b\}}(t)=\sigma(X^{(c)}_{t+1},X^{(c)}_{t},
%\ldots;c\in V\backslash\{b\})$. 

We now relate the coefficients in models \eqref{eq:conditionalspecification} and 
\eqref{eq:fcondgeneral} to
Definition \ref{defin:CondIndependence}. We first define the neighbourhood set 
\begin{eqnarray}\label{eq:cestgm_nhbr}
  \mathcal{N}_{a} = \{b\in [p]\backslash\{a\};\textrm{ if either }\Phi^{(a,b)}_{0}\neq 0 \textrm{ or }
\Phi^{(a,b)}_{1}\neq 0 \textrm{ or }\Phi^{(a,b)}_{-1}\neq 0\}
\end{eqnarray}
and $\mathcal{N}_{a}^{\prime} = [p]\backslash \{\mathcal{N}_{a}\cup
\{a\}\}$. 
From the conditional specification in
(\ref{eq:conditionalspecification}) it is easily seen that 
\begin{eqnarray}
\label{eq:XtaCond}  
  \sigma(X_{t}^{(a)})\independent \sigma(X_{\tau}^{(c)};c\in
  \mathcal{N}_{a}^{\prime},\tau\in \mathbb{Z})
  \, \big | \,
  %\mathcal{H}_{(a,t)}.
  \sigma(X_{t-1}^{(a)},X_{t+1}^{(a)},X_{t}^{(b)},X_{t-1}^{(b)},X_{t+1}^{(b)};b\in
  \mathcal{N}_{a}). 
\end{eqnarray}  
This is conditional independence at the individual random variable level.
It is not immediately obvious how this  relates to
our notion of process-wise conditional independence; a clear difficulty is the
appearance of $X_{t-1}^{(a)},X_{t+1}^{(a)}$ in the conditioning
set. 
% {\cred (Eq. (17) and the last two lines are really important to
%   deliver the subtlety; I wonder if we could point to this informally
%   earlier.)} {\color{blue}This is an interesting idea, where do you think?}
However, we show in Theorem \ref{theorem:localMarkov} below
that for CEStGM the neighbourhood set 
at the individual random
variable level is equivalent to the process-wide neighbourhood set. 

From Definition \ref{defin:CondIndependence}, one may also define pairwise, local and global Markov properties for stochastic processes (analogous to the partial correlation properties given in Section 3, \cite{p:dah-00}. The Markov properties (see Section 3.2.1, \cite{b:lau-96} and \cite{p:pea-86}) form a fundamental component of a graphical model. It is well known that the Gibbs distribution in (\ref{eq:jointexponential}) satisfies the pairwise, local and global Markov properties (which we define below). 
Since the construction of CEStGM is motivated by the Gibbs distribution it is natural to ask if CEStGM also satisfies the Markov properties. 
%However, it is not immediately obvious what this means for a %stochastic process.
For a finite dimensional random vector  
$X=(X^{(1)},\ldots,X^{(p)})$ 
whose joint distribution 
has a density wrt to a $\sigma$-finite product measure on $\times_{i=1}^{p}\mathcal{X}^{(i)}$,
the Markov properties are stated and deduced from
their densities (see Proposition 2.21, \cite{b:lau-20}). 
Indeed, in most applications the 
Markov properties are stated in terms of their densities and the joint density is assumed to be strictly positive. 
However, in the context of stochastic processes the existence of a density with respect to a product measure does not, necessarily, hold (see, for example, Kakutani's Dichotomy).
Therefore, in the following definition we state the Markov properties in terms of sigma-algebras. In conjunction with Definition \ref{defin:CondIndependence}, this allows us to define 
process-wide Markov properties. From this we 
show that CEStGM satisfies a 
process-wide version of the classical Markov properties.

% We first recall 
% the classical definition of the pairwise, local and global Markov properties on a
% collection of random variables  (see Section 3.2.1, \cite{b:lau-96}, should cite Pearl here too).

\begin{defin}\label{defin:Markov}[Process-wide Markov properties]
  Let $\bX =
  (\bX^{(1)},\ldots,\bX^{(p)})$ be a multivariate stochastic process, where $\bX^{(j)} = \{X_t^{(j)} : t \in \mb Z\}$. 
  %for some countable index set $\mb T$. 
  For any $A \subseteq V$, let $\bX^A = (\bX^{(j)} : j \in A)$ and $\m F^A = \sigma(X_t^A : t \in \mb T\}$. Consider an undirected graph $\mathcal{G} = (V,E)$
  where $V=[p]$ and $E \subseteq V\times V$ (where $E$ excludes self-loops
  of the form $(a,a)$). Let $\mathcal{N}_{a} =\{b:(a,b)\in E\}$
  and $\mathcal{N}_{a}^{\prime} = V\backslash \{\mathcal{N}_{a} \cup \{a\}\}$.
 $S$ is called a separator between two disjoint subsets $A,B\subset
  V$ relative to the graph $\mathcal{G}$, if all paths from $A$ to $B$ are
blocked by $S$. If this is the case, then we
write $A\perp_{\mathcal{G}} B|S$. 
  \begin{itemize}
  \item $\bX$ satisfies the
    pairwise Markov property relative to the graph $\mathcal{G}$ if for
    all $a\in [p]$ and $b\in\mathcal{N}_{a}^{\prime}$ 
    we have $\mathcal{F}^{(a)}\independent
  \mathcal{F}^{(b)}|\mathcal{F}^{V\backslash\{a,b\}}$. %$\sigma(X^{(a)})\independent \sigma(X^{(b)})|\sigma(X^{(c)};c\in V\backslash\{a,b\})$.
  \item $\bX$ satisfies the
    local Markov property relative to the graph $\mathcal{G}$ if for
    all $a\in [p]$ we have $\mathcal{F}^{(a)}\independent
\mathcal{F}^{\mathcal{N}_{a}^{\prime}}|\mathcal{F}^{\mathcal{N}_{a}}$.
  \item $\bX$ satisfies the  
    global Markov property relative to the graph $\mathcal{G}$
    if for any triple $(A,B,S)$ where $A\perp_{\mathcal{G}} B|S$ then
     $\mathcal{F}_{}^{A}\independent \mathcal{F}^{B}|
  \mathcal{F}_{}^{S}$.
\end{itemize}    
\end{defin}
%The above definition is specifically designed for random vectors.
The above definition is specifically stated for stochastic processes. If we change the index set 
$\mb Z$ to $\mb I$ where $\mb I$ is a finite set, then the above definition coincides with the usual definition, which is typically based on strictly positive densities.
% {\ccb When the index set $\mb T$ is finite, the above definition coincides with the usual definition based on strictly positive densities. Of our interest is when $\mb T = \mb Z$, so that we can define process-wide Markov properties for the multivariate CEStGM process.}
As in the finite-dimensional case, the global Markov property implies the local Markov property which in turn implies the pairwise Markov property. In the finite-dimensional case, if the density of $\bX$ exists and is strictly positive, the three Markov properties are equivalent (see Chapter 2, \cite{b:lau-96}). 
In Theorems \ref{theorem:localMarkov} and \ref{theorem:globalMarkov} we show that the multivariate CEStGM process satisfies all three Markov properties.

\begin{theorem}\label{theorem:localMarkov}[Local Markov property of CEStGM]
 Suppose the interaction kernel $R$ is Hilbert--Schmidt, i.e.
  $\|R(\cdot, \cdot)\| < \infty$ as in \eqref{eq:HS_defn}. Then
for each $a \in [p]$, one has $\mathcal{F}^{(a)}\independent
\mathcal{F}^{\mathcal{N}_{a}^{\prime}}|\mathcal{F}^{\mathcal{N}_{a}}$, where the neighborhood set $\m N_a$ is defined in \eqref{eq:cestgm_nhbr}.

Note that the above immediately implies the pairwise Markov property
(Theorem 2.18 (A.3), \cite{b:lau-20}). 
That is if
$b\notin \mathcal{N}_{a}$ then 
$\mathcal{F}^{(a)}\independent
  \mathcal{F}^{(b)}|\mathcal{F}^{V\backslash\{a,b\}}$.

\end{theorem}
The equivalence between 
equations (\ref{eq:XtaCond}) and Theorem \ref{theorem:localMarkov}
in terms of the same neighbourhood set $\mathcal{N}_{a}$
is analogous to Proposition 2.1 and Theorem 2.2, \cite{p:bas-sub-23} where
the equivalence between different definitions of partial correlations
time series is shown.
This result provided the initial motivation for Theorem \ref{theorem:localMarkov}, but  the
techniques used to prove the two results are entirely different. The proof
of Proposition 2.1 and Theorem 2.2, \cite{p:bas-sub-23} applies the
Schur complement to an infinite dimensional covariance operator. Whereas,
the proof of Theorem \ref{theorem:localMarkov} is based on a
Hammersley--Clifford type factorisation of the interaction kernel $R$
together with some subtle measure theoretic arguments. We
give a sketch of the proof at the end of this section, the full
details are found in Section \ref{sec:proofgraph}.

\begin{defin}\label{defin:CIG}[Conditional Independence Graph]
The pairwise Markov property given in Theorem
\ref{theorem:localMarkov}) defines a conditional independence graph
for the multivariate time series 
$\{X_{t}\}_{t \in \mb Z}$. Let $\mathcal{G} = (V,E)$ where $V = [p]$, and the edge set
for CEStGM is defined by the rule
$(a,b)\notin E$ iff $\mathcal{F}^{(a)}\independent
  \mathcal{F}^{(b)}|\mathcal{F}^{V\backslash\{a,b\}}$. From Theorem \ref{theorem:localMarkov}, $(a, b) \in E$ iff $b \in \m N_a$ (or equivalently, $a \in \m N_b$), with the neighborhood set $\m N_a$ defined in \eqref{eq:cestgm_nhbr}.
\end{defin}
Some examples, together with their conditional independence graphs, are given in Section \ref{sec:examples}.

Using the conditional independence graph $\mathcal{G}$ we prove 
the stronger process-wide Global Markov property for CEStGM. 

\begin{theorem}\label{theorem:globalMarkov}[Global Markov property of CEStGM]
 Suppose the interaction kernel $R$ is Hilbert-Schmidt, i.e.
 $\|R(\cdot, \cdot)\| < \infty$ as in \eqref{eq:HS_defn}.
For all  $A\perp_{\mathcal{G}} B|S$ (where $\mathcal{G}$ is defined in
Definition \ref{defin:CIG}), we have 
$\mathcal{F}_{}^{A}\independent \mathcal{F}^{B}|
  \mathcal{F}_{}^{S}$.
\end{theorem}

% {\cred (Keeping in line with our earlier main theorems, it would be
%   nice to give a peak into the machinery behind the proof.)}
% {\color{blue}Good idea, below is a first attempt.}
We briefly sketch the key ideas in the proof of Theorems
\ref{theorem:localMarkov} and \ref{theorem:globalMarkov}. We first
connect the neighbourhood set $\mathcal{N}_{a}$ to the
joint distribution of $X_{-n},\ldots,X_{n}$ by defining the graph
 $\mathcal{G} = (V,E)$ where $V=[p]$ and
  $(a,b)\in E$ iff $b\in \mathcal{N}_{a}$. From $\mathcal{G}$ we obtain the set of all induced cliques $\mathcal{C}$. In other words if $D\in \mathcal{C}$ then $D\subseteq V$
and the corresponding subgraph in $\mathcal{G}$ is complete (all nodes
in $D$ are connected to each other). In Lemma \ref{lemma:cliques} we
show that the cliques $\mathcal{C}$ lead to a ``Hammersley-Clifford'' type
factorisation of the interaction kernel $R(\cdot,\cdot)$ 
\begin{eqnarray*}
R(x,y) = \prod_{D\in \mathcal{C}}R_{D}(x^{(D)},y^{(D)})
\end{eqnarray*}  
where $x^{(D)} = (x^{(a)};a\in D)$. Consequently, the joint density of
$X_{-n},\ldots,X_{n}$ is 
\begin{eqnarray}
\label{eq:pnpd2}  
p_{[-n:n]}(x_{-n},\ldots,x_{n})
  &=& \frac{1}{r^{2n}} \,   v(x_{-n})
      \left[\prod_{t=-n+1}^{n} \prod_{D\in \mathcal{C}}R_{D}(x_{t-1}^{(D)},x_{t}^{(D)}) \right] w(x_{n}). 
\end{eqnarray}
Now consider the sets $(A,B,S)$ where $A\perp_{\mathcal{G}}
B|S$. Given $A$ and $B$ we define the sets $\alpha$ and $\beta$ as
follows. Let $\alpha$ denote the
connectivity components in $\mathcal{G}_{V\backslash S}$ that contain $A$ (but, of course, not $B$) and
$\beta = V\backslash\{\alpha\cup S\}$. Clearly
$A\subseteq \alpha$, $B\subseteq \beta$ and $\alpha$
and $\beta$ are disjoint sets such that
$V=\alpha\cup\beta\cup S$. For example, if $A = \{a\}$,
$S=\mathcal{N}_{a}$, and $B \subseteq V\backslash\{\mathcal{N}_{a} \cup\{a\}\}$
then $\alpha = \{a\}$ and $\beta = V\backslash\{\mathcal{N}_{a} \cup \{a\}\}$. We show $\m F^\alpha \independent \m F^\beta \mid \m F^S$, from which the conclusion follows, since $\mathcal{F}^{A}\subseteq
\mathcal{F}^{\alpha}$ and $\mathcal{F}^{B}\subseteq
\mathcal{F}^{\beta}$.

Since $S$ separates $\alpha$ and $\beta$,
any clique of $\mathcal{G}$ is
either in $\alpha\cup S$ or $\beta\cup S$ (but not both).
Let $\mathcal{C}_{\alpha}$ denote all cliques contained in
$\alpha\cup S$, in addition let ${\bf X}^{C}_{[-m,m]} =
(X_{-m}^{a},X_{-m+1}^{a},\ldots,X_{m}^{a};a\in C)$,
${\bf X}_{(\infty,m)}^{C} = (X_{k}^{a};|k|\geq m,a\in C)$,
 $\mathcal{I}_1 =
\{-n-L+1,\ldots,-n-1\}\cup\{n+1,\ldots,n+L-1\}$ and  $\mathcal{I}_2 =
\{-n-L+1,\ldots,-n-1,-n,\ldots,n,n+1,\ldots,n+L-1\}$. 

Using this notation and the
factorisation in (\ref{eq:pnpd2}) we show in Lemma
\ref{lemma:separator} that for all $n,L,k_1,k_2>0$ the conditional
density satisfies the relation
\begin{eqnarray*}
 && p({\bf x}_{[-n,n]}^{\alpha}|{\bf
     x}^{S}_{[-n-L-k_1,n+L+k_1]},{\bf
    x}_{[-n-L-k_2,n+L+k_2]}^{\beta},{\bf x}_{(\infty,n+L)}^{\alpha})\nonumber\\
 &=& \frac{ \int 
      \prod_{t=-n-L+1}^{n+L}\prod_{D\in \mathcal{C}_{\alpha}}R_{D}(x_{t-1}^{(D)},x_{t}^{(D)})
      \prod_{i \in \mathcal{I}_1}\mu(dx_i^{\alpha})}{
  \int  
  \prod_{t=-n-L+1}^{n+L}\prod_{D\in \mathcal{C}_{\alpha}}R_{D}(x_{t-1}^{(D)},x_{t}^{(D)}) \prod_{i \in
     \mathcal{I}_2} \mu(dx_i^{\alpha})} \nonumber\\
  &=& p({\bf x}_{[-n,n]}^{\alpha}|{\bf
    x}^{S}_{[-n-L:n+L]}, {\bf x}_{(\infty,n+L)}^{\alpha}),%\label{eq:tvconditional2}
\end{eqnarray*}
where the last line is due to the fact that for all $D\in \mathcal{C}_{\alpha}$ we have
$D\subseteq \alpha\cup S$. Consequently,
\begin{eqnarray*}
  \sigma({\bf X}_{[-n,n]}^{\alpha}) \independent
  \sigma({\bf X}^{\beta}_{[-n-L-k_2,n+L+k_2]})  |\sigma({\bf
  X}^{S}_{[-n-L-k_1,n+L+k_1]}, {\bf X}_{(\infty,n+L)}^{\alpha}).
\end{eqnarray*}
As the above holds for all positive $n,L,k_1$ and $k_2$ the final step
is to let $n,L,k_1,k_2\rightarrow \infty$. This is achieved by Theorem \ref{theorem:cond-independenceSigma}, which may be of independent interest. A key observation is that
since CEStGM is beta-mixing (see Theorem \ref{theorem:mixing})
$\mathcal{F}_{-\infty}^{\alpha}=\cap_{L=1}^{\infty} \sigma({\bf X}_{(\infty,n+L)}^{\alpha})$ limits
to the trivial sigma algebra (a set that only contains sets which only
have probability zero or one). Thus in its limit the conditioning set
$\mathcal{F}_{-\infty}^{\alpha}$ plays no role. Consequently, by applying Theorem
\ref{theorem:cond-independenceSigma}, 
we obtain $\mathcal{F}^{\alpha}\independent
\mathcal{F}^{\beta}|\mathcal{F}^{S}$. 
% Since $\mathcal{F}^{A}\subseteq
% \mathcal{F}^{\alpha}$ and $\mathcal{F}^{B}\subseteq
% \mathcal{F}^{\beta}$, this implies  
% $\mathcal{F}^{A}\independent \mathcal{F}^{B}|\mathcal{F}^{S}$.
%The details can be found in Section \ref{sec:proofgraph}.

\section{Probabilistic properties of CEStGM}\label{sec:prob_cestgm}

In this section, we offer further insights into CEStGM. In particular, we show that the CEStGM process is geometrically $\beta$-mixing, and describe a Gibbs sampling algorithm to approximately simulate sample paths from the CEStGM process. Interestingly, both these endeavors  require control over iterated applications of the operators $T$ and $T^*$, and possess similar proof techniques. To that end, we introduce the following notation. 

For any integer $k \ge 2$, we adopt standard practice to iteratively define operators $T^k: L^2(\m X, \mu) \to L^2(\m X, \mu)$ as $T^k = T \circ T^{k-1}$, where $\circ$ denotes the composition operation. For example, when $k = 2$, it is straightforward to see that 
\begin{align*}
  T^2(f)[y] = \int_{\m X} R(x, y) T(f)[x] d\mu(x) = \int_{\m X} \int_{\m X} R(z, x) R(x, y) f(z) d\mu(z) d\mu(x). 
\end{align*}
Let $\m X_k = \underbrace{\m X \times \ldots \times \m X}_{k \text{ times}}$ denote the $k$-fold product of $\m X$. Iterating, we have for any $k \ge 2$,
\begin{align}\label{eq:iter_op}
\begin{aligned}
T^k(f)[y] &= \int_{\m X_k} \left\{\prod_{i=2}^k R(x_{i-1}, x_i)\right\} \, R(x_k, y) \, f(x_1) \prod_{i=1}^k \mu(dx_i), \\
(T^*)^k(f)[y] & = \int_{\m X_k} R(y, x_1) \left\{\prod_{i=2}^k R(x_{i-1}, x_i)\right\} \, f(x_k) \prod_{i=1}^k \mu(dx_i). 
\end{aligned}
\end{align}
%{\color{blue}Should we replace all the $dx_i$ with $\mu(dx_i)$?}
In Lemma \ref{lemma:pow_it_main}, we establish the following important decomposition,
\begin{align}\label{eq:main_decomp}
T^k(f) = r^k \langle f, w\rangle v + \Delta_k(f), \quad (T^*)^k(f) = r^k \langle f, v\rangle w + \Delta^*_k(f), 
\end{align}
where $\Delta_k, \Delta_k^*$ are operators on $L^2(\m X, \mu)$ whose operator norms we have control over. If $\langle f, v \rangle \ne 0$ and $\langle f, w \rangle \ne 0$, then the first term in each identity can be shown to dominate, which in particular implies $T^k(f)$ (resp. $(T^*)^k$), scaled appropriately, converges to $v$ (resp. $w$) geometrically fast. This decomposition is crucial to both results below.   

\medskip 

\noindent {\it Geometric Mixing.} We first establish the mixing result. %Recall the following definition first. 
% \begin{defin}[Geometric $\alpha$-mixing]\label{def:mixing}
% Let $\{X_t\}_{t \in \mb Z}$ be a stationary time series. For integers $m_1 \le m_2$, let $\mathcal{F}_{m_1}^{m_2}=\sigma(X_{m_1},X_{m_1+1}\ldots,X_{m_2})$. $\{X_t\}_{t \in \mb Z}$ is called geometrically $\alpha$-mixing if there exists $\rho \in (0, 1)$ and $C > 0$ such that for any positive integer $n$,
% \begin{align}
% \sup_{A \in \mathcal{F}_{-\infty}^{0}, \, B\in
%   \mathcal{F}_{n+1}^{\infty}}|P(A\cap B) - P(A)P(B)| \le C \rho^n. 
% \end{align}
% Moreover, if $\{X_{t}\}$ is a Markov
% process, then
% \begin{eqnarray*}
%   \sup_{A \in \mathcal{F}_{-\infty}^{0},B\in
%   \mathcal{F}_{n+1}^{\infty}}|P(A\cap B) - P(A)P(B)| =
%   \sup_{A \in\sigma(X_0),B\in \sigma(X_{n+1})}|P(A\cap B) - P(A)P(B)|,
% \end{eqnarray*}  
% see for example, \cite[Section 3.1]{p:bra-06}. 
% \end{defin}
\begin{defin}[Geometric $\beta$-mixing]\label{def:mixing}
Let $\{X_t\}_{t \in \mb Z}$ be a stationary time series, and let $\m F = \sigma(X_t \,:\, t \in \mb Z)$. For any two sub-sigma fields $\m A$ and $\m B$ of $\m F$, define 
\begin{align}\label{eq:def_beta}
\beta(\m A, \m B) = \frac{1}{2} \sum_{i=1}^I \sum_{j=1}^J |P(A_i \cap B_j) - P(A_i)P(B_j)|
\end{align}
where $\{A_i\}_{i=1}^I$ (resp. $\{B_j\}_{j=1}^J$) is a partition of $\m X$ with each $A_i \in \m A$ (resp. each $B_j \in \m B$), and the supremum in \eqref{eq:def_beta} is over all such finite partition pairs. 

For integers $m_1 \le m_2$, let $\mathcal{F}_{m_1}^{m_2}=\sigma(X_{m_1},X_{m_1+1}\ldots,X_{m_2})$. $\{X_t\}_{t \in \mb Z}$ is called geometrically $\beta$-mixing if there exists $\rho \in (0, 1)$ and $C > 0$ such that for any $n \in \mb N$,
\begin{align}
\beta\left(\mathcal{F}_{-\infty}^{0}, \mathcal{F}_{n+1}^{\infty} \right) \le C \rho^n. 
\end{align}
Moreover, if $\{X_{t}\}$ is a Markov
process, then %\textcolor{red}{(hope below is correct)}
\begin{eqnarray*}
\beta\left(\mathcal{F}_{-\infty}^{0}, \mathcal{F}_{n+1}^{\infty} \right) = \beta\left(\sigma(X_0), \sigma(X_{n+1}) \right),
\end{eqnarray*}  
see for example, \cite[Section 3.1]{p:bra-06}. 
\end{defin}
We now establish $\beta$-mixing of the CEStGM process. $\beta$-mixing
is an important tool in statistical inference involving dependent
data. It is used to show Gaussian approximations of estimators and
covariance bounds. In addition, we crucially use $\beta$-mixing of CEStGM in the
proof of the local and global Markov properties in Theorems
\ref{theorem:localMarkov} and \ref{theorem:globalMarkov}. 
%With the above, we state the following result. 
\begin{theorem}[CEStGM is Geometric $\beta$-mixing]\label{theorem:mixing}
Suppose the interaction kernel $R$ is Hilbert--Schmidt, i.e. $\|R(\cdot, \cdot)\| < \infty$ as in \eqref{eq:HS_defn}.
Then, the CEStGM process constructed in Theorem \ref{theorem:one_stationarity} is geometrically $\beta$-mixing, and hence ergodic. %{\cred (assumption on $R$)}
\end{theorem}  
%{\ccb The ergodicity follows from ...}. 
The proof of the theorem is provided in Section \ref{sec:pf_mixing}; we sketch some of the salient features below. Given Corollary \ref{cor:Markovian}, it amounts to bound $\beta\left(\sigma(X_0), \sigma(X_{n+1}) \right)$. It can be shown, see proof for details, that 
\begin{align}\label{eq:mixAB}
2 \beta\left(\sigma(X_0), \sigma(X_{n+1}) \right) \le \int_{\m X} \int_{\m X} |p_{0,n+1}(x_0,x_{n+1})-p_{1}(x_0)p_1(x_{n+1})| \mu(d x_0) \mu(d x_{n+1}),
\end{align}
% we use the second part of Definition \ref{def:mixing} to verify the mixing condition. In particular, we have that for any $A \in\sigma(X_0), B\in \sigma(X_{n+1})$, 
% \begin{align}\label{eq:mixAB}  
% P(A\cap B) - P(A)P(B) = \int_{A}\int_{B}[p_{0,n+1}(x_0,x_{n+1})-p_{1}(x_0)p_1(x_{n+1})]d\mu(x_0)d\mu(x_{n+1}),   
% \end{align}
where $p_{0,n+1}$ denotes the joint density of $(X_{0},X_{n+1})$. From Theorem \ref{theorem:one_stationarity}
we have
\begin{eqnarray*}
p_{1}(x_0)p_1(x_{n+1})  = v(x_0)w(x_0) \, v(x_{n+1})w(x_{n+1}).
\end{eqnarray*}
From Theorem \ref{theorem:one_stationarity}, 
we can obtain an expression for
$p_{0,n+1}(x_0,x_{n+1})$ by integrating out $x_1, \ldots, x_n$ from the joint $p_{[0:n+1]}(x_0, \ldots, x_{n+1})$. We have 
\begin{eqnarray}
p_{0,n+1}(x_0,x_{n+1}) &=& 
  \int_{\m X_n} p_{[0:n+1]}(x_{0},\ldots,x_{n+1})\prod_{i=1}^{n}\mu(dx_{i}) \nonumber \\
  &=&  \frac{1}{r^{n+1}}
  v(x_{0}) w(x_{n+1})\int_{\m X_n} \prod_{t=1}^{n+1}R(x_{t-1},x_{t})
      \prod_{i=1}^{n} \mu(dx_{i}) \nonumber\\
  &=& \frac{1}{r^{n+1}}
      v(x_{0}) w(x_{n+1})\int_{\m X_n} R(x_0, x_1) \left[\prod_{t=2}^{n}R(x_{t-1},x_{t}) \right] \,
      \underbrace{R(x_{n},x_{n+1})}_{=R_{x_{n+1}}(x_{n})}
      \prod_{i=1}^{n} \mu(dx_{i}) \nonumber \\
                       &=&  \frac{1}{r^{n+1}} v(x_{0}) w(x_{n+1}) \, (T^{*})^{n}(R_{x_{n+1}})[x_{0}]. \label{eq:mixing_joint}
%  &=&  \frac{1}{\lambda^{n+1}}\left(v(x_{0}) w(x_{n+1})\lambda^{n+1} v_{n+1}w(x_0) +
%         (T^{*}Q_{w})^{n}(R_{x_{n+1}})[x_{0}]\right),
\end{eqnarray}
In the last line of the above display, we make use of the iterated representation of $(T^*)^k(f)$ from \eqref{eq:iter_op} with $f = R_{x_{n+1}}$, where recall that for any $y \in \m X$, the function $R_y$ on $\m X$ is given by $R_y(x) = R(x, y)$ for $x \in \m X$. By the Hilbert--Schmidt assumption, $R_y \in L^2(\m X, \mu)$ for all $y$. Now, using \eqref{eq:main_decomp}, we obtain 
\begin{align*}
(T^*)^n(R_{x_{n+1}}) = r^n \langle R_{x_{n+1}}, v\rangle w + \Delta^*_n(R_{x_{n+1}}).
\end{align*}
Since $\langle R_{x_{n+1}}, v\rangle = \int R(x, x_{n+1}) v(x) dx = r \, v(x_{n+1})$, evaluating both sides of the above display at $x_0$, and substituting in \eqref{eq:mixing_joint}, we obtain that 
\begin{align}\label{eq:mixing_rem}
 p_{0,n+1}(x_0,x_{n+1}) = p_1(x_0) p_1(x_{n+1}) + \frac{1}{r^{n+1}} v(x_0) w(x_{n+1}) \, \Delta^*_n(R_{x_{n+1}})[x_0], 
\end{align}
which expresses the joint $p_{0,n+1}(x_0, x_{n+1})$ as the product of marginals $p_1(x_0) p_1(x_{n+1})$ plus a ``small'' remainder term. The rest of the proof is deferred to \ref{sec:pf_mixing}.

\medskip 

\noindent {\it Simulating sample paths from CEStGM.} Next, we address the question of simulating finite-length sample paths from the CEStGM process. Specifically, suppose we wish to sample $(X_0, \ldots, X_{n+1})$ distributed according to the joint distribution $p_{[0,n+1]}$ in \eqref{eq:joint_cestgm}. Given the conditional exponential structure underlying CEStGM, it is natural to employ a Gibbs sampler (or Glauber dynamics; \cite[Chapter 3]{b:levin2017markov}) to iteratively sample each coordinate from its full conditional distribution. For each $1 \le t \le n$, the conditional distribution of $X_t^{(a)} \mid (X_\tau^{(b)} \,: \, (\tau, b) \ne (t, a))$ for $\tau \in \{0, 1, \ldots, n+1\}$ and $b \in [p]$, has an exponential family distribution given by \eqref{eq:conditionalspecification}, and therefore is straightforward to sample from. 
% $X_t \mid X_{-t}$ has an exponential family distribution given by \eqref{eq:conditionalspecification} {\color{blue}univariate? of should we say $X_{t}^{(a)}|\mathcal{H}_{(a,t)}$}, and therefore is straightforward to sample from. 
However, one encounters a difficulty at either boundary $t \in \{0, n+1\}$, since even in the univariate case, $p(x_{0}|x_{1}) \,\propto\, R(x_{0},x_{1})v(x_{0})$ and
$p(x_{n+1}|x_{n})\,\propto\, R(x_{n},x_{n+1})w(x_{n+1})$, with the eigenfunctions $v$ and $w$ typically unknown. Even if they are estimated numerically, e.g. using Nystr{\"o}m type methods, the ensuing sampling task remains challenging. We instead take an augmented Gibbs sampling based approach. We state the following result first. 
\begin{theorem}\label{theorem:simulation}
Suppose the assumptions in Theorem \ref{theorem:one_stationarity} hold, and the interaction kernel $R$ is Hilbert--Schmidt, i.e. $\|R(\cdot, \cdot)\| < \infty$ as in \eqref{eq:HS_defn}.
% Further, suppose the interaction kernel $R$ satisfies $R_{\max} :\,= \sup_y [\max\{\|R_y\|, \|\widetilde{R}_y\|\}] < \infty$ {\color{red}Do we know how this condition fits with the HS-condition? On a compact support I think they are equivalent?}, where recall $R_y(x) = R(x, y)$, and define $\widetilde{R}_y(x) = R(y, x)$. 
For any integer $m > 0$, define joint density $g_{n, m}$ on $\m X^{n+2m}$ with respect to the $(n+2m)$-fold product of $\mu$ as
\begin{eqnarray*}
g_{n,m}(x_{-m+1},\ldots,x_{0},\ldots,x_{n+1},\ldots,x_{n+m}) = C_{n+2m}^{-1} 
f(x_{-m+1}) \left[\prod_{i =-m+2}^{n+m}  R(x_{i-1},x_{i}) \right] f(x_{n+m}),
\end{eqnarray*} 
where $f \in L^2(\m X, \mu)$ is a strictly positive function and $C_{n+2m}$ is the constant of integration. Using $g_{n,m}$, define the marginal density 
  \begin{eqnarray*}
     h_{[0:n+1]}^{(m-1)}(x_{0},\ldots,x_{n+1})
 &=&    \int_{\m X_{|\m I|}} g_{n,m}(x_{-m+1},\ldots,x_{n+m}) \, \prod_{i \in \m I} \mu(dx_i),
  \end{eqnarray*}
where $\m I = \{-m+1, \ldots, -1\} \cup \{n+2, \ldots, n+m\}$. Then for all $n$ and $m$ we have 
  \begin{eqnarray*}
  \|h_{[0:n+1]}^{(m-1)} - p_{[0:n+1]}\|_{\rm TV}  = O(\rho^{m-1}),  
  \end{eqnarray*}
where $0 \le \rho <1$, and {\rm TV} denotes the total variation distance.   
\end{theorem}
We remark that $f \in L^2(\m X, \mu)$ ensures finiteness of $C_{n+2m}$; refer to the proof in Section \ref{subsec:pf_simul}. We discuss some key features of the proof. Using \eqref{eq:iter_op}, one can write 
\begin{eqnarray*}
h_{[0:n+1]}^{(m-1)}(x_{0},\ldots,x_{n+1}) = C_{n+2m}^{-1} \, q({\bf x}_{[0:n+1]}) \, \left\{ T^{m-1}(f)[x_0] \right\} \, \left\{ (T^*)^{(m-1)}(f)[x_{n+1}] \right\},
\end{eqnarray*}
where $q({\bf x}_{[0:n+1]})=R(x_{0},x_{1}) \times \cdots \times R(x_{n},x_{n+1})$. Next, using \eqref{eq:main_decomp}, this simplifies to
\begin{align}\label{eq:hnm}
h_{[0:n+1]}^{(m-1)}(x_{0},\ldots,x_{n+1}) = \frac{1}{r^{n+1}}q({\bf x}_{[0:n+1]})\frac{
         \left\{  \langle f,w\rangle v(x_{0}) + \delta_1(x_0)\right\}
           \left\{  \langle f,v\rangle w(x_{n+1})
         + \delta_2(x_{n+1})\right\}}{\langle f,w\rangle \langle f,v\rangle + \delta_3}
\end{align}
where $\delta_1(x_0) = r^{-(m-1)} \Delta_{m-1}(f)[x_0]$, $\delta_2(x_{n+1}) = r^{-(m-1)} \Delta_{m-1}^*(f)[x_{n+1}]$, and 
\begin{align*}
\delta_3 = r^{-(n+2m-1)}\int_{\m X} \Delta_{n+2m-1}(f)[x_{n+m}] \, f(x_{n+m}) \mu(d x_{n+m}). 
\end{align*}
Observe that if we drop the $\delta_i$s from the right hand side of \eqref{eq:hnm}, it reduces to $\frac{v(x_0) \,q({\bf x}_{[0:n+1]}) \, w(x_{n+1})}{r^{n+1}} = p_{[0:n+1]}(x_0, \ldots, x_{n+1})$. Thus, control over the $\delta_i$s help us bound $\|h_{[0:n+1]}^{(m-1)} - p_{[0:n+1]}\|_{\rm TV}$, and we defer the remaining details to the proof. 

We next discuss how Theorem \ref{theorem:simulation} enables
simulating sample paths. Suppose for a given $f$, the joint density
$g_{n,m}$ can be sampled from. If we simulate $(Y_{-m+1}, \ldots,
Y_{n+m}) \sim g_{n, m}$ and throw away the $(m-1)$ random variables at
each side to retain $(Y_0, \ldots, Y_{n+1})$, then by definition
$(Y_0, \ldots, Y_{n+1}) \sim h_{[0,n+1]}^{(m-1)}$. For $m$ large,
$h_{[0,n+1]}^{(m-1)}$ offers an increasingly close approximation to
$p_{[0:n+1]}$, and thus we  treat $(Y_0, \ldots, Y_{n+1})$ as an
approximate simulation from $p_{[0:n+1]}$. Regarding the choice of
$f$, one obvious candidate is $f = G^{1/2}$, since it then follows
from \eqref{eq:pn_newform} that $g_{n,m}$
equals $p^{(n+2m)}$, the joint density arising from the reflective
boundary
condition in Theorem \ref{theorem:distributionedge}. One may sample
from
$p^{(n+2m)}$ using Gibbs sampling in a straightforward manner, since
all full
conditionals belong to the exponential family. The resulting algorithm
in
the {\it univariate case} takes the following form: \\[1ex]
(i) Initialize $(Y_{-m+1}^{0}, \ldots, Y_{n+m}^{0})$. \\[1ex]
(ii) For $j = 1, \ldots, J$, sample $(Y_{-m+1}^{j}, \ldots, Y_{n+m}^{j})$ by \\
-- drawing $Y_{-m+1}^{j}$ from $Y_{-m+1} \mid Y_{-m+2}^{j-1}$ (using the reflective boundary distribution), \\
-- for $t = -m+2, \ldots, n+m-1$, drawing $Y_t^{j}$ from $Y_t \mid Y_{t-1}^{j}, Y_{t+1}^{j-1}$, \\
-- drawing $Y_{n+m}^{j}$ from $Y_{n+m} \mid Y_{n+m-1}^{j}$ (using the reflective boundary distribution). \\[1ex]
(iii) Discard the first $B$ ($B < J$) iterates as burn-in and retain $\{(Y_0^{j}, \ldots, Y_{n+1}^{j})\}_{j=B+1}^J$ as approximate samples from $p_{[0:n+1]}$.  Using the above algorithm, some simulations from various different models are given in Section \ref{sec:examples}.

% For $m > 0$, consider the random vector $(Y_{-m+1}, \ldots, Y_{-1}, Y_0, \ldots, Y_{n+1}, Y_{n+2}, \ldots, Y_{n+m})$ with joint density  
% \begin{eqnarray*}
% g_{n,2m}(x_{-m+1},x_{-m},\ldots,x_{0},\ldots,x_{n},\ldots,x_{n+m}) \,\propto\, 
% f(x_{-m+1}) \left[\prod_{i =-m+1}^{n+m}  R(x_{i-1},x_{i}) \right] f(x_{n+m}),
% \end{eqnarray*}
% where $f > 0$ is an $L^2$ function such that $g_{n, 2m}$ is a valid density. In particular, if we set $f = G^{1/2}$, then from \eqref{eq:pn_newform}, $g_{n, 2m}$ equals $p^{(n+2m)}$, the joint density arising from the reflective boundary condition in Theorem \ref{theorem:distributionedge}. 

\subsection{Powers of compact positive operators}\label{subsec:pow}
Let $T$ satisfy the assumptions of Theorem \ref{theorem:one_stationarity}, i.e., it is a positive compact integral operator. In addition, if it is self-adjoint (which is the case when $p = 1$), then by the spectral theorem for compact self-adjoint operators \cite[Chapter 2]{b:con-90}; see also Section \ref{sec:FA_rev}; there exists an orthonormal eigenbasis $\{v_j\}_{j=1}^{\infty}$ of $L^2(\m X, \mu)$ with corresponding {\it real} eigenvalues $\{\lambda_j\}_{j=1}^{\infty}$ such that for any $f \in L^2(\mathcal{X})$,  
$T f = \sum_{j=1}^\infty \lambda_j \, \langle f, v_j \rangle \, v_j$. From the Krein--Rutman theorem (Theorem \ref{thm:KR_without_int_main}), we additionally know that the spectral radius $r > 0$ is an eigenvalue of $T$ (with a.e. positive eigenfunction $v$), and all other eigenvalues are strictly smaller than $r$ in magnitude. Therefore, without loss of generality, we can order the eigenvalues in decreasing order of magnitude as $r = \lambda_1 > |\lambda_2| \ge |\lambda_3| \ge \ldots $. For any $f \in L^2(\m X, \mu)$, exploiting {\it orthonormality} of the eigenfunctions, we have 
\begin{align}\label{eq:pow_sa}
T^k(f) = \sum_{j=1}^\infty \lambda_j^k \langle f, v_j \rangle = r^k \langle f, v \rangle v + \sum_{j =2}^{\infty} \lambda_j^k \langle f, v_j \rangle v_j.
\end{align}
If $\langle f, v \rangle \ne 0$, then the first term dominates, which implies that $T^k f$, appropriately scaled, approaches $v$ geometrically fast as $k \to \infty$. To see this, from \eqref{eq:pow_sa}, 
\begin{align}
\begin{aligned}\label{eq:pow_it_sa}
& \left \|\frac{T^k(f)}{r^k \langle f, v \rangle} - v \right \|
= |\langle f, v \rangle|^{-1} \left(\sum_{j=2}^{\infty} (|\lambda_j|/r)^{2k} |\langle f, v_j\rangle|^2\right)^{1/2} \\
& \le |\langle f, v \rangle|^{-1} (|\lambda_2|/r)^{k} \left(\sum_{j=2}^{\infty} |\langle f, v_j \rangle|^2\right)^{1/2} \le |\langle f, v \rangle|^{-1}  \, \|f\| \, (|\lambda_2|/r)^{k}. 
\end{aligned}
\end{align}
%implying $\|T^k f/r^k \langle f, v \rangle - v\| \le |\langle f, v \rangle|^{-1}  \, \|f\| \, (|\lambda_2|/r)^{k}$. 
\eqref{eq:pow_it_sa} can be viewed as an Hilbert space version of the popular power-iteration \citep{golub2013matrix} for symmetric matrices to iteratively obtain the eigenvector corresponding to the dominant eigenvalue. 

Recall that our operator $T$ is not self-adjoint in the $p > 1$ case. For general (non-diagonalizable) matrices, {\color{blue}an} analogous result as above is typically established using Jordan decomposition, and we are unaware of a similar argument for compact none self-adjoint operators.  A version of this result for normal operators (with self-adjoint operators as a special case) on separable Hilbert spaces was established in \cite{erickson1995power}. %; see also \cite[Theorem 2.2.1]{eastman2005analysis}. \cite[Section 2.2.2]{eastman2005analysis} establishes a version of this result for compact non-normal operators on a Hilbert space under a strong assumption which is difficult to verify. Neither of these results apply to the present setting. 
%{\cred (Discuss Bolekins?)}
We state the following Lemma which generalizes \eqref{eq:pow_sa}-\eqref{eq:pow_it_sa} to non self-adjoint operators and subsequently plays an important role in various places. Recall that an operator $P$ on a vector space is called a {\it projection} if it is linear and satisfies $P^2 = P$, that is, {\it it is idempotent}. 
\begin{lemma}\label{lemma:pow_it_main}
Let $T$ be a compact positive integral operator as in Theorem \ref{theorem:one_stationarity}, with $v$ and $w$ the a.e. positive eigenvectors of $T$ and $T^*$ corresponding to $r = r(T)$ with $\langle v, w \rangle = 1$. Define projection operators $P_{v, w}$ and $Q_{v,w}$ on $L^2(\m X, \mu)$ as 
\begin{align}\label{eq:proj_ops}
P_{v, w}(f)[x] = \langle f, w \rangle v(x), \, Q_{v, w}(f)[x] = f(x) - \langle f, w \rangle v(x), \quad f \in L^2(\m X, \mu),
\end{align}
so that $f = P_{v, w}(f) + Q_{v, w}(f)$. Then, for any $f \in L^2(\m X, \mu)$, 
\begin{align}\label{eq:pow_non_sa}
T^k(f) = r^k \langle f, w\rangle v + (T Q_{v,w})^k(f), \quad k \ge 1.
\end{align}
Further, given any $\varepsilon > 0$, there exists a constant $C_\varepsilon > 0$ such that 
\begin{align}\label{eq:op_norm_bd}
\|(T Q_{v,w})^k\|_{\rm op} < C_\varepsilon (|\lambda_2| + \varepsilon)^k, \quad k \ge 1, 
\end{align} 
where $|\lambda_2| :\,= \sup \{|\lambda| \,:\, \lambda \in \sigma(T), \lambda \ne r\}$, and $\|\cdot\|_{\rm op}$ denotes the operator norm. This in particular implies that for any $f \in L^2(\m X, \mu)$ with $\langle f, w\rangle \ne 0$, 
\begin{align}\label{eq:non_sa_geom_cgence}
\left \|\frac{T^k(f)}{r^k \langle f, w \rangle} - v \right \| \le C_\varepsilon \, |\langle f, w \rangle|^{-1}  \, \|f\| \, \left(\frac{|\lambda_2| + \varepsilon}{r}\right)^k, \quad k \ge 1.  \end{align} 
%{\cred (to complete)}
\end{lemma}

\begin{remark}\label{rem:power_it}
The construction of the projection operators $P_{v,w}$ and $Q_{v,w}$ is adapted from \cite[Chapter 5.2]{boelkins1998spectral}. 
$P_{v,w}$ and $Q_{v,w}$ are projections 
because $\langle v,w \rangle=1$.
Unless $v = w$, $P_{v,w}$ and $Q_{v,w}$ are not orthogonal projections\footnote{non-orthogonal projections are sometimes called oblique projections in the literature}. 
A key property 
of $P_{v,w}$ and $Q_{w,v}$ is that they both commute with $T$, i.e., 
$TP_{v,w} = P_{v,w}T$ and $TQ_{v,w} = Q_{v,w}T$
(this can be deduced from
$\langle Tf,w\rangle = \langle f,T^{*}w\rangle $). 
Since $P_{v,w}$ and $Q_{v,w}$ are projections and commute with $T$ identity (\ref{eq:pow_non_sa}) can be inferred; see Section \ref{subsec:pf_pow_it_main} for details. 

Equation \ref{eq:pow_non_sa} extracts the contribution in the direction of $v$ from $T^k(f)$ analogous to \eqref{eq:pow_sa}, which is the dominating contribution due to \eqref{eq:op_norm_bd}. From Theorem \ref{thm:KR_without_int_main}, we know $|\lambda_2| < r$. Fix $\varepsilon > 0$ such that $|\lambda_2| + \varepsilon < r$. Then, the right hand side of \eqref{eq:non_sa_geom_cgence} converges to zero geometrically fast, recovering the same behavior (with a slightly worse contraction constant) as in \eqref{eq:pow_it_sa}.

Since $(T^*)^* = T$, one can define projection operators $P_{w,v}$ and $Q_{w, v}$ to obtain similar conclusions for $(T^*)^k(f)$. In particular, \eqref{eq:pow_non_sa} and \eqref{eq:op_norm_bd} hold with $T$ replaced by $T^*$, $v$ by $w$ (and $w$ by $v$), $Q_{v,w}$ with $Q_{w,v}$, and  $|\lambda_2|$ replaced with $|\lambda_2^*| :\, = \sup \{|\lambda| \,:\, \lambda \in \sigma(T^*), \lambda \ne r\}< r$. The modified version of \eqref{eq:non_sa_geom_cgence} implies $(T^*)^k(f)/r^k \langle f, v \rangle$ converges geometrically to $w$. Thus, in \eqref{eq:main_decomp}, $\Delta_k = (T Q_{v,w})^k$ and $\Delta_k^* = (T^* Q_{w,v})^k$. 
\end{remark}

 \section{The general $d$-Markov CEStGM specification}\label{sec:d_nn}

So far the CEStGM model is a $1$-Markov model. 
Our aim in this section is to generalize 
the conditional specification given in 
(\ref{eq:conditionalspecification}) and 
(\ref{eq:fcondgeneral}) to a $d$-Markov model to allow for 
$d$ additional lags. 

We define a stochastic 
 process $\{X_{t}\}_{t\in \mathbb{Z}}$ such that the conditional distribution of $X_{t}^{(a)}$ given $\mathcal{H}_{(a,t)}$ is 
\begin{eqnarray}
\label{eq:fcond}
  p(x_{t}^{(a)}|\mathcal{H}_{(a,t)}) &\propto&
       \exp\left({\bf s}^{(a)}(x_{t}^{(a)})^{\top}{\boldsymbol \Theta}_{a}(\mathcal{H}_{a,t})+c^{(a)}(x_{t}^{(a)})\right) \quad x_{t}^{(a)}\in \mathcal{X}^{(a)}
\end{eqnarray}
where ${\boldsymbol \Theta}_{a}(\mathcal{H}_{a,t})$ is the $K_{a}$-dimensional vector
\begin{equation}
\label{eq:dlag}
{\boldsymbol \Theta}_{a}(\mathcal{H}_{a,t}) = 
{\boldsymbol \theta}^{(a)} +\sum_{b\in [p]\backslash\{a\}}\Phi^{(a,b)}_{0}{\bf s}^{(b)}(x_{t}^{(b)})+
  \sum_{\ell=1}^{d} \sum_{b\in [p]}
   \left[\Phi^{(a,b)}_{-\ell}{\bf s}^{(b)}(x_{t-\ell}^{(b)}) +
                                                \Phi^{(a,b)}_{\ell}{\bf s}^{(b)}(x_{t+\ell}^{(b)})   \right] %\nonumber\\
 % &&
\end{equation}
 with ${\boldsymbol \theta}^{(a)} =
(\theta_{1}^{(a)},\ldots,\theta_{K}^{(a)})^{\top}$,
  ${\bf s}^{(a)}(x^{(a)}) = (s^{(a)}_{1}(x^{(a)}),\ldots,
 s^{(a)}_{K}(x^{(a)}))^{\top}$ and $\Psi_{\ell}^{(a,b)}$ are 
 $(K_{a}\times K_{b})$-dimensional matrices. 
 The difference between (\ref{eq:dlag}) and (\ref{eq:1lag}) is the inclusion of additional lags.

For the purpose of defining the joint distribution we 
define the $p\times p$ block matrix
\begin{eqnarray*}
 \boldsymbol{\Psi}_{\ell} = \left(\Phi^{(a,b)}_{\ell};1\leq a,b \leq p\right) \textrm{ for }\ell\in \{-1,0,1\}.
 \end{eqnarray*}
 In order to ensure that the conditional distributions form a compatable joint distribution we require that 
$\Phi^{(a,a)}_0=0$ for $a\in [p]$,  $(\Phi^{(a,b)}_{0})^{\top} = \Phi^{(b,a)}_0$
and  $(\boldsymbol{\Psi}_{\ell})^{\top} = \boldsymbol{\Psi}_{-\ell}$ (or equivalently, 
$(\Phi^{(a,b)}_{-\ell})^{\top} = \Phi^{(b,a)}_{\ell})$.

To obtain the distribution of the stochastic process we need to define the corresponding interaction kernel $R$, which requires us to define an appropriate $G$ and $H$ function. In the case that $d=1$, $G$ is simply a function of $x_t = (x_{t}^{(1)},\ldots,x_{t}^{(p)})$ and $H$ is a function of $x_{t}$ and 
$x_{t+1}$. On the other hand, when $d>1$, $G$ is a function of $d$ lags $x_{1},\ldots,x_{d}$
and $H$ a function of $2d$-lags 
$x_{d+1},\ldots,x_{2d}$. This leads to a distribution that (a) has the conditional specification in (\ref{eq:fcond}) and  (b) a stochastic process that is stationary. 
With this in mind we define $G$ as
\begin{eqnarray*}
 G(x_{1},\ldots,x_{d}) 
  = \exp\left(\sum_{t=1}^{d}{\boldsymbol \theta}^{\top}
      {\bf s}(x_{t})+
      \frac{1}{2} \sum_{t=1}^{d}
      {\bf s}(x_{t})^{\top}{\Psi}_0 {\bf s}(x_{t})
+    \sum_{t=1}^{d}\sum_{\ell=1}^{d-t}
  {\bf s}(x_{t})^{\top}{\Psi}_{\ell} {\bf s}(x_{t+\ell})+
  \sum_{t=1}^{d}{\boldsymbol 1}_{p}^{\top}{\bf c}(x_{t})\right)   
\end{eqnarray*}
where ${\bf c}(x_{t}) = (c^{(1)}(x_{t}^{(1)}),\ldots,c^{(p)}(x_{t}^{(p)}))^{\top}$,
${\boldsymbol 1} = (1,\ldots,1)^{\top}$,
${\bf s}(x)^{\top} = \textrm{vec}({\bf
   s}^{(1)}(x^{(1)}),\ldots, {\bf s}^{(p)}(x^{(p)}))$,  ${\boldsymbol
   \theta}^{\top} = \textrm{vec}({\boldsymbol \theta}^{(1)},\ldots, {\boldsymbol
   \theta}^{(p)})$. Whereas
$G$ contains all the interactions between and within the entries of $x_{1}$ to $x_{d}$. 
$H$ is defined similarly, but only contains interactions
\emph{between} $x_{1},\ldots,x_{d}$ and
$x_{d+1},\ldots,x_{2d}$
\begin{eqnarray}      
&& H(x_{1},\ldots,x_{d};x_{d+1},\ldots, x_{2d}) 
  =  \exp\left(\sum_{t=1}^{d}\sum_{\ell=d-t+1}^{d}
      {\bf s}(x_{t})^{\top}\Psi_{\ell} {\bf s}(x_{t+\ell})\right).    \label{eq:Hdefind}
\end{eqnarray}
In the following graphic we illustrate the interactions in a specific interaction kernel for the
case $d=2$ and $p=2$. The lines correspond to interactions between the nodes. The colors denote to which function 
($G$ or $H$) the interactions belong to. 
\begin{center}
  \includegraphics[scale =0.13]{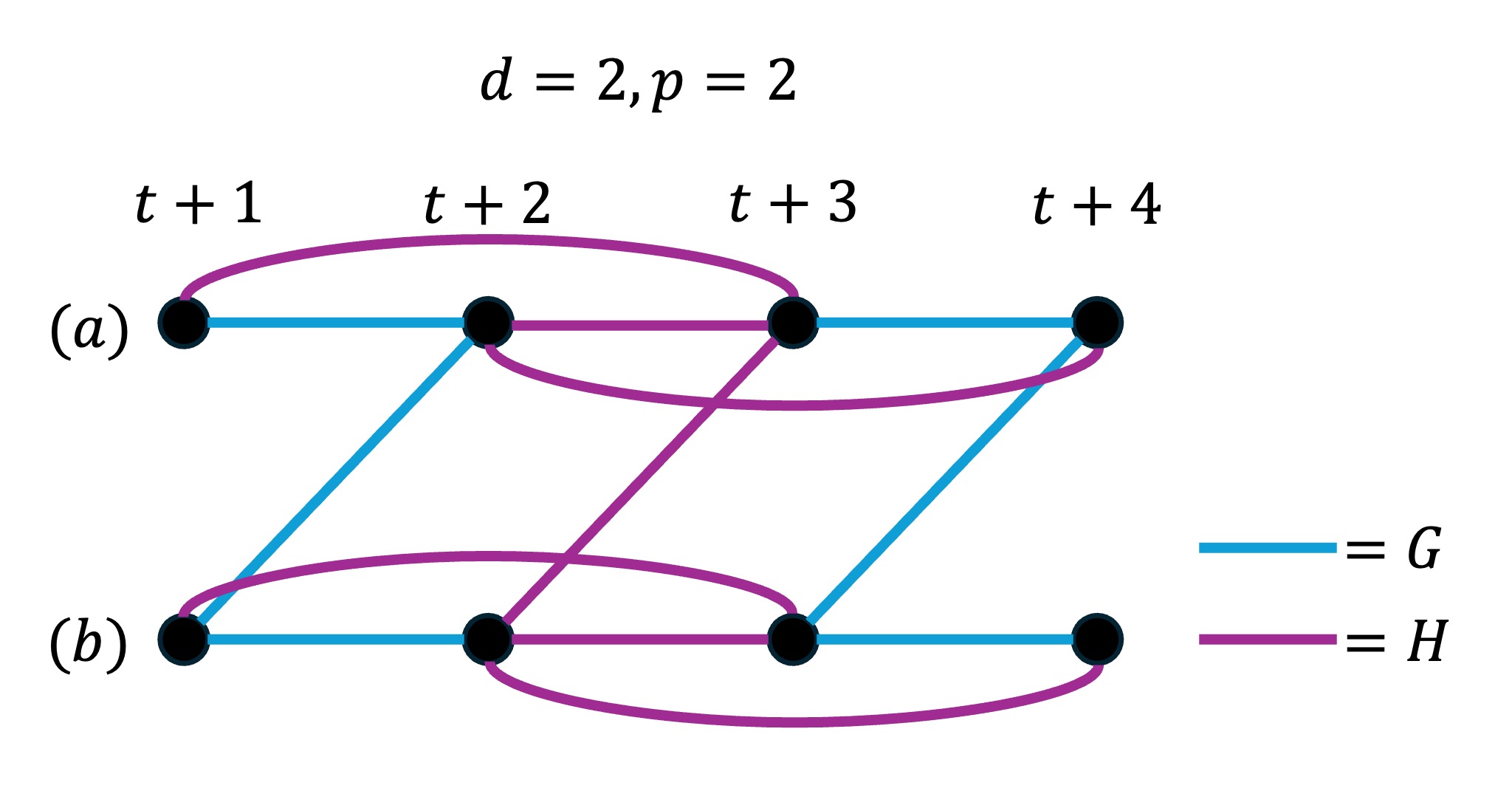}
 \end{center} 
Recall $\m X_d = \underbrace{\m X \times \ldots \times \m X}_{d \text{ times}}$, and let $\mu_d = \underbrace{\m \mu \otimes \ldots \otimes \mu}_{d \text{ times}}$ denote the $d$-fold product measure on $\m X_d$. Using $G$ and $H$, we follow our earlier prescription in \eqref{eq:int_ker} to define the interaction kernel $R$, now defined on $\m X_d \times \m X_d$, as 
$R(\cdot, \cdot) = 
G(\cdot)^{1/2}H(\cdot; \cdot)G(\cdot)^{1/2}$. 
%{\cred (To be consistent with the proofs, and also our earlier notation, should we write $R(\cdot, \cdot)$ instead of $R(\cdot; \cdot)$? Tentatively made the change. Can keep $H(\cdot; \cdot)$.)}

As in \eqref{eq:int_oper_main}, define the {\it linear integral operator}
$T : L^2(\m X_d, \mu_d) \to L^2(\m X_d, \mu_d)$ corresponding to $R$ as, 
\begin{align}\label{eq:int_oper_maind}
T(f)[y] = \int_{\m X_d} R(x, y) f (x) \, \mu_d(dx), \quad y \in \m X_d. 
\end{align}
 From Theorem \ref{thm:KR_without_int_main} and Remark \ref{rem:only_positive}, let $v$ and $w$ denote the unique a.e. positive eigenfunctions of $T$ and $T^*$ respectively, scaled so that $\langle v, w\rangle = 1$.
% {\cred (used the notation $\m X_d$ to avoid $\m X_i$, which may confuse with the space for the $i$th coordinate of $x$. look okay? should we also use $\mu_d$? if look good, will change in theorem below. )}
% {\color{green}It is a good idea, my only concern is that $X^{(a)}$ is used as the $a$th node? eg on page 7 we use that $X_{t}^{(a)}\in \mathcal{X}^{(a)}$} {\ccb (Oops, my bad. How about $\m X_d$ or $\m X_{\rm pr}$ to denote product space?)}

The important feature in the construction of 
$R(\cdot;\cdot)$ is that the connectivity at each time point $t$ can be rearranged, such that there exists functions $E_{j}$ and $\widetilde{E}_{d-j}$ such that for $j\in \{1,\ldots,d-1\}$, we have, 
\begin{align}\label{eq:R_prop}
& \prod_{t=1}^{n-1}R({\bf x}_{[(t-1)d+1:td]},
{\bf x}_{[td+1:(t+1)d]}) = E_{j}(x_{1},\ldots,x_{j}) \times \\ & \left[\prod_{t=1}^{n-2}
  R({\bf x}_{[(t-1)d+1+j:td+j]},
{\bf x}_{[td+1+j:(t+1)d+j]}) \right] \times 
  \widetilde{E}_{d-j}(x_{nd-d+j+1},\ldots,x_{nd}). \notag
 \end{align}
In other words, there is a shift invariance in the interactions between
$x_{t},\ldots,x_{t+d}$ and $x_{t+d+1},\ldots,x_{t+2d}$ for all $t$. 
This forms the core component in the stationarity result that we state below. 
%We now obtain an analogous stationarity
%result to Theorem 
%\ref{theorem:one_stationarity}, but now for the general 
%$d$-Markov specification. 

\begin{theorem}\label{theorem:stationarityd} 
 Suppose that the conditional distributions are specified by
  (\ref{eq:fcond}) with $T$ defined as
  in (\ref{eq:int_oper_maind}). We assume that $T: L_{2}(\mathcal{X}_d)
  \rightarrow L_{2}(\mathcal{X}_{d})$ is a compact  operator.
Then, there exists a unique strictly stationary stochastic process $\{X_t\}_{t \in \mb Z}$ with $X_t = (X_t^{(1)}, \ldots, X_t^{(p)})$ and $X_t^{(a)} \in \m X^{(a)}$, such that the conditional distribution of any $X_t^{(a)}$ given $\m H_{(a, t)}$ is given by 
\eqref{eq:fcond}. 

Let $r :\, = r(T) = r(T^*)$ be the spectral radius of $T$ and $T^*$. Let $v$ and $w$ be a.e. positive eigenfunctions of $T$ and $T^*$ corresponding to $r$, appropriately scaled so that $\langle v, w\rangle = 1$. Then, for any $n \in \mb N$ and shift $m \in \mb Z$, the joint density function of $(X_{m},\ldots,X_{m+dn-1})$ (with respect to the $n$-fold product of $\mu$) is given by 
\begin{eqnarray*}
&&p_{[1:nd]}(x_{1},\ldots, x_{nd})\\
  &=&  \frac{1}{r^{n-1}}
  v(x_{1},\ldots,x_{d}) \left[\prod_{t=1}^{n-1}R(x_{(t-1)d+1},\ldots
  x_{td};x_{td+1},\ldots,x_{(t+1)d})\right]
  w(x_{(n-1)d+1},\ldots,x_{nd}).
  \end{eqnarray*}
%{\color{blue}need to edit}
\end{theorem}
Despite the apparent similarity, this result does not directly follow from Theorem \ref{theorem:one_stationarity} for reasons explained below. To keep the exposition simple, we focus on the $d = 2$ case, which nonetheless encapsulates most of the subtleties that arise when $d > 1$. Employing Theorem \ref{theorem:one_stationarity}, one obtains a {\it vector stationary process} $\{\xi_t = (Y_t, Z_t)\}_{t \in \mb Z}$ on $\m X_2$. We project this on a different timeline to {\it define} a stochastic process $\{X_t\}_{t \in \mb Z}$, where $X_{2t-1} = Y_t$ and $X_{2t} = Z_t$. By virtue of vector stationarity of $\{\xi_t\}$, it follows that for any {\it even} $s \in \mb Z$, $(X_1, \ldots, X_{2n}) \overset{d}= (X_{s+1}, \ldots, X_{s + 2n})$. However, this does not imply $\{X_t\}$ is stationary, since it is not even clear whether $X_{2t-1} \overset{d} = X_{2t}$. The key to establishing stationarity is to show that for {\it any} shift $s \in \mb Z$, $(X_1, \ldots, X_{2n}) \overset{d}= (X_{s+1}, \ldots, X_{s + 2n})$, with common joint distribution $p_{[1:2n]}$ as above. We establish this by exploiting the structural properties of $R$ (c.f. \eqref{eq:R_prop}) and the uniqueness (upto scaling) of the positive eigenfunctions of $T$ and $T^*$.

In Section \ref{sec:prob_cestgm} we obtain
$\beta$-mixing (Theorem \ref{theorem:mixing}) 
of the $1$-Markov CEStGM process and a result that allows us to simulate the process the $1$-Markov CEStGM process 
(Theorem \ref{theorem:simulation}). 
The same result holds for the $d$-Markov CEStGM process under the assumption that the interaction kernel $R$ is Hilbert-Schmidt. 

In the theorem below we show that the $d$-Markov 
CEStGM process statisfies the three Markov properties. 
The result is given without proof, as the proof is identical to the proof of Theorems
\ref{theorem:localMarkov} and \ref{theorem:globalMarkov}.

\begin{theorem}\label{theorem:graphd}
Suppose the conditions in Theorem \ref{theorem:stationarityd}
hold. We define the neighbourhood set as \begin{eqnarray*}
\mathcal{N}_{a} = \{b\in [p]\backslash\{a\};\textrm{ if }\Phi^{(a,b)}_{r}\neq 0 
\textrm{ for some }r\in \{-d,-d+1,\ldots,d-1,d\}\}.
\end{eqnarray*}
With %the neighborhood sets 
such $\{\mathcal{N}_{a};a\in [p]\}$,
we obtain analogous results to Theorems
\ref{theorem:localMarkov} \& \ref{theorem:globalMarkov}.
\end{theorem}

\subsection*{Concluding remarks}

In this article, we introduced CEStGM, a class of multivariate time series models
%(called CEStGM), 
which conveniently encode a process-wide conditional independence  graph. In CEStGM, a 
$p$-dimensional multivariate time series is defined in terms of  
$p$, separate, conditional distributions;
the general form is given in (\ref{eq:fcond}).
%The technical analysis of CEStGM draws on classical results in 
%spectral theory of integral operators.  
%to develop probabilistic tools. 
%which lead to the probabilistic tools at draw from spectral theory of integral operators. 
We show that these $p$ conditional distributions uniquely define a
positive kernel $R$ which we call the {\it interaction kernel}. The interaction kernel and its associated linear integral operator are instrumental in verifying stationarity.  Specifically,
given the conditional specification in (\ref{eq:fcond}), we show that if the parameters satisfy the   
 following process-wide compatability conditions (i)
$\Phi_0^{(a,a)} = 0$, $\Phi^{(a,b)}_{0} = (\Phi^{(b,a)}_0)^{\top}$, 
$\Phi^{(a,b)}_{\ell} = (\Phi^{(b,a)}_{-\ell})^{\top}$ (for $\ell\in \{1,\ldots,d\}$) and (ii) the corresponding integral operator is compact, then the existence of a stationary process with the stated conditionals is ensured. A sufficient condition for (ii) to hold is that the interaction kernel is Hilbert--Schmidt; thus conditions (i) and (ii) are simple to verify on a case by case basis. Under the Hilbert--Schmidt condition we show that CEStGM is ergodic and geometrically $\beta$-mixing. Further,  we show the Hilbert-Schmidt condition also implies 
that the sparsity structure of 
$\{\Phi_{\ell}^{(a,b)}\}$ leads to a process-wide conditional independence graph that satisfies  process-wide versions of the classical Markov properties.

Kernels and integral operators share a rich connection with many prominent developments in statistics. For example, symmetric positive-definite kernels are fundamental to the construction of 
reproducing kernel Hilbert spaces or covariance kernels of Gaussian processes. Markov kernels have played a major role in time-series modeling as well as Markov chain Monte Carlo (MCMC) algorithms for scientific computation, notably in Bayesian statistics. Interestingly, our interaction kernel is neither positive-definite (it may not even be symmetric), nor is it Markovian. All the technical developments in this article instead stem solely from positivity of the interaction kernel; thanks to the fascinatingly rich spectral properties of positive linear operators, whose application in statistics is less prevalent to the best of our knowledge. Using this machinery we develop the crucial probabilistic tools, which, as far as are aware, are new. 

In the case that $\{X_{t}\}$ are independent, identically distributed high dimensional random vectors the partial likelihood together with 
a penalty has shown promise in 
estimating the conditional independence graph \cite{p:yan-14,p:yan-15,p:rav-10}. 
We conjecture that the partial likelihood can also be utilized to learn the non-zero parameters in CEStGM and thus estimate the conditional independence process-wide graph. In future work, we hope to 
use the probabilistic properties derived in this article, such as 
ergodicity and geometric $\beta$-mixing, to develop statistical theory for such estimation methods. Another interesting direction would be to develop Bayesian estimation for CEStGM, using the partial likelihood as a pseudo-likelihood \cite{p:chernozhukov-hong-2003,p:atchade-2017}, and develop theoretical guarantees for parameter estimation and graph selection.

\subsection*{Acknowledgments}

The research of AB and SSR is partially supported by the National
Science Foundation (grants DMS-2210689 and DMS-2210726
respectively). The authors thank 
Dr.\ Jochen Gl\"uck, whose clear and detailed comments on Math StackExchange helped guide the authors to \cite{schaefer1974banach}, and Jaeseon Lee, for careful reading of the manuscript.

\bibliography{Biblio/graph.bib,Biblio/FA_refs.bib,Biblio/tsreg}
\bibliographystyle{plainnat}

\newpage

\section*{Supplementary Material}\label{supplementary_material}

% Ensure a new page
%\addcontentsline{toc}{section}{SUPPLEMENTARY MATERIAL}

%\section*{Supplement to ``Conditionally specified stationary stochastic processes for graphical modeling of multivariate time series"}\label{supplementary_material}

% Reset subsection numbering
%\newcounter{counter}[section]
\renewcommand*{\theHsection}{\thesection}
\renewcommand*{\theHsubsection}{\thesubsection}
% Reset section numbering
\setcounter{section}{0} 
\setcounter{equation}{0}
\setcounter{subsection}{0}

\renewcommand \thesection{S\arabic{section}}
\renewcommand \thesubsection{\thesection.\arabic{subsection}}
\renewcommand \thesubsubsection{\thesubsection.\arabic{subsubsection}}
\renewcommand\thetable{S.\arabic{table}}
\renewcommand \thefigure{S.\arabic{figure}}
\renewcommand{\theequation}{S.\arabic{equation}}

\setcounter{table}{0}
\setcounter{figure}{0}

% --- AB defs -- %
\def\m{\mathcal}
\def\mb{\mathbb}
\def\mr{\mathrm}

In the supplementary material we prove the results in the paper and provide some additional
background and examples. We give a summary below.
\begin{enumerate}
\item In Section \ref{sec:FA_rev} we recall some results from functional analysis which pertain to the
  Krein-Rutman Theorem summarized in Theorem \ref{thm:KR_without_int_main}.
\item In Section \ref{sec:pf_one_stationarity} we apply the Krein-Rutman theorem to prove Theorem \ref{theorem:one_stationarity}.
\item In Section \ref{sec:proofgraph} we prove the process-wide conditional independence properties
  stated in Section \ref{sec:graph}.
\item In Section \ref{sec:power} we prove the probabilistic properties and 
  power iteration results stated in  Section \ref{sec:prob_cestgm}.
\item In Section \ref{sec:cased} we prove the stationarity result for the 
  general $d$-Markov case stated in Theorem \ref{theorem:stationarityd}.
\item In Section \ref{sec:addnl_derivs} we give the derivations stated in several of the remarks in the paper.
\item In Section \ref{sec:examples} we give concrete examples of multivariate models that satisfy the
  process-wide compatability conditions, together with their process-wide conditional independence graphs
  (according to the process-wide Markov properties) and some realisations using the approximate
  Gibbs sampler described in Section \ref{sec:prob_cestgm}.
\end{enumerate}

\section{Relevant results from functional analysis}\label{sec:FA_rev}
We first recall some standard facts about bounded linear operators (esp. compact operators) on Hilbert spaces. Details can be found in any standard functional analysis textbook, e.g., \cite{b:con-90}[Chapter 2]. 

Let $H$ be a Hilbert space over $\C$ with inner product $\langle \cdot, \cdot \rangle$ and norm $\|\cdot\|$. Let $\m B(H)$ denote the space of bounded linear maps $T: H \to H$. The {\it operator norm} of $T \in \m B(H)$ (induced by the Hilbert space norm $\|\cdot\|$) is $\|T\|_{\rm op} :\, = \sup \{\|T x\| \,:\, \|x\| = 1\}$. For every $T \in \m B (H)$, there exists a unique operator $T^*$, called the {\it adjoint} of $T$, with $\langle Tf, g \rangle  = \langle f, T^* g \rangle$ for all $f, g \in H$. It is known that $\|T\|_{\rm op} = \|T^*\|_{\rm op}$. $T$ is called {\it self-adjoint} if $T = T^*$. 

$T$ is called {\it finite rank} if $\mbox{range}(T)$ is finite-dimensional. $T$ is called {\it compact} if $T$ maps the unit ball to a relatively compact set, that is, a set whose closure is compact. Any finite rank operator is compact. Moreover, $T$ is compact if and only if there exists a sequence of finite-rank operators such that $\|T_n - T\|_{\rm op} \to 0$. If $T$ is compact, then $T^*$ is also compact. 

The {\it resolvent set} of $T$ is defined as 
\begin{align*}
\rho(T) :\, = \{ z \in \C \,:\, (T - z \, \mbox{id}) \text{ is bijective and } (T - z \, \mbox{id})^{-1} \in \m B(H)\}, 
\end{align*}
where $\mbox{id}$ is the identity operator. 
The {\it spectrum} of $T$ is defined as the complement set $\sigma(T) :\,= \mathbb{C} \setminus \rho(T)$. The {\it point spectrum} $\sigma_p(T) := \{\lambda \in \C \,:\, Tf = \lambda f \text{ for some } f \in H\}$ is the set of all eigenvalues. $\lambda \in \sigma_p(T)$ is called {\it simple} if the nullspace of $(T - \lambda \, \mbox{id})$ has dimension 1. If $H$ is finite-dimensional, $\sigma(T)$ is the set of all eigenvalues. However, this no longer holds generally for infinite-dimensional spaces, and one only has $\sigma_p(T) \subset \sigma(T)$. However, {\it if $T$ is compact}, then the spectrum is countable, zero is the only accumulation point inside the spectrum, and every non-zero element of $\sigma(T)$ is an eigenvalue. 

The spectral radius of $T$ is defined as 
\begin{align}\label{eq:sp_radius}
  r(T) :\, = \sup\{ |\lambda| \,:\, \lambda \in \sigma(T)\}  
\end{align} 
An important fact is that $\sigma(T^*) = \sigma(T)$, and therefore $r(T^*) = r(T)$.

Gelfand's spectral radius formula gives 
\begin{align} \label{eq:gelfand}
r(T) = \lim_{n \to \infty} \|T^n\|_{\rm op}^{1/n}. 
\end{align}
Since by sub-multiplicativity of the operator norm $\|T^n\|_{\rm op} \le \|T\|_{\rm op}^n$, this in particular implies $r(T) \le \|T\|$. Another important characterization \citep{holmes1968formula} of the spectral radius is that 
\begin{align}\label{eq:sp_inf}
r(T) = \inf \Big\{|T|_{\rm op} \,:\, \text{ where the inf is over all norms $|\cdot|$ equivalent to $\|\cdot\|$} \Big\},
\end{align}
where recall $|\cdot|$ is equivalent to $\|\cdot\|$ if there exist constants $a, b > 0$ (depending on $|\cdot|$) such that $a |x| \le \|x\| \le b |x|$ for all $x \in H$, and $|T|_{\rm op}$ denotes the norm induced by $|\cdot|$, i.e., $|T|_{\rm op} = \sup\{ |Tx| \,: \, |x| = 1\}$. 
% {\color{blue}Should we in some way denote that $a$ and $b$ depend on the norm $|\cdot|$?
% Further, do you think we should change $|\cdot|$ to $|\cdot|_{op}$ when the operator norm is used? To be consistent with $\|\cdot\|$ and $\|\cdot\|_{op}$
% }
\\[1ex]
{\it Integral operator.} Of particular interest to us is the following setting. Let $(\m X, \m A, \mu)$ be a sigma-finite measure space and let $L^2(\m X, \mu) :\, = \{f : \m X \to \C \,:\, \int |f|^2 d \mu < \infty\}$, with the usual $L^2$ inner product $\langle f, g \rangle = \int f \overline{g} d\mu$. Let $R: \m X \times \m X \to \C$ be a measurable function (commonly called {\it kernel}), and let $T : L^2(\m X, \mu) \to L^2(\m X, \mu)$ be the corresponding {\it integral operator} given by 
\begin{align}\label{eq:int_oper}
T(f)[y] = \int R(x, y) f (x) \mu(dx), \quad y \in \m X. 
\end{align}
If $\max \left\{\int |R(x, y)| \mu(dx), \int |R(x, y)| \mu(dy) \right\} \le c$ $\mu$-a.e., then $T$ is a bounded linear operator with $\|T\| \le c$. Also, the adjoint 
\begin{align}
T^*(f)[y] = \int \overline{R(y,x)} f(x) \mu(dx). 
\end{align}
If the kernel $R$ is square integrable, i.e., if $R(\cdot, \cdot) \in L^2(\m X \times \m X, \mu \otimes \mu)$, so that 
\begin{align}\label{eq:HS_def}
    \|R(\cdot, \cdot)\|^2  = \int |R(x, y)|^2 \mu(dx)\mu(dy) < \infty,
\end{align} 
then $T$ is called {\it Hilbert--Schmidt}, and one has $\|T\|_{\rm op} \le \|R(\cdot, \cdot)\|$. A Hilbert--Schmidt operator is compact. 
%, a result we make repeated use of to verify compactness. 
\\[2ex]
{\it Positive operators and Krein--Rutman Theorem.} 
% Recall from Section \ref{sec:cstgm} that an integral operator $T$ on $L^2(\m X, \mu)$ is called positive (resp. non-negative) if the kernel $R(x, y) > (resp. \ge ) 0$ a.e. $\mu \otimes \mu$. Let $K = \{f \in L^2(\mathcal{X}, \mu) \,:\, f \ge 0 \text{ a.e. } \mu\}$ denote the {\it cone} of non-negative functions in $L^2(\mathcal{X}, \mu)$. It is straightforward to see that if $f \in K$, then for a non-negative operator $T$, one has $T(f) \in K$, which implies $T(K) \subseteq K$. %Moreover, if $f \ne 0 \in K$, then $T(f) > 0$ a.e. $\mu$. 
% The notion of a cone can be generalized to Banach spaces, and the condition $T(K) \subseteq K$ is referred to as {\it the operator $T$ leaving the cone $K$ invariant} \citep[Chapter 2]{boelkins1998spectral}. 
%We state (a special case of) \cite{schaefer1974banach}[Theorem 6.6]. 
Recall Definition \ref{def:spo}. We now state the complete version of \cite[Theorem 6.6 (Chapter V)]{schaefer1974banach}, which implies Theorem \ref{thm:KR_without_int_main} as a special case. The only modification we make is to state the result for $L^2$ spaces, while \cite[Theorem 6.6 (Chapter V)]{schaefer1974banach} holds more generally for $L^p$ spaces. 
\begin{theorem}[Theorem 6.6 (Chapter V) of \cite{schaefer1974banach}]\label{thm:KR_without_int}
Let $(\m X, \m A, \mu)$ be a sigma-finite measure space, and let $T: L^2(\m X, \mu) \to L^2(\m X, \mu)$ be an integral operator as in \eqref{eq:int_oper} with the kernel $R(x, y) \ge 0$ for all $(x, y) \in \mathcal{X} \times \mathcal{X}$. 
%Suppose $T$ is compact. 
%Let $T \in \m B(X)$ be an integral operator as in \eqref{eq:int_oper} with the kernel $R(x, y) \ge 0$ for all $(x, y) \in \mathcal{X} \times \mathcal{X}$. 
Suppose 
\begin{itemize}
\item [(i)] Some power of $T$ is compact. 

\item[(ii)] $S \in \m A$ and $\mu(S) > 0, \mu(\m X \setminus S) > 0$ implies 
\begin{align}\label{eq:irred}
    \int_{\m X \setminus S} \int_S R(x, y) \mu(dx) \mu(dy) > 0. 
\end{align}
\end{itemize}
Then, 
\begin{itemize}
\item [(a)] The spectral radius $r(T)$ is positive. 

\item [(b)] The spectral radius $r(T)$ is an eigenvalue of $T$, and has a unique normalized eigenfunction $v \in L^2(\m X, \mu)$ with $\|v\| = 1$ satisfying $v > 0 \text{ a.e. } \mu$. 

\item [(c)] If $R(x, y) > 0 \text{ a.e. } \mu \otimes \mu$, then every other eigenvalue $\lambda$ of $T$ satisfies $|\lambda| < r(T)$. %{\color{blue}Are you citing Victory here?}
\end{itemize}
\end{theorem}
%An integral operator $T$ with $R \ge 0$ (resp. $R > 0$) is called a {\it non-negative} (resp. {\it positive}) operator. 
The condition \eqref{eq:irred} is called {\it irreducibility}, and ensures the spectral radius is positive\footnote{Note that there are examples of compact positive operators, such as the Volterra operator, for which $r(T) =0$; see Example 3.3 of Boelkins.}. If $R(x, y) > 0$ a.e., then the irreducibility condition is automatically satisfied, and hence Theorem \ref{thm:KR_without_int_main} follows as a special case of Theorem \ref{thm:KR_without_int}. %Theorem \ref{thm:KR_without_int} asserts that the spectral radius of a positive compact irreducible operator is a simple eigenvalue with a unique positive normalized eigenfunction, and any other eigenvalue has smaller magnitude. 

\begin{remark}\label{rem:schaefer_vs}
Many versions of the Krein--Rutman theorem exist in the literature. Some of them (see, e.g. \cite[Theorem 1.4]{boelkins1998spectral}, \cite[Theorem 19.2]{deimling2013nonlinear}) only guarantee that the spectral radius is an eigenvalue with a positive eigenfunction. To ensure $r(T)$ is simple and any other eigenvalue $\lambda$ satisfies $|\lambda| < r(T)$, there are versions of these results (see, e.g \cite[Theorem 1.5]{boelkins1998spectral}, \cite[Theorem 19.3]{deimling2013nonlinear}) that use additional conditions, one of which is that the {\it interior of the cone is non-empty}. When the Banach space $X$ is $C[0,1]$ (the space of continuous functions on $[0,1]$ equipped with the supremum norm) and the cone $K = \{f \in C[0, 1] \,:\, f \ge 0\}$, it has a non-empty interior. 
%However, if we let the domain be infinite and consider continuous functions which decay to zero at the boundary, then the non-negative cone has empty interior (CITE). Moreover, 
However, the non-negative cone in $L^2(\m X, \mu)$ has empty interior (see Example 2.4 of \cite{boelkins1998spectral} for a proof), and hence these stronger results do not apply to our situation. Theorem 6.6 (Chapter V) of \cite{schaefer1974banach}  quoted above fits our need since it covers $L^p$ spaces. 
\end{remark}

\section{Proof of Theorem \ref{theorem:one_stationarity}}\label{sec:pf_one_stationarity}
As noted in the discussion after Theorem \ref{theorem:one_stationarity}, we crucially use the following identities: for any positive integer $n > 1$, one has  
\begin{align}
\begin{aligned}\label{eq:eigenmagic}
\int p_{[1:n]}(x_1, \ldots, x_n) \mu(dx_1) & = p_{[1:n-1]}(x_2, \ldots, x_n), \\
\int p_{[1:n]}(x_1, \ldots, x_n) \mu(dx_n) & = p_{[1:n-1]}(x_1, \ldots, x_{n-1}). 
\end{aligned}
\end{align}
These follow from the facts that $\int R(x, y) v(x) \mu(dx) = r v(y)$ and $\int R(x, y) w(y) \mu(dy) = r w(x)$. 

We begin by specifying finite-dimensional distributions $\pi_{s_1, \ldots, s_k}$ of $(X_{s_1}, \ldots, X_{s_k})$ under CEStGM. Since the index set is $\mb Z$, it suffices to construct $\pi_{s_1, \ldots, s_k}$ for ordered tuples $s_1 < \ldots < s_k$; see \cite[Example 36.4]{billingsley2017probability}. Kolmogorov's consistency theorem then assumes the following simplified form (see eq. (36.19) in \cite[Example 36.4]{billingsley2017probability}); for any $i \in [k]$, one needs to verify that  
\begin{align}\label{eq:kolmogorov_simplified}
& \pi_{s_1, \ldots, s_{i-1}, s_{i+1}, \ldots, s_k}(B_1 \times \cdots \times B_{i-1} \times B_{i+1} \times \cdots \times B_k)  \notag \\
& = \pi_{s_1, \ldots, s_k}(B_1 \times \cdots \times B_{i-1} \times \m X \times B_{i+1} \times \cdots \times B_k) 
\end{align}
for all %{\color{blue}for all maybe} 
measurable sets $B_i \subseteq \m X$. 

Recall for any $m \in \mb Z$, the joint distribution of $(X_m, \ldots, X_{m+n-1})$ at $n$ contiguous locations is given by $\pi_{[1:n]}$, which has density $p_{[1:n]}$ with respect to the $n$-fold product of $\mu$.  Using these, we define $\pi_{s_1, \ldots, s_k}$ as follows. For $s_1 < \ldots < s_k$ with $s_i \in \mb Z$ for all $i$, let $\ell = s_k - s_1 + 1$, and define 
\begin{align}
\pi_{s_1, \ldots, s_k}(B_1 \times \cdots \times B_k) = \pi_{[1:\ell]}(\widetilde{B}), \quad \widetilde{B} = \widetilde{B}_1 \times \cdots \times \widetilde{B}_\ell%\times_{j=1}^{\ell} \widetilde{B}_j, 
\end{align}
where $\widetilde{B}_j = B_j$ if $j \in \{1, \ldots, k\}$, and $\widetilde{B}_j  = \m X$ otherwise. Clearly, $\pi_{s_1, \ldots, s_k}$ has density $p_{s_1, \ldots, s_k}$ with respect to the $\ell$-fold product of $\mu$, where 
\begin{align}\label{eq:nu_den}
p_{s_1, \ldots, s_k}(x_{s_1}, \ldots, x_{s_k}) = \int p_{[1:\ell]}(x_{s_1}, x_{s_1+1}, \ldots, x_{s_k-1}, x_{s_k}) \prod_{j \in \m J}  \mu(d x_j),   
\end{align}
where $\m J = [s_1: s_k] \setminus \{s_1, \ldots, s_k\}$, and $[a:b] = \{a, a+1, \ldots, b\}$ for $a, b \in \mb Z$ with $a < b$. For example, 
\begin{align*}
p_{3, 5, 8}(x_3, x_5, x_8) = \int p_{[1:6]}(x_3, x_4, x_5, x_6, x_7, x_8) \, \mu(dx_4) \mu(dx_6) \mu(dx_7).   
\end{align*}
Thus, \eqref{eq:kolmogorov_simplified} amounts to showing for any $i \in [k]$ that 
\begin{align}\label{eq:ks_den}
p_{s_1, \ldots, s_{i-1}, s_{i+1}, \ldots, s_k}(x_{s_1}, \ldots, x_{s_{i-1}}, x_{s_{i+1}}, \ldots x_{s_k}) = \int p_{s_1, \ldots, s_k}(x_{s_1}, \ldots, x_{s_k}) \, \mu(dx_{s_i}). 
\end{align}
If $i \in \{2, \ldots, k-1\}$, the above conclusion is immediate, since $p_{s_1, \ldots, s_{i-1}, s_{i+1}, \ldots, s_k}$ and $p_{s_1, \ldots, s_k}$ both originate upon marginalization from $p_{[1:\ell]}$ with $\ell = s_k - s_1 + 1$. %However, if $i = 1, k$, this ceases to be the case. 
However, the two boundary cases ($i = 1, k$) are more subtle.
We prove the case $i = 1$, the case $i = k$ follows similarly. 
By definition of the joint measure in 
\eqref{eq:nu_den}, we have 
%We have from \eqref{eq:nu_den} 
that
\begin{align}\label{eq:ks_den_leftminus}
p_{s_2, \ldots, s_k}(x_{s_2}, \ldots, x_{s_k}) = \int p_{[1:\ell']}(x_{s_2}, x_{s_2+1}, \ldots, x_{s_k-1}, x_{s_k}) \, \prod_{j \in \m J'} \mu(d x_j),
\end{align}
where $\ell' = s_k -s_2 + 1$ and $\m J' = [s_2: s_k] \setminus \{s_2, \ldots, s_k\}$. In order %that the measure is consistent
to verify the consistency condition \eqref{eq:ks_den},
we need to show that it is equivalent to the joint measure 
derived from $[X_{s_1},\ldots,X_{s_{n}}]$ i.e. that the above is equal to 
\begin{align*}
 \int p_{s_1, \ldots, s_k}(x_{s_1}, \ldots, x_{s_k}) \, \mu(dx_{s_1}). 
\end{align*}
We establish that 
\begin{align*}
p_{s_2, \ldots, s_k}(x_{s_2}, \ldots, x_{s_k}) = \int p_{s_1, \ldots, s_k}(x_{s_1}, \ldots, x_{s_k}) \, \mu(dx_{s_1}). 
\end{align*}
Start from the right hand side of the above display and substitute its definition from \eqref{eq:ks_den}. Interchanging the order of integrals and repeatedly applying the first identity of \eqref{eq:eigenmagic}, we have 
\begin{align*}
& \int p_{s_1, \ldots, s_k}(x_{s_1}, \ldots, x_{s_k}) \, \mu(dx_{s_1})
\\
& = \int \int p_{[1:\ell]}(x_{s_1}, x_{s_1+1}, \ldots, x_{s_k-1}, x_{s_k}) \, \mu(dx_{s_1}) \, \prod_{j \in \m J} \mu(dx_j) \\
& = \int p_{[1:\ell-1]}(x_{s_1+1}, x_{s_1+2}, \ldots, x_{s_k-1}, x_{s_k}) \, \prod_{j \in \m J} \mu(dx_j) \\
& = \cdots \\
& = \int \int p_{[1:\ell'+1]}(x_{s_2-1}, x_{s_2}, \ldots, x_{s_k-1}, x_{s_k}) \mu(dx_{s_2-1}) \, \prod_{j \in \m J'} \mu(dx_j) \\
& = \int p_{[1:\ell']}(x_{s_2}, x_{s_2+1}, \ldots, x_{s_k-1}, x_{s_k}) \, \prod_{j \in \m J'} \mu(dx_j) = p_{s_2, \ldots, s_k}(x_{s_2}, \ldots, x_{s_k}). 
\end{align*}
This establishes the desired result. The case $i = k$ proceeds similarly, where we repeatedly use the second identity of \eqref{eq:eigenmagic}. 

% \begin{align}
% \widetilde{B}_j = 
% \begin{cases}
% B_j & \text{ if } j \in \{1, \ldots, k\}, \\
% \m X & \text{ otherwise }. 
% \end{cases}
% \end{align}

%\textcolor{blue}{Should we call this section proof of results in this paper? Because there isn't really a remaining I think.}

\section{Proof of results in Section \ref{sec:graph}}\label{sec:proofgraph}

We first analysis the interaction kernel $R(\cdot,\cdot)$. 

\paragraph{A Hammersley-Clifford type factorisation of the interaction kernel $R$}
We define the following graph $\mathcal{G} = (V,E)$ where $V=[p]$ and
$(a,b)\in E$ iff either $\Phi^{(a,b)}_{0}\neq 0$, 
$\Phi^{(a,b)}_{1}\neq 0$ or $\Phi^{(a,b)}_{-1}\neq 0$  (i.e., if $b \in \m N_a$ or equivalently, $a \in \m N_b$, with the neighborhood set $\m N_a$ defined in \eqref{eq:cestgm_nhbr}). From $\mathcal{G}$ we obtain the set of all induced cliques $\mathcal{C}$. In other words if $D\in \mathcal{C}$ then $D\subseteq V$
and the corresponding subgraph in $\mathcal{G}$ is complete (all nodes
in $D$ are connected to each other). By using the definition of $G(x)$ and $H(x,y)$ in
(\ref{eq:GH_def}) we obtain a factorisation of the interaction kernel $R(\cdot,\cdot)$.
\begin{lemma}\label{lemma:cliques}
Suppose the assumptions in Theorem  \ref{theorem:one_stationarity}
hold.   Let $R(\cdot,\cdot)$ be defined as in (\ref{eq:int_ker}). Then 
\begin{eqnarray}\label{eq:cli-ha}
R(x,y) = \prod_{D\in \mathcal{C}}R_{D}(x^{(D)},y^{(D)})
\end{eqnarray}  
where $x^{(D)} = (x^{(a)};a\in D)$. Further,
\begin{eqnarray}
\label{eq:pnpd}  
p_{[-M:M]}(x_{-M},\ldots,x_{M})
  &=& \frac{1}{r^{2M}} \,   v(x_{-M})
      \left[\prod_{t=-M+1}^{M} \prod_{D\in \mathcal{C}}R_{D}(x_{t-1}^{(D)},x_{t}^{(D)}) \right] w(x_{M}). 
\end{eqnarray}
\end{lemma}
\begin{proof}
  We construct $R_{D}(x^{(D)},y^{(D)})$, the result follows
  immediately from this construction.
  
  Let $p_{a}$ and $p_{(a,b)}$ denote the number of cliques associated with node $a$
  and edge $(a,b)$ respectively.   For $x, y \in \m X$, define  
\begin{align*}
G_{D}(x^{(D)}) =  \exp\left(\theta_{D}^{\top}{\bf
      s}_{D}(x^{(D)})+ {\bf 1}_{|D|}^\top {\bf c}_{D}(x) + \frac{1}{2} \, x_{D}^{\top}\Phi_{D,0}x_{D}\right), \quad
H_{D}(x^{(D)},y^{(D)}) = \exp\left(
      x^{\top}_{D}\Phi_{D,1} y_{D}\right)
\end{align*}
where
\begin{align*}
  \theta_{D} &= \left(\frac{1}{p_{a}}\theta_{a};a\in D\right), \quad
               s_{A}(x^{(D)}) = \left(\frac{1}{p_a}s_{a}(x);a\in D\right),
\end{align*}
and  $\Phi_{D,j}$ is an $|D|\times|D|$-dimensional matrix where
\begin{align*}
  \Phi_{D,j} & = \left( \frac{1}{p_{a}}\Phi_{j}^{(a,a)} \textrm{ for
               }a\in D; \,\frac{1}{p_{(a,b)}}\Phi_{j}^{(a,b)} \textrm{ for
               }a\neq b\in D \right).
\end{align*}
We let $R_{A}(x^{(D)},y^{(D)}) = G(x^{(D)},y^{(D)})^{1/2}H(x^{(D)},y^{(D)}) G(x^{(D)},y^{(D)})^{1/2}$
and the factorisation in (\ref{eq:cli-ha}) immediately follows.

Further by using Theorem \ref{theorem:one_stationarity} and the above factorisation, the joint
density of $x_{-M},\ldots,x_{M}$ is (\ref{eq:pnpd}).
\end{proof}  

Let $C_{a}$ denote the set of all cliques which contain the node $a$. Let $x^{(D_{a})} = (x^{(a)},x^{(D_{a}\backslash{a})})$. Then
we define the conditional kernel associated with node $a$ as 
\begin{eqnarray}
\label{eq:Ra}  
  R^{(a)}(x^{(a,N_{a})},y^{(a,N_{a})}) =
  \frac{\prod_{D\in \mathcal{C}_{a}}R_{D}((x^{(a)},x^{D\backslash{a}}), (y^{(a)},y^{D\backslash{a}}))}{\prod_{D\in \mathcal{C}_{a}}R_{D}((0,x^{D\backslash{a}}), (0,y^{D\backslash{a}}))}.
\end{eqnarray}
We divide by $R_{D}((0,x^{D\backslash{a}}),
  (0,y^{D\backslash{a}})$ to remove all interactions not associated
  with $x^{(a)}$.

We will use the above result to obtain the
conditional density of 
$(X_{-n}^{(a)},\ldots,X_{n}^{(a)})$ based on different 
conditional sets. We will then invoke  
Theorem \ref{theorem:cond-independenceSigma}, 
in Section \ref{sec:sigmalimit}, to 
prove the local and global Markov properties.

\subsection{The Local Markov property}

Some notation: Let $n \in \mb N$. 
\begin{itemize}
\item For $a \in [p]$, denote $x_{[-n,n]}^{(a)} =
  (x_{-n}^{(a)},x_{-n+1}^{(a)},\ldots,x_{n-1}^{(a)},x_{n}^{(a)})$.
\item  For $K \subseteq [p]$, denote
$x_{[-n,n]}^{K} =
  (x_{-n}^{(b)},x_{-n+1}^{(b)},\ldots,x_{n-1}^{(b)},x_{n}^{(b)};b\in
  K)$.
 \item Let $x_{(\infty,n)}^{(a)} = (x_{t}^{(a)};|t|\geq n)$.
\item Recall $\mathcal{N}_{a}$ denotes the neighbourhood set of $a$, and
  $\mathcal{N}_{a}^{\prime}=[p]\backslash\{a\cup N_{a}\}$.  
\end{itemize}  

\begin{lemma}\label{lemma:finitelocal}[Local Markov property]
Suppose the assumptions in Theorem  \ref{theorem:one_stationarity}
hold. Then the conditional density of
$X_{-n}^{(a)},\ldots,X_{n}^{(a)}$ given ${\bf X}^{\mathcal{N}_a}_{[-n-L-k_1,n+L+k_1]},
{\bf X}^{\mathcal{N}_{a}^{\prime}}_{[-n-L-k_2,n+L+k_2]},{\bf
  X}_{(\infty,n+L)}^{(a)}$ is equal to the conditional density of
$X_{-n}^{(a)},\ldots,X_{n}^{(a)}$ given ${\bf X}^{\mathcal{N}_a}_{[-n-L-k_1,n+L+k_1]},{\bf
  X}_{(\infty,n+L)}^{(a)}$ i.e.
\begin{eqnarray}
 \label{eq:tvconditional0}   
  &&p(x_{-n}^{(a)},\ldots,x_{n}^{(a)}|{\bf x}^{\mathcal{N}_a}_{[-n-L-k_1,n+L+k_1]},
{\bf x}^{\mathcal{N}_{a}^{\prime}}_{[-n-L-k_2,n+L+k_2]},{\bf x}_{(\infty,n+L)}^{(a)})\nonumber\\
  &=& p(x_{-n}^{(a)},\ldots,x_{n}^{(a)}|{\bf
    x}^{\mathcal{N}_a}_{[-n-L:n+L]},x_{-n-L}^{(a)},x_{n+L}^{(a)})
      \quad\textrm{for all } n,L,k_1,k_2>0.
\end{eqnarray}
\end{lemma}  
\begin{proof}
We start by studying the conditional densities for finite subsets of the time
series $\{X_{t}^{(a)}:t\in \mathbb{Z}\}$.
By using (\ref{eq:cli-ha}), (\ref{eq:Ra}) and  (\ref{eq:pnpd}), the conditional distribution of
$X_{-n}^{(a)},\ldots,X_{n}^{(a)}$ given
$X^{V\backslash\{a\}}_{[-(n+1):(n+1)]},X_{-(n+1)}^{(a)},X_{n+1}^{(a)}$ is
\begin{align}
\label{eq:pnpdcond}  
  p(x_{-n}^{(a)},\ldots,x_{n}^{(a)}|{\bf
     x}^{V\backslash\{a\}}_{[-(n+1):(n+1)]},x_{-(n+1)}^{(a)},x_{n+1}^{(a)}) 
  =
  \frac{\prod_{t=-n}^{n+1}R^{(a)}(x^{(a,\mathcal{N}_{a})}_{t-1},
  x_{t}^{(a,\mathcal{N}_{a})})}{\int
  \prod_{t=-n}^{n+1}R^{(a)}(x^{(a,\mathcal{N}_{a})}_{t-1},x_{t}^{(a,\mathcal{N}_{a})})
  \prod_{i=-n}^{n}\mu(dx^{(a)}_{i})}.
\end{align}
Hence, by conditioning on all the other time series (except for $a$)
in the range $t\in [-(n+1),(n+1)]$ \emph{and} conditioning on $x_{-(n+1)}^{(a)}$
and $x^{(a)}_{n+1}$, the (conditional) distribution of 
$X_{-n}^{(a)},\ldots,X_{n}^{(a)}$ only depends 
series in the neighborhood set $\mathcal{N}_{a}$ and the anchor
points $(x_{-(n+1)}^{(a)},x^{(a)}_{(n+1)})$. That is  
\begin{align}
\label{eq:pmx}  
p(x_{-n}^{(a)},\ldots,x_{n}^{(a)}|{\bf
  x}^{V\backslash\{a\}}_{[-(n+1):n+1]},x_{-(n+1)}^{(a)},x_{n+1}^{(a)}) 
  = p(x_{-n}^{(a)},\ldots,x_{n}^{(a)}|{\bf
  x}^{\mathcal{N}_{a}}_{[-(n+1):(n+1)]},
  x_{-(n+1)}^{(a)},x_{(n+1)}^{(a)}).
\end{align}
Now we focus on the conditional distribution of
$X_{-n}^{(a)},\ldots,X_{n}^{(a)}$ conditioned on 
random variables that are further in the past and future.
Using a similar argument as above we have 
\begin{eqnarray}
 &&p(x_{-n}^{(a)},\ldots,x_{n}^{(a)}|{\bf
    x}^{V\backslash\{a\}}_{[-n-L:n+L]},x_{-n-L}^{(a)},x_{n+L}^{(a)}) \nonumber\\
  &=& \int p(x_{-n-L+1}^{(a)},\ldots,x_{n+L-1}^{(a)}|{\bf x}^{V\backslash\{a\}}_{[-n-L:n+L]},x_{-n-L}^{(a)},x_{n+L}^{(a)})
      \prod_{i\in \mathcal{I}}\mu(dx^{(a)}_{t}) \nonumber\\
  &=& p(x_{-n}^{(a)},\ldots,x_{n}^{(a)}|{\bf
    x}^{\mathcal{N}_a}_{[-n-L:n+L]},x_{-n-L}^{(a)},x_{n+L}^{(a)})\label{eq:conditionalmore}
\end{eqnarray}
where $\mathcal{I} = \{-n-L+1,\ldots,-n+1\}\cup \{n+1,\ldots,n+L-1\}$ and the
last line is due to (\ref{eq:pnpdcond}).
Next we show that for all $k>0$
\begin{eqnarray*}
 p(x_{-n}^{(a)},\ldots,x_{n}^{(a)}|x^{V\backslash\{a\}}_{[-n-L-k,n+L]},x_{-n-L}^{(a)},x_{n+L}^{(a)})
  &=& p(x_{-n}^{(a)},\ldots,x_{n}^{(a)}|{\bf
    x}^{\mathcal{N}_a}_{[-n-L:n+L]},x_{-n-L}^{(a)},x_{n+L}^{(a)}).
\end{eqnarray*}
We start with the joint density in (\ref{eq:pnpd})
partition the product of the interaction kernel into two parts
\begin{eqnarray*}
\prod_{t=-n-L-k+1}^{n+L} R(x_{t-1},x_{t}) = \prod_{t=-n-L-k+1}^{-n-L}R(x_{t-1},x_{t})
  \prod_{s=-n-L+1}^{n+L}R(x_{s-1},x_{s})
\end{eqnarray*}
using this we have
\begin{eqnarray*}
&&p(x_{-n}^{(a)},\ldots,x_{n}^{(a)}|{\bf x}^{V\backslash\{a\}}_{[-n-L-k,n+L]},x_{-n-L}^{(a)},x_{n+L}^{(a)})\\
%  &=&  \frac{ \int v(x_{-n-L-k})\prod_{t=-n-L-k}^{-n-L}R(x_{t-1},x_{t})
%      \prod_{s=-n-L+1}^{n+L}R(x_{s-1},x_{s}) w(x_{n+L})
%      \prod_{i \in \mathcal{I}_0\cup \mathcal{I}_1}d\mu(x_i^{(a)})}{
%  \int  v(x_{-n-L-k})\prod_{t=-n-L-k}^{-n-L}R(x_{t-1},x_{t})
%  \prod_{s=-n-L+1}^{n+L}R(x_{s-1},x_{s}) w(x_{n+L})\prod_{i \in
%      \mathcal{I}_0\cup \mathcal{I}_2} d\mu(x_i^{(a)})} \\
  &=&  \frac{ \int v(x_{-n-L-k})\prod_{t=-n-L-k+1}^{-n-L}R(x_{t-1},x_{t})
      \prod_{i \in \mathcal{I}_0}\mu(dx_i^{(a)})}{
  \int  v(x_{-n-L-k})\prod_{t=-n-L-k+1}^{-n-L}R(x_{t-1},x_{t})
\prod_{i \in  \mathcal{I}_0} \mu(dx_i^{(a)})}\\
    &&   \times \frac{ \int \prod_{s=-n-L+1}^{n+L}R(x_{s-1},x_{s}) w(x_{n+L})
      \prod_{i \in \mathcal{I}_1}\mu(dx_i^{(a)})}{
  \int  \prod_{s=-n-L+1}^{n+L}R(x_{s-1},x_{s}) w(x_{n+L})\prod_{i \in
   \mathcal{I}_2} \mu(dx_i^{(a)})}
\end{eqnarray*}
where $\mathcal{I}_0 = \{-n-L-k,\ldots,-n-L-1\}$, $\mathcal{I}_1 =
\{-n-L+1,\ldots,-n-1\}\cup\{n+1,\ldots,n+L\}$ and  $\mathcal{I}_2 =
\{-n-L+1,\ldots,-n-1,-n,\ldots,n+1,\ldots,n+L\}$.
Since the variable $x_{-n-L}^{(a)}$ is not integrated out, there is a
separation in the two terms and the 
numerator and denominator in the first term above cancels.
Thus the above reduces to only the second term
(with  $w_{n+L}(\cdot)$ in the numerator and denominator also being
cancelled) i.e.
\begin{eqnarray}
 && p(x_{-n}^{(a)},\ldots,x_{n}^{(a)}|{\bf x}^{V\backslash\{a\}}_{[-n-L-k,n+L]},x_{-n-L}^{(a)},x_{n+L}^{(a)})\nonumber\\
 &=& \frac{ \int \prod_{t=-n-L+1}^{n+L}R(x_{t-1},x_{t}) 
      \prod_{i \in \mathcal{I}_1}\mu(dx_i^{(a)})}{
  \int    \prod_{t=-n-L+1}^{n+L}R(x_{t-1},x_{t}) \prod_{i \in
     \mathcal{I}_2} \mu(dx_i^{(a)})} = \frac{ \int 
      \prod_{t=-n-L+1}^{n+L}R(x_{t-1}^{(a,\mathcal{N}_a)},x_{t}^{(a,\mathcal{N}_a)}) 
      \prod_{i \in \mathcal{I}_1}\mu(dx_i^{(a)})}{
  \int  
  \prod_{t=-n-L+1}^{n+L}R(x_{t-1}^{(a,\mathcal{N}_a)},x_{t}^{(a,\mathcal{N}_a)}) \prod_{i \in
     \mathcal{I}_2} \mu(dx_i^{(a)})} \nonumber\\
 % &=&
  %  \int p(x_{1}^{(a)},\ldots,x_{n}^{(a)}|{\bf x}^{\mathcal{N}_a}_{[-m:n+1]},x_{-m}^{(a)},x_{n+1}^{(a)})
  %    \prod_{i=-m+1}^{0}d\mu(x^{(a)}_{i}) \nonumber\\
  &=& p(x_{-n}^{(a)},\ldots,x_{n}^{(a)}|{\bf
    x}^{\mathcal{N}_a}_{[-n-L:n+L]},x_{-n-L}^{(a)},x_{n+L}^{(a)}).
    \label{eq:densityseparation}
\end{eqnarray}
By the same argument we can extend the right limit and for all $k>0$
we have
\begin{eqnarray*}
 &&
    p(x_{-n}^{(a)},\ldots,x_{n}^{(a)}|{\bf x}^{V\backslash\{a\}}_{[-n-L-k,n+L+k]},x_{-n-L}^{(a)},x_{n+L}^{(a)})\\
  &=& p(x_{-n}^{(a)},\ldots,x_{n}^{(a)}|{\bf
    x}^{\mathcal{N}_a}_{[-n-L:n+L]},x_{-n-L}^{(a)},x_{n+L}^{(a)}).
\end{eqnarray*}
Note, by partitioning ${\bf x}^{V\backslash\{a\}}_{[-n-L-k,n+L+k]}$
we can write the above as 
\begin{eqnarray*}
 &&     p(x_{-n}^{(a)},\ldots,x_{n}^{(a)}|{\bf x}^{\mathcal{N}_a}_{[-n-L-k,n+L+k]},
{\bf x}^{\mathcal{N}_{a}^{\prime}}_{[-n-L-k,n+L+k]},x_{-n-L}^{(a)},x_{n+L}^{(a)})\\
  &=& p(x_{-n}^{(a)},\ldots,x_{n}^{(a)}|{\bf
    x}^{\mathcal{N}_a}_{[-n-L:n+L]},x_{-n-L}^{(a)},x_{n+L}^{(a)}).
\end{eqnarray*}
Since the process is Markovian, replacing
$(x_{-n-L}^{(a)},x_{n+L}^{(a)})$ with the entire tail ${\bf
  x}_{(\infty,n+L)}^{(a)}=(x_{t}^{(a)};|t|\geq n+L)$
does not change the above distribution i.e.
\begin{eqnarray*}
  &&p(x_{-n}^{(a)},\ldots,x_{n}^{(a)}|{\bf x}^{\mathcal{N}_a}_{[-n-L-k,n+L+k]},
{\bf x}^{\mathcal{N}_{a}^{\prime}}_{[-n-L-k,n+L+k]},{\bf x}_{(\infty,n+L)}^{(a)})\\
  &=& p(x_{-n}^{(a)},\ldots,x_{n}^{(a)}|{\bf
    x}^{\mathcal{N}_a}_{[-n-L:n+L]},x_{-n-L}^{(a)},x_{n+L}^{(a)}).
\end{eqnarray*}
Further, we can replace ${\bf x}^{\mathcal{N}_a}_{[-n-L-k,n+L+k]}$ and
${\bf x}^{\mathcal{N}_{a}^{\prime}}_{[-n-L-k,n+L+k]}$ with
${\bf x}^{\mathcal{N}_a}_{[-n-L-k_1,n+L+k_1]}$ and
${\bf x}^{\mathcal{N}_{a}^{\prime}}_{[-n-L-k_2,n+L+k_2]}$ ($k_1,k_2>0$)
respectively and the conditional distribution does not change
\begin{eqnarray*}
 \label{eq:tvconditional0}   
  &&p(x_{-n}^{(a)},\ldots,x_{n}^{(a)}|{\bf x}^{\mathcal{N}_a}_{[-n-L-k_1,n+L+k_1]},
{\bf x}^{\mathcal{N}_{a}^{\prime}}_{[-n-L-k_2,n+L+k_2]},{\bf x}_{(\infty,n+L)}^{(a)})\nonumber\\
  &=& p(x_{-n}^{(a)},\ldots,x_{n}^{(a)}|{\bf
    x}^{\mathcal{N}_a}_{[-n-L:n+L]},x_{-n-L}^{(a)},x_{n+L}^{(a)})
      \quad\textrm{for all } n,L,k_1,k_2>0.
\end{eqnarray*}
Thus proving the result.
\end{proof}

For the remainder of the proof we focus on the
sigma-algebras associated with (\ref{eq:tvconditional0}). This requires
the following notation. 
\begin{itemize}
\item 
  $\mathcal{F}^{(a)}_{n} = \sigma(X_{-n}^{(a)},X_{-n+1}^{(a)},\ldots,X_{n-1}^{(a)},X_{n}^{(a)})$.
\item 
   $\mathcal{F}^{K}_{n} =
   \sigma(X_{-n}^{(b)},X_{-n+1}^{(b)},\ldots,X_{n-1}^{(b)},X_{n}^{(b)};b\in
   K)$.
 \item 
   $\mathcal{F}^{(a)}_{(\infty,n)} = \sigma(X_{t}^{(a)};|t|\geq n)$.
 \item We let $\mathcal{F}^{} = \sigma(X_{t}^{};t\in
   \mathbb{Z})$, 
   $\mathcal{F}^{(a)} = \sigma(X_{t}^{(a)};t\in \mathbb{Z})$
  and $\mathcal{F}^{K} = \sigma(X_{t}^{(b)};t\in \mathbb{Z},b\in K)$.
 \item  Let $\mathcal{G}_1\vee \mathcal{G}_2$ denote the smallest
sigma-algebra that contains both $\mathcal{G}_1$ and $\mathcal{G}_2$. 
\end{itemize}  

\vspace{2mm}

\noindent {\bf Proof of Theorem \ref{theorem:localMarkov}}
%  Suppose the interaction kernel $R$ is Hilbert-Schmidt, i.e.
%  $\|R(\cdot, \cdot)\| < \infty$ as in \eqref{eq:HS_defn}.
%Then 
%\begin{eqnarray*}
%\mathcal{F}_{}^{(a)}\independent \mathcal{F}^{\mathcal{N}_{a}^{\prime}}|
%  \mathcal{F}_{}^{\mathcal{N}_{a}}.
%\end{eqnarray*}  
%\end{lemma}  
%\begin{proof}
From (\ref{eq:conddef2}), Lemma \ref{lemma:finitelocal} implies 
\begin{align*}
\mathcal{F}_{n}^{(a)}\independent \mathcal{F}_{n+L+k_{2}}^{\mathcal{N}_{a}^{\prime}}|
  (\mathcal{F}_{n+L+k_{1}}^{\mathcal{N}_{a}}\vee
  \mathcal{F}_{(\infty,n+L)}^{(a)}) \textrm{ for all }n,k_1,k_2,L>0.
\end{align*}  
From Theorem \ref{theorem:mixing} we have
that the process $\{X_{t}\}_{t\in \mathbb{Z}}$ is $\beta$-mixing.  Thus 
by using the above and applying Theorem
\ref{theorem:cond-independenceSigma}
with $C =a$, $D = \mathcal{N}_{a}^{\prime}$ and $E=\mathcal{N}_{a}$ we have
the result.

\subsection{The Global Markov property}

From Definition 2.30 in \cite{b:lau-20}, two subsets $A$ and $B$ of a vertex set
$V$ of graph $\mathcal{G}= (V,E)$ are
said to be $g$-separated by $S$ if all paths from $A$ to $B$ are
blocked by $S$. If this is the case, then we
write $A\perp_{\mathcal{G}} B|S$.

Let $\alpha$ denote the
connectivity components in $\mathcal{G}_{V\backslash S}$ that contain $A$ and
$\beta = V\backslash\{\alpha\cup S\}$. 
Note by connectivity components, we mean 
the maximal number of nodes in 
$V\backslash S$ that contains $A$ and ever pair of nodes in that set has a path connecting them.
Clearly
$A\subseteq \alpha$, $B\subseteq \beta$ and $\alpha$
and $\beta$ are disjoint sets such that
$V=\alpha\cup\beta\cup S$.  Since $S$ separates $\alpha$ and $\beta$,
any clique of $\mathcal{G}$ is
either in $\alpha\cup S$ or $\beta\cup S$. This will be useful
in the lemma below. 

\begin{lemma}\label{lemma:separator}
Suppose the assumptions in Theorem  \ref{theorem:one_stationarity}
hold.  If $A\perp_{\mathcal{G}} B|S$, we define $\alpha$
as the connectivity components in $\mathcal{G}_{V\backslash S}$ that contain $A$ and
$\beta = V\backslash\{\alpha\cup S\}$.
Then for all $k_1,k_2,L,n>0$, 
the conditional density of
$X_{-n}^{\alpha},\ldots,X_{n}^{\alpha}$ given ${\bf X}^{S}_{[-n-L-k_1,n+L+k_1]},
{\bf X}^{\beta}_{[-n-L-k_2,n+L+k_2]},{\bf
  X}_{(\infty,n+L)}^{\alpha}$ is equal to the conditional density of
$X_{-n}^{\alpha},\ldots,X_{n}^{\alpha}$ given ${\bf X}^{S}_{[-n-L,n+L]},{\bf
  X}_{(\infty,n+L)}^{\alpha}$ i.e.
\begin{eqnarray}
 \label{eq:tvconditional2}   
  &&p(x_{-n}^{\alpha},\ldots,x_{n}^{\alpha}|{\bf x}^{S}_{[-n-L-k_1,n+L+k_1]},
{\bf x}^{\beta}_{[-n-L-k_2,n+L+k_2]},{\bf x}_{(\infty,n+L)}^{\alpha})\nonumber\\
  &=& p(x_{-n}^{\alpha},\ldots,x_{n}^{\alpha}|{\bf
    x}^{S}_{[-n-L:n+L]},x_{-n-L}^{\alpha},x_{n+L}^{\alpha})
      \quad\textrm{for all } n,L,k_1,k_2>0.
\end{eqnarray}
\end{lemma}
\begin{proof}
Following the proof of Proposition 2.40 in \cite{b:lau-20} and using
(\ref{eq:cli-ha}) we obtain the following
factorisation of $R(x_{t-1},x_{t})$ 
\begin{eqnarray*}
  \prod_{D\in \mathcal{C}}R_{D}(x_{t-1}^{(D)},x_{t}^{(D)})  =
  \prod_{D\in \mathcal{C}_{\alpha}}R_{D}(x_{t-1}^{(D)},x_{t}^{(D)})
  \prod_{E\in \mathcal{C}\backslash\mathcal{C}_{\alpha}}R_{E}(x_{t-1}^{(E)},x_{t}^{(E)})
\end{eqnarray*}
where $\mathcal{C}_{\alpha}$ denotes all cliques contained in
$\alpha\cup S$. Thus the joint density of
$X_{-M},X_{-M+1},\ldots,X_{M}$ has the factorisation
\begin{eqnarray}
\label{eq:pnpd2}  
&&p_{[-M:M]}(x_{-M},\ldots,x_{M}) \nonumber\\
  &=& \frac{1}{r^{2M}} \,   v(x_{-M})
      \left[\prod_{t=-M+1}^{M}
\prod_{D\in \mathcal{C}_{\alpha}}R_{D}(x_{t-1}^{(D)},x_{t}^{(D)})
  \prod_{E\in
      \mathcal{C}\backslash\mathcal{C}_{\alpha}}R_{E}(x_{t-
      1}^{(E)},x_{t}^{(E)})\right]w(x_{M}).
\end{eqnarray}
Using (\ref{eq:pnpd2}), the conditional distribution of
$X_{-n+1}^{\alpha},\ldots,X_{n-1}^{\alpha}$ given
${\bf X}^{S}_{[-n:n]},{\bf X}_{[-n,n]}^{\beta},X_{-n}^{\alpha},X_{n}^{\alpha}$ is
\begin{align}
\label{eq:pnpdcond}  
 & p(x_{-n}^{\alpha},\ldots,x_{n}^{\alpha}|{\bf
     x}^{S}_{[-n-1:n+1]},{\bf x}_{[-n-1,n+1]}^{\beta},x_{-n-1}^{\alpha},x_{n+1}^{\alpha}) \nonumber\\
  &=
  \frac{\prod_{t=-n}^{n+1}\prod_{D\in \mathcal{C}_{\alpha}}R_{D}(x_{t-1}^{(D)},x_{t}^{(D)})}{\int
  \prod_{t=-n}^{n+1}\prod_{D\in \mathcal{C}_{\alpha}}R_{D}(x_{t-1}^{(D)},x_{t}^{(D)})
  \prod_{i=-n}^{n}\mu(dx^{\alpha}_{i})}.
\end{align}
We extend the conditioning set and 
follow the same proof as in 
equation (\ref{eq:densityseparation}), in
Lemma \ref{lemma:finitelocal} to give 
\begin{eqnarray*}
 && p(x_{-n}^{\alpha},\ldots,x_{n}^{\alpha}|{\bf
     x}^{S}_{[-n-L-k:n+L+k]},{\bf x}_{[-n-L-k,n+L+k]}^{\beta},x_{-n-L}^{\alpha},x_{n+L}^{\alpha})\\
 &=& \frac{ \int 
      \prod_{t=-n-L}^{n+L}\prod_{D\in \mathcal{C}_{\alpha}}R_{D}(x_{t-1}^{(D)},x_{t}^{(D)})
      \prod_{i \in \mathcal{I}_1}\mu(dx_i^{\alpha})}{
  \int  
  \prod_{t=-n-L+1}^{n+L}\prod_{D\in \mathcal{C}_{\alpha}}R_{D}(x_{t-1}^{(D)},x_{t}^{(D)}) \prod_{i \in
     \mathcal{I}_2} \mu(dx_i^{\alpha})} \\
  &=& p(x_{-n}^{\alpha},\ldots,x_{n}^{\alpha}|{\bf
    x}^{S}_{[-n-L:n+L]},x_{-n-L}^{\alpha},x_{n+L}^{\alpha}),
\end{eqnarray*}
where $\mathcal{I}_1 =
\{-n-L+1,\ldots,-n-1\}\cup\{n+1,\ldots,n+L\}$ and  $\mathcal{I}_2 =
\{-n-L+1,\ldots,-n-1,-n,\ldots,n+1,\ldots,n+L\}$. As in the local
Markov case since the process is Markovian, replacing
$(X_{-n-L}^{\alpha},X_{n+L}^{\alpha})$ with the entire tail
${\bf X}_{(\infty,n+L)}^{\alpha} = (X_{t}^{a};|t|\geq n+L,a\in
\alpha)$ does not change the distribution. Similarly replacing
${\bf X}^{S}_{[-n-L-k:n+L+k]}$ and ${\bf X}_{[-n-L-k,n+L+k]}^{\beta}$
with ${\bf X}^{S}_{[-n-L-k_1:n+L+k_1]}$ and ${\bf X}_{[-n-L-k_2,n+L+k_2]}^{\beta}$
respectively does not change the distribution i.e.
\begin{eqnarray*}
 && p(x_{-n}^{\alpha},\ldots,x_{n}^{\alpha}|{\bf
     x}^{S}_{[-n-L-k_1:n+L+k_1]},{\bf x}_{[-n-L-k_2,n+L+k_2]}^{\beta},{\bf x}_{(\infty,n+L)}^{\alpha})\\
  &=& p(x_{-n}^{\alpha},\ldots,x_{n}^{\alpha}|{\bf
    x}^{S}_{[-n-L:n+L]},x_{-n-L}^{\alpha},x_{n+L}^{\alpha}).
\end{eqnarray*}
As the above holds for all $k_1,k_2,n,L>0$ we obtain the result. 
\end{proof}

\vspace{2mm}

%\begin{lemma}[Global Markov property]
%  Suppose the interaction kernel $R$ is Hilbert-Schmidt, i.e.
%  $\|R(\cdot, \cdot)\| < \infty$ as in \eqref{eq:HS_defn}. If
%  $A\perp_{\mathcal{G}} B|S$, then 
%\begin{eqnarray*}
%\mathcal{F}_{}^{A}\independent \mathcal{F}^{B}|
%  \mathcal{F}_{}^{S}.
%\end{eqnarray*}  
%\end{lemma}  
%\begin{proof}
\noindent {\bf Proof of Theorem \ref{theorem:globalMarkov}}
From (\ref{eq:conddef2}), Lemma \ref{eq:tvconditional2} implies 
\begin{align*}
\mathcal{F}_{n}^{\alpha}\independent \mathcal{F}_{n+L+k_{2}}^{\beta}|
  (\mathcal{F}_{n+L+k_{1}}^{S}\vee
  \mathcal{F}_{(\infty,n+L)}^{\alpha}) \textrm{ for all }n,k_1,k_2,L>0.
\end{align*}
From Theorem \ref{theorem:mixing} we have
that the process $\{X_{t}\}_{t\in \mathbb{Z}}$ is $\beta$-mixing.
Thus by using the above and applying Theorem
\ref{theorem:cond-independenceSigma}
with $C =\alpha$, $D = \beta$ and $E=S$ we have
\begin{align*}
\mathcal{F}_{}^{\alpha}\independent \mathcal{F}^{\beta}|
  \mathcal{F}^{S}.
\end{align*}
Since $A\subseteq \alpha$ and $B\subseteq \beta$ then the above
immediately implies that
\begin{align*}
\mathcal{F}_{}^{A}\independent \mathcal{F}^{B}|
  \mathcal{F}^{S},
\end{align*}
which gives the result. 
%\end{proof}  

\subsection{Conditional independence between
  stochastic processes}\label{sec:sigmalimit}

\paragraph{Background results}
Let $(\Omega,\mathcal{F},P)$ denote a probability space. 
We first recall some results which will be useful in the remainder of
the proof. Suppose that $\mathcal{G}$ is a subsigma-algebra of $\mathcal{F}$.
We define the set of functions
\begin{eqnarray*}
  [\mathcal{G}]^{+}  = \{ \textrm{all positive, bounded }
  \mathcal{G}\textrm{-measurable  functions}\}.
\end{eqnarray*} 
For example, if $A\in \mathcal{G}$, then $\mathbbm{1}_{A}\in [\mathcal{G}]^{+}$.

\emph{Conditional expectation}
Suppose that $\mathcal{G}\subseteq\mathcal{F}$. 
We recall that $Y$ is the conditional expectation of $X$ given the
sigma-algebra $\mathcal{G}$ (written $Y=\Ex[X|\mathcal{G}]$) if
$Y$ is $\mathcal{G}$-measurable and for all $m\in [\mathcal{G}]^{+}$
we have
\begin{eqnarray*}
  \int_{}mXdP = \int mYdP, 
\end{eqnarray*}  
i.e. $\Ex[mX] = \Ex[\Ex[mX|\mathcal{G}]]$. We make frequent use of this definition.

\emph{Doob's Martingale convergence theorem and the reverse
  Martingale convergence theorem} 
\begin{itemize}
\item[(i)]Suppose $\{\mathcal{G}_{h}\}_{h}$ are sub-sigma algebras of
$\mathcal{A}$ where $\mathcal{G}_{h_1}\subseteq \mathcal{G}_{h_2}$ for
$h_1\leq h_2$ and $\mathcal{G}_{\infty} =
\vee_{h=1}^{\infty}\mathcal{G}_{h}$. Suppose that
$\{U_{h}\}_{h=1}^{\infty}$ is a martingale sequence with respect to
the filtration $\{\mathcal{G}_{h}\}$
where $\sup_{h}\Ex|U_{h}|<\infty$, then (by the Martingale
convergence theorem) the pointwise limit
$U=\lim_{h\rightarrow \infty} U_{h}$ exists almost surely with $U\in
\mathcal{G}_{\infty}$ and $\Ex|U|<\infty$.
Example: Suppose $m\in [\mathcal{G}]^{+}$,
let $m_{h} = \Ex[m|\mathcal{G}_{h}]$, then $\{m_{h}\}_{h=1}^{\infty}$ forms a
martingale sequence. By the martingale convergence theorem
$m_{h}\rightarrow m$ almost surely.
\item[(ii)] For the reverse martingale
convergence theorem we again assume that $\mathcal{G}_{h_{1}}\subseteq
\mathcal{G}_{h_{2}}$ where $h_{1}\leq h_{2}$, but this time focus on
the case $h_{1}$ and $h_{2}$ are negative and consider the limit
$h\rightarrow -\infty$. Let $\mathcal{G}_{-\infty} =
\cap_{h=-\infty}^{1}\mathcal{G}_{h}$ and $\{V_{-h}\}_{h=1}^{\infty}$
be a reverse martingale sequence with respect to the reverse
filtration $\{\mathcal{G}_{-h}\}_{h=1}^{\infty}$, then $V =
\lim_{h\rightarrow \infty} V_{-h}$ exists almost surely with
$\Ex[V]<\infty$. 
\end{itemize}

\emph{Trivial sigma-algebras}
A trivial sigma-algebra is defined as a sigma algebra whose sets have
probability of either one or zero with respect to the probability measure
$\mu$. For further reference we denote the trivial sigma-algebra as
$\mathcal{T}$. A complete trivial sigma-algebra  is defined as
$\overline{\mathcal{T}} = \{A\in \mathcal{F}: \mu(A)^{2}=\mu(A)\}$. 
We use the standard notation
$\overline{\mathcal{G}} = \mathcal{G}\vee\overline{\mathcal{T}}$ (this is often
referred to as the completion of the sub sigma-algebra
$\mathcal{G}$). Sections 0.3.2 and 2.2.3 in
\cite{b:flo-mou-rol-90} summarizes some of the important properties of
$\overline{\mathcal{G}}$; the results relevant to the proof are given below.

Equation (2.2.3) in \cite{b:flo-mou-rol-90} states that
for all $m\in [\mathcal{F}]^{+}$ 
\begin{eqnarray}
\label{eq:completionexpectation}
  \Ex\left(m|\mathcal{F}\right) = \Ex\left(m|\overline{\mathcal{F}}\right).  
\end{eqnarray}
%Lemma 2.2.5 in \cite{b:flo-mou-rol-90}  states for any %sub-sigma algebras
%$\mathcal{F}_1$ and $\mathcal{F}_2$ we have 
%\begin{eqnarray}
%\label{eq:completionM}
%\overline{\mathcal{F}_1\vee\mathcal{F}_2} =
%  \overline{\mathcal{F}}_1\vee \overline{\mathcal{F}}_2.
%\end{eqnarray}

\emph{Conditional independence} We use the definition of conditional
independence given in \cite{b:lau-20} (equations (2.3) and
(2.5)) and  Section 2.2.2, \cite{b:flo-mou-rol-90}. The random variables
$X_{1}$ and $X_{2}$ given $X_{3}$ is conditional independent if for
any $m_{1}\in [\sigma(X_{1})]^{+}$ and $m_{2}\in [\sigma(X_{2})]^{+}$
we have 
\begin{eqnarray}
  \label{eq:conddef1}
 \Ex\left[m_{1}m_{2}|\sigma(X_{3})\right]  =
  \Ex\left[m_{1}|\sigma(X_{3})\right]
  \Ex\left[m_{2}|\sigma(X_{3})\right] \textrm{ almost surely}.
\end{eqnarray}
If the above holds we often write $X_{1}\independent X_{2}|X_{3}$ or
$\sigma(X_{1})\independent \sigma (X_{2})|\sigma(X_{3})$. 
An equivalent definition to the above (see Theorem 2.2.1 in
\cite{b:flo-mou-rol-90}) is that for all $m_{1}\in
[\sigma(X_{1})]^{+}$
\begin{eqnarray}
  \label{eq:conddef2}
 \Ex\left[m_{1}|\sigma(X_{2})\vee \sigma(X_{3})\right]  =
  \Ex\left[m_{1}|\sigma(X_{3})\right] \textrm{ almost surely}.
\end{eqnarray}
If the conditional density of $X_{1}|X_{2},X_{3}$ is equivalent to the
conditional density of $X_{1}|X_{2}$, then by (\ref{eq:conddef2}) we
have $\sigma(X_{1})\independent \sigma (X_{2})|\sigma(X_{3})$.

\begin{theorem}\label{theorem:cond-independenceSigma}
Suppose that $\{X_{t}\}_{t\in \mathbb{Z}}$ is a $\beta$-mixing
stationary time series, with $X_{t} = (X_{t}^{(1)},\ldots,X_{t}^{(p)})^{\top}$.   
Suppose that $C,D$ and $E$ are disjoint sets of $[p]$ where
\begin{align*}
\mathcal{F}_{n}^{C}\independent \mathcal{F}_{n+L+k_{2}}^{D}|
  (\mathcal{F}_{n+L+k_{1}}^{E}\vee
  \mathcal{F}_{(\infty,n+L)}^{C}) \textrm{ for all }n,k_1,k_2,L>0.
\end{align*}  
Then
\begin{align*}
\mathcal{F}^{C}\independent \mathcal{F}^{D}|
  \mathcal{F}^{E}
\end{align*}  
\end{theorem}
\begin{proof} We prove the results in four steps, where we sequentially
  let $k_2,L,k_1$ and $n\rightarrow \infty$.
   
 \noindent {\bf Step 1} Since
\begin{eqnarray}
\label{eq:Fnstep1}  
  \mathcal{F}_{n}^{(a)}\independent \mathcal{F}_{n+L+k_{2}}^{D}|
  (\mathcal{F}_{n+L+k_{1}}^{E}\vee \mathcal{F}_{(\infty,n+L)}^{C})
\end{eqnarray}
holds for all $k_2>0$ we will show that
\begin{eqnarray}
  \label{eq:Fnstep1proof} 
  \mathcal{F}_{n}^{C}\independent \mathcal{F}^{D}|
  (\mathcal{F}_{n+L+k_{1}}^{E}\vee \mathcal{F}_{(\infty,n+L)}^{C}).
\end{eqnarray}

Let $\mathcal{G}_{k_1} = \mathcal{F}_{n+L+k_{1}}^{E}\vee
  \mathcal{F}_{(\infty,n+L)}^{C}$. Then 
(\ref{eq:Fnstep1}) implies that
for all $k_2>0$, $m_{n+L+k_2}^{D}\in [\mathcal{F}^{D}_{n+L+k_2}]^{+}$,
$m_{n}^{C}\in  [\mathcal{F}^{C}_{n}]^{+}$ we have
\begin{eqnarray*}
\Ex[m_{n}^{C}m_{n+L+k_2}^{D}|\mathcal{G}_{k_1}] =
\Ex[m_{n}^{C}|\mathcal{G}_{k_1}]\,\Ex[m_{n+L+k_2}^{D}|\mathcal{G}_{k_1}]  
\end{eqnarray*}  
almost surely.
%Further, by definition of conditional expectations
%for all $m_{k_1}^{E}\in [\mathcal{G}_{k_1}]^{+}$
%\begin{eqnarray*}
%\Ex(m_{k_1}^{E}\Ex[m_{n}^{C}m_{n+L+k_2}^{D}|\mathcal{G}_{k_1}]) =
%  \Ex(m_{k_1}^{E}\Ex[m_{n}^{C}|\mathcal{G}_{k_1}]\,
%  \Ex[m_{n+L+k_2}^{D}|\mathcal{G}_{k_1}]). 
%\end{eqnarray*}
To prove (\ref{eq:Fnstep1proof})  for every $m^{D}\in
[\mathcal{F}^{D}]^{+}$ we define a sequence of 
functions $m_{j}^{D} = \Ex[m^{D}|\mathcal{F}^{D}_{j}]\in
[\mathcal{F}^{D}_{j}]^{+}$. By the martingale
convergence theorem we have $m_{j}^{D}\rightarrow
m ^{D}$ almost surely (and in $L_{p}$ for all $p\geq
1$). Since $\{m_{j}^{D}\}_{j}$ are bounded functions, by dominated
convergence we can take the limit inside the expectation to give 
\begin{align*}
\lim_{k_{2}\rightarrow
  \infty}\Ex[m_{n}^{C}m_{n+L+k_2}^{D}|\mathcal{G}_{k_1}])
  &=\lim_{k_{2}\rightarrow \infty}
  \Ex[m_{n}^{C}|\mathcal{G}_{k_1}]\,
  \lim_{k_{2}\rightarrow
  \infty}\Ex[m_{n+L+k_2}^{D}|\mathcal{G}_{k_1}]) \\
  \Rightarrow 
       \Ex[m_{n}^{C}m^{D}|\mathcal{G}_{k_1}]) &=
  \Ex[m_{n}^{C}|\mathcal{G}_{k_1}]\,
\Ex[m^{D}|\mathcal{G}_{k_1}])
\end{align*}
almost surely (since the number of null sets over $k_2$ are countable).
As the above holds for all functions $m ^{D}\in
[\mathcal{F}^{D}]^{+}$ by (\ref{eq:conddef1}) we have
\begin{eqnarray*}
  \mathcal{F}_{n}^{C}\independent \mathcal{F}^{D}|
  (\mathcal{F}_{n+L+k_{1}}^{E}\vee \mathcal{F}_{(\infty,n+L)}^{C}).
\end{eqnarray*}
Thus proving (\ref{eq:Fnstep1proof}). 

\noindent {\bf Step 2} Since   
\begin{eqnarray}
\label{eq:Fnstep2}  
  \mathcal{F}_{n}^{C}\independent \mathcal{F}_{}^{D}|
  (\mathcal{F}_{n+L+k_{1}}^{E}\vee \mathcal{F}_{(\infty,n+L)}^{C})
\end{eqnarray}
holds for all $k_1>0$ we will show that
\begin{eqnarray*}
  \mathcal{F}_{n}^{C}\independent \mathcal{F}^{D}|
  (\mathcal{F}_{}^{E}\vee \mathcal{F}_{(\infty,n+L)}^{C}).
\end{eqnarray*}

(\ref{eq:Fnstep2}) implies that
for any $m^{C}_{n}\in [\mathcal{F}_{n}^{C}]^{+}$ and $m^{D}\in
[\mathcal{F}_{}^{D}]^{+}$ we have
\begin{eqnarray}
\label{eq:mCD}  
\Ex[m_{n}^{C}m^{D}|\mathcal{F}_{n+L+k_{1}}^{E}\vee
  \mathcal{F}_{(\infty,n+L)}^{C}] &=&
  \Ex[m_{n}^{C}|\mathcal{F}_{n+L+k_{1}}^{E}\vee
  \mathcal{F}_{(\infty,n+L)}^{C}]\,\Ex[m^{D}|\mathcal{F}_{n+L+k_{1}}^{E}\vee
  \mathcal{F}_{(\infty,n+L)}^{C}] \nonumber\\
  && 
\end{eqnarray}
almost surely. $\{\mathcal{G}_{k_1} = \mathcal{F}_{k_1}^{E}\vee
\mathcal{F}_{(\infty,n+L)}^{C}\}_{k_1}$ forms a filtration of sigma-algebras
where $\mathcal{G}_{k_1}\subseteq \mathcal{G}_{k_2}$ for all $k_1<k_2$.
Let
$$\mathcal{G} = \vee_{k_{1}=1}^{\infty}\mathcal{G}_{k_1} =
\vee_{k_1=1}^{\infty}(\mathcal{F}_{k_1}^{E}\vee
\mathcal{F}_{(\infty,n+L)}^{C})=
\mathcal{F}_{}^{E}\vee\mathcal{F}_{(\infty,n+L)}^{C}.$$

For every $m_{n}^{C}\in [\mathcal{F}_{n}^{C}]^{+}$ and
$m^{D} = \mathcal{F}^{D}$
we define 
$U_{k_1} = \Ex[m_{}^{C}m^{D}|\mathcal{G}_{k_{1}}]$,
$V_{k_1}=\Ex[m_{}^{C}|\mathcal{G}_{k_{1}}]$
and $W_{k_1}=\Ex[m_{}^{D}|\mathcal{G}_{k_{1}}]$. By the martingale convergence
  theorem we have $\lim_{k_1\rightarrow \infty}U_{k_1}= U$,
  $\lim_{k_1\rightarrow \infty}V_{k_1}= V$ and
  $\lim_{k_1\rightarrow \infty}W_{k_1}=
  W$ almost surely where $U,V,W\in \mathcal{G}$. Thus by letting $k_1\rightarrow \infty$ in
  (\ref{eq:mCD}) we have 
\begin{eqnarray*}
\Ex[m_{n}^{C}m^{D}|\mathcal{F}^{E}\vee
  \mathcal{F}_{(\infty,n+L)}^{C}] =
  \Ex[m^{C}_{n}|\mathcal{F}^{E}\vee
  \mathcal{F}_{(\infty,n+L)}^{C}]\,\Ex[m^{D}|\mathcal{F}^{E}\vee
  \mathcal{F}_{(\infty,n+L)}^{C}]
\end{eqnarray*}
almost surely. As the above holds for all $m_{n}^{C}\in [\mathcal{F}_{n}^{C}]^{+}$ and
$m^{D} = [\mathcal{F}^{D}]^{+}$
\begin{eqnarray*}
  \mathcal{F}_{n}^{C}\independent \mathcal{F}^{D}|
  (\mathcal{F}_{}^{E}\vee \mathcal{F}_{(\infty,n+L)}^{C}).
\end{eqnarray*}

\noindent {\bf Step 3} 
Next we show that since 
\begin{eqnarray}
\label{eq:step3}  
  \mathcal{F}_{n}^{C}\independent \mathcal{F}^{D}|
  (\mathcal{F}_{}^{E}\vee \mathcal{F}_{(\infty,n+L)}^{C})
\end{eqnarray}
holds for all $L$ we have
\begin{eqnarray*}
  \mathcal{F}_{n}^{C}\independent \mathcal{F}^{D}|
  \mathcal{F}_{}^{E}.
\end{eqnarray*}

 Let
$\mathcal{G}_{-L} = \mathcal{F}_{}^{E}\vee
\mathcal{F}_{(\infty,n+L)}^{C}$, where we recall
$\mathcal{F}^{C}_{(\infty,n+L)}= \sigma(X_{t}^{C};|t|\geq
n+L)$. 
Then for $0<L_{1}<L_{2}$  we have $\mathcal{G}_{-L_{2}}\subset
\mathcal{G}_{-L_{1}}$. Hence
$$\mathcal{G}_{-\infty} = \cap_{L=1}^{\infty}\mathcal{G}_{L} =
\cap_{L=1}^{\infty}(\mathcal{F}_{}^{E}\vee
\mathcal{F}_{(\infty,n+L)}^{C}) = \mathcal{F}_{}^{E}\vee
\mathcal{F}_{-\infty}^{C}.$$
where $\mathcal{F}_{-\infty}^{C}=\cap_{L=1}^{\infty}\mathcal{F}_{(\infty,n+L)}$.

To prove this result we apply the reverse martingale theorem.
For every $m_{n}^{C}\in [\mathcal{F}_{n}^{C}]^{+}$ and
$m^{D} = [\mathcal{F}^{D}]^{+}$
we define 
$U_{-L} = \Ex[m_{}^{C}m^{E^c}|\mathcal{G}_{-L}]$,
$V_{-L}=\Ex[m_{}^{C}|\mathcal{G}_{-L}]$
and $W_{-L}=\Ex[m_{}^{D}|\mathcal{G}_{-L}]$.
By the reverse martingale convergence
  theorem we have $\lim_{k_1\rightarrow\infty}U_{-k_1}= U_{-\infty}$,
  $\lim_{k_1\rightarrow \infty}V_{-k_1}= V_{-\infty}$ and
  $\lim_{k_1\rightarrow \infty}W_{-k_1}=
  W_{-\infty}$ almost surely where $U_{-\infty},V_{-\infty},W_{-\infty}\in \mathcal{G}_{-\infty}$.
  This gives
\begin{eqnarray*}
\Ex[m_{n}^{C}m^{D}|\mathcal{F}^{E}\vee
  \mathcal{F}_{-\infty}^{C}] =
  \Ex[m^{C}_{n}|\mathcal{F}^{E}\vee
  \mathcal{F}_{-\infty}^{C}]\,\Ex[m^{D}|\mathcal{F}^{E}\vee
  \mathcal{F}_{-\infty}^{C}].
\end{eqnarray*}   
This implies 
\begin{eqnarray}
  \label{eq:step3-part2}
  \mathcal{F}_{n}^{C}\independent \mathcal{F}^{D}|
  (\mathcal{F}_{}^{E}\vee \mathcal{F}_{-\infty}^{C}).
\end{eqnarray}
Under the assumptions of the theorem,
the process $\{X_{t}\}_{t\in \mathbb{Z}}$ is $\beta$-mixing. 
By  \cite{p:bra-06}, Section 2.5 point (e), since the time series
$\{X_{t}^{C}\}_{t}$ is $\beta$-mixing, then the double-tail sigma-algebra 
$\mathcal{F}_{-\infty}^{C}$ is trivial
i.e. $\mathcal{F}_{-\infty}^{C}=\mathcal{T}$. We now apply this result. 

Since $\mathcal{F}_{-\infty}^{C}=\mathcal{T}$, we have 
$\mathcal{F}_{}^{E}\vee \mathcal{F}_{-\infty}^{C}=
\mathcal{F}_{}^{E}\vee \mathcal{T}$. Under conditional independence, for all
 $m^{C}\in [\mathcal{F}_{n}^{C}]^{+}$
and $m^{D}\in
[\mathcal{F}^{D}]^{+}$ we have
\begin{align}
\label{eq:mCmD}  
  \Ex[m^{C}m^{D}|\mathcal{F}^{E}\vee\mathcal{T}]
  &= \Ex[m^{C}|\mathcal{F}^{E}\vee\mathcal{T}]
    \Ex[m^{D}|\mathcal{F}_{}^{E}\vee\mathcal{T}].
 \end{align}   
Now
$ \mathcal{F}_{}^{E}\subseteq\mathcal{F}_{}^{E}\vee
\mathcal{T} \subseteq \mathcal{F}_{}^{E}\vee \overline{\mathcal{T}}=
\overline{\mathcal{F}}^{E}$, and by
 (\ref{eq:completionexpectation}) for any $m\in [\mathcal{F}]^{+}$ we
 have
 $\Ex[m|\mathcal{F}_{}^{E}]=\Ex[m|\overline{\mathcal{F}}_{}^{E}]$. Thus
$\Ex[m|\mathcal{F}_{}^{E}\vee
\mathcal{T}]=\Ex[m|\mathcal{F}_{}^{E}]$. Hence 
from (\ref{eq:mCmD}) and the above 
for any $m^{C}\in [\mathcal{F}_{n}^{C}]^{+}$
and $m^{D}\in
[\mathcal{F}^{D}]^{+}$ we have
\begin{align*}
  \Ex[m^{C}m^{D}|\mathcal{F}_{}^{E}]  &=
  \Ex[m^{C}|\mathcal{F}_{}^{E}]
  \Ex[m^{D}|\mathcal{F}_{}^{E}] 
\end{align*}  
almost surely. This gives
\begin{eqnarray}
  \label{eq:step3-part2}
  \mathcal{F}_{n}^{C}\independent \mathcal{F}^{D}|
  \mathcal{F}_{}^{E}.
\end{eqnarray}

\noindent {\bf Step 4} The final step is to show that 
since 
  $\mathcal{F}_{n}^{C}\independent \mathcal{F}^{D}|
  \mathcal{F}_{}^{E}$
holds for all $n$ then
 $\mathcal{F}_{}^{C}\independent \mathcal{F}^{D}|
  \mathcal{F}_{}^{E}$.
The proof of this result mirrors the proof of Step 1, hence we omit
the details. 
\end{proof}

\section{Proof of results in Section \ref{sec:prob_cestgm}}\label{sec:power}

\subsection{Proof of Theorem \ref{theorem:mixing}}\label{sec:pf_mixing}
We first establish \eqref{eq:mixAB}. Fix partitions $\{A_i\}_{i=1}^I$ and $\{B_j\}_{j=1}^J$ of $\m X$ such that $A_i \in \sigma(X_0)$ for each $i \in [I]$, and $B_j \in \sigma(X_{n+1})$ for each $j \in [J]$. Then, 
\begin{align*}
& \sum_{i=1}^I \sum_{j=1}^J |P(A_i \cap B_j) - P(A_i)P(B_j)| \\
 = & \sum_{i=1}^I \sum_{j=1}^J \left|\int_{A_i} \int_{B_j} [p_{0, n+1}(x_0, x_{n+1}) - p_1(x_0) \, p_1(x_{n+1})] \, \mu(dx_0)\mu(dx_{n+1})\right|  \\
\le & \sum_{i=1}^I \sum_{j=1}^J \int_{A_i} \int_{B_j} |p_{0, n+1}(x_0, x_{n+1}) - p_1(x_0) \, p_1(x_{n+1})| \, \mu(dx_0)\mu(dx_{n+1}) \\ 
= & \int_{\m X} \int_{\m X} |p_{0, n+1}(x_0, x_{n+1}) - p_1(x_0) \, p_1(x_{n+1})| \, \mu(dx_0)\mu(dx_{n+1}),
\end{align*}
where going from the second to the third line, we use Jensen's inequality, and in the next step use that $\{A_i \times B_j\}_{(i \in [I], j \in [J])}$ form a partition of $\m X \times \m X$. Since the right hand side is independent of the partition sets, a supremum over such partition pairs delivers \eqref{eq:mixAB}. 

To complete the rest of the proof, we continue from \eqref{eq:mixing_rem}. From Lemma \ref{lemma:pow_it_main}, we have that $\Delta_n^*(R_{x_{n+1}}) = (T^* Q_{w, v})^n$. Substituting in \eqref{eq:mixing_rem}, we get 
\begin{align*}
p_{0,n+1}(x_0,x_{n+1}) -  p_1(x_0) p_1(x_{n+1}) = \frac{1}{r^{n+1}} v(x_{0}) w(x_{n+1}) \, (T^{*} Q_{w,v})^{n}(R_{x_{n+1}})[x_{0}]. 
\end{align*}
%Fix $A \in\sigma(X_0), B\in \sigma(X_{n+1})$. 
Substituting the equality from the above display in \eqref{eq:mixAB}, obtain
\begin{align*}
2 \beta\left(\sigma(X_0), \sigma(X_{n+1}) \right) \le \frac{1}{r^{n+1}} \int_{\m X}\int_{\m X} \left| v(x_{0}) w(x_{n+1}) \, (T^{*} Q_{w,v})^{n}(R_{x_{n+1}})[x_{0}] \right| \, \mu(dx_0)\mu(dx_{n+1}). 
\end{align*}
Using Cauchy--Schwarz inequality, bound 
\begin{align*}
 \int_{\m X}\int_{\m X} \left| v(x_{0}) w(x_{n+1}) \, (T^{*} Q_{w,v})^{n}(R_{x_{n+1}})[x_{0}] \right| \, \mu(dx_0)\mu(dx_{n+1})   \le L_1 \, L_2,
% & \le \left(\int_A \int_B v^2(x_0) w^2(x_{n+1}) \mu(dx_0)\mu(dx_{n+1}) \right)^{1/2} \, \left(\int_A \int_B \left\{(T^{*} Q_{w,v})^{n}(R_{x_{n+1}})[x_{0}]\right\}^2 \mu(dx_0)\mu(dx_{n+1}) \right)^{1/2}. 
\end{align*}
where 
\begin{align*}
L_1^2 :\,= \int_{\m X}\int_{\m X} v^2(x_0) w^2(x_{n+1}) \mu(dx_0)\mu(dx_{n+1}) \le \|v\|^2 \|w\|^2 = \langle v, v^*\rangle^{-2}, 
\end{align*}
where we used $\|v\| = 1$ and $w = v^*/\langle v, v^*\rangle$ with $\|w\| = 1/\langle v, v^*\rangle$, 
and 
%\textcolor{red}{(do we need to clarify below that $(T^{*} Q_{w,v})^{n}(R_{x_{n+1}})[x_{0}]$ is real because we are now applying the CS with the modulus inside?)}
\begin{align*}
L_2^2 &:\,= \int_{\m X} \int_{\m X} \left|(T^{*} Q_{w,v})^{n}(R_{x_{n+1}})[x_{0}]\right|^2 \mu(dx_0)\mu(dx_{n+1})  \\
& =  
\int_{\m X} \left[\int_{\m X} \left\{(T^{*} Q_{w,v})^{n}(R_{x_{n+1}})[x_{0}]\right\}^2 \mu(dx_0) \right] \mu(dx_{n+1})  \\
& = \int_{\m X} \|(T^{*} Q_{w,v})^{n}(R_{x_{n+1}})\|^2 \mu(dx_{n+1}) \\
& \le \|(T^{*} Q_{w,v})^{n}\|_{\rm op}^2 \, \int_{\m X} \|R_{x_{n+1}}\|^2 \mu(dx_{n+1}) = \|(T^{*} Q_{w,v})^{n}\|_{\rm op}^2 \, \|R(\cdot, \cdot)\|^2. 
\end{align*}
%{\color{blue}I agree and like.}
Here, going from the first to the second line, we used that 
%we used the positivity of the integrand to bound the integral over $A \times B$ by the integral over $\m X \times \m X$, and then 
$(T^{*} Q_{w,v})^{n}(R_{x_{n+1}})[x_{0}]$ is real to drop the modulus, and 
then wrote the joint integral as an iterated integral. In the last step, we used $\int_{\m X} \|R_{x_{n+1}}\|^2 \mu(dx_{n+1}) = \int_{\m X} \int_{\m X} R(x, x_{n+1})^2 \mu(dx) \mu(dx_{n+1}) = \|R(\cdot, \cdot)\|^2$ from \eqref{eq:HS_defn}. By assumption, $\|R(\cdot, \cdot)\|$ is finite. Combining the inequalities, we obtain 
\begin{align*}
2 \beta\left(\sigma(X_0), \sigma(X_{n+1}) \right)
& \le r^{-(n+1)} \, \langle v, v^*\rangle^{-1} \, \|(T^{*} Q_{w,v})^{n}\|_{\rm op} \, \|R(\cdot, \cdot)\|. 
\end{align*}
We now use Lemma \ref{lemma:pow_it_main} to bound $\|(T^{*} Q_{w,v})^{n}\|_{\rm op}$ by $C_\varepsilon (|\lambda_2^*| + \varepsilon)^n$, where $\varepsilon > 0$ is chosen such that $|\lambda_2^*| + \varepsilon < r$. This gives the overall bound 
\begin{align*}
\beta\left(\sigma(X_0), \sigma(X_{n+1}) \right)
& \le 0.5 \, r^{-1} \, \langle v, v^*\rangle^{-1} \, \|R(\cdot, \cdot)\| \, C_\varepsilon \, \left(\frac{|\lambda_2^*| + \varepsilon}{r}\right)^n,
\end{align*}
establishing geometric $\beta$-mixing. It is known that a geometrically $\beta$-mixing Markov process is ergodic (see Section 2.5 in \citep{p:bra-06}; a result that is due to 
\citep{p:vin-57}).

\subsection{Proof of Theorem \ref{theorem:simulation}}\label{subsec:pf_simul}
%\textcolor{red}{(In case I forget later, we should probably write $q({\bf x}_{[0:n+1]})$ instead of $q({\bf x}_{[0,n+1]})$ for parity.)}

We first rewrite the density $g_{n, 2m}$ as 
\begin{eqnarray*}
  &&
     g_{n,2m}(x_{-m+1}, \ldots, \ldots,x_{n+m}) \\
  &=&
      C_{n+2m}^{-1}f(x_{-m+1}) \left[\prod_{i =-m+2}^{n+m}  R(x_{i-1},x_{i}) \right] f(x_{n+m})\\
    &=&  C_{n+2m}^{-1} \, q({\bf x}_{[0:n+1]}) \, \left(f(x_{-m+1}) \, \prod_{i=-m+2}^{0}R(x_{i-1},x_{i})\right)
  \left(f(x_{n+m}) \, \prod_{i=1}^{m-1}R(x_{n+i},x_{n+i+1})\right)
\end{eqnarray*}
where $q({\bf x}_{[0:n+1]})=R(x_{0},x_{1}) \times \cdots \times R(x_{n},x_{n+1})$ and $C_{n+2m}$ is the normalizing constant. From the second line of the above display and \eqref{eq:iter_op}, it follows that 
\begin{eqnarray*}
 C_{n+2m}  &=& \int \left\{T^{n+2m-1}(f)[x_{n+m}]\right\} \, f(x_{n+m})\mu(dx_{n+m}).
\end{eqnarray*}  
Also, integrating over the variables in $\m I$ in the last line of the above display, and again invoking \eqref{eq:iter_op}, it follows that 
\begin{eqnarray*}
h_{[0:n+1]}^{(m-1)}(x_{0},\ldots,x_{n+1})  &=& C_{n+2m}^{-1} \, q({\bf x}_{[0:n+1]}) \, \left\{ T^{m-1}(f)[x_0] \right\} \, \left\{ (T^*)^{(m-1)}(f)[x_{n+1}] \right\}.
\end{eqnarray*}
Using Lemma \ref{lemma:pow_it_main}, we can write
   \begin{eqnarray*}
     (T^*)^{(m-1)}(f)[x_{n+1}] &=& 
     \langle f,v\rangle r^{m-1}w(x_{n+1}) +(T^{*}Q_{w,v})^{m-1}(f)[x_{n+1}] \\
      T^{m-1}(f)[x_0]  &=& \langle f,w\rangle r^{m-1}v(x_{0}) +(TQ_{v,w})^{m-1}(f)[x_{0}] \\
     C_{n+2m} &=& \langle f,w\rangle \langle f,v\rangle r^{n+2m-1} +
           \int \left\{(TQ_{v,w})^{n+2m-1}(f)[x_{n+m}]\right\} \, f(x_{n+m})\mu(dx_{n+m}).
   \end{eqnarray*} 
  %$C_{n+2m}$ is 
%   \begin{eqnarray*}
%  C_{n+2m} &=&   \int g_{n,2m}(x_{-m-1},\ldots,x_{n+m})
%                    \prod_{i=-m+2}^{n+m}dx_{i}.
% %    &=& \int((T^{n+2m}f)[x_{n+m}])f(x_{n+m})dx_{n+m}.
% \end{eqnarray*}
Substituting these into the expression for $h_{[0:n+1]}^{(m-1)}$ and 
dividing the numerator and denominator by $r^{n+2m-1}$ gives
    \begin{eqnarray*}
    && h_{[0:n+1]}^{(m-1)}(x_{0},\ldots,x_{n+1})\\
     &=& \frac{1}{r^{n+1}}q({\bf x}_{[0:n+1]})\frac{
         \left\{  \langle f,w\rangle v(x_{0}) + \delta_1(x_0)\right\}
           \left\{ \langle f,v\rangle w(x_{n+1})
         + \delta_2(x_{n+1})\right\}}{\langle f,w\rangle \langle f,v\rangle + \delta_3}, 
   \end{eqnarray*}
where $\delta_1(x_0) = r^{-(m-1)} \, (TQ_{v,w})^{m-1}(f)[x_{0}]$, $\delta_2(x_{n+1}) = r^{-(m-1)} \, (T^{*}Q_{w,v})^{m-1}(f)[x_{n+1}]$, and 
\begin{align*}
\delta_3 = r^{-(n+2m-1)} \, \int \left\{(TQ_{v,w})^{n+2m-1}(f)[x_{n+m}] \right\} \, f(x_{n+m})\mu(dx_{n+m}).
\end{align*}
With this, we arrive at \eqref{eq:hnm}. We now make use of Lemma \ref{lemma:pow_it_main} to show that 
\begin{align}\label{eq:delta_bds}
\|\delta_1\| = O(\rho^{m-1}), \ \|\delta_2\| = O(\rho^{m-1}), \ |\delta_3| = O(\rho^{n+2m-1}). 
\end{align}
To that end, in Lemma \ref{lemma:pow_it_main}, choose $\varepsilon > 0$ small enough so that $\max\{|\lambda_2| + \varepsilon, |\lambda_2^*| + \varepsilon\} < r$, and define $\rho = r^{-1} \, \max\{|\lambda_2| + \varepsilon, |\lambda_2^*| + \varepsilon\} \in (0, 1)$. We then have from Lemma \ref{lemma:pow_it_main} that
\begin{align*}
\|\delta_1\| = r^{-(m-1)} \, \|(TQ_{v,w})^{m-1}(f)\| \le r^{-(m-1)} \, \|(T^{*}Q_{w,v})^{m-1}\|_{\rm op}  \, \|f\| \le C_\varepsilon \, \left(\frac{|\lambda_2|+\varepsilon}{r}\right)^{m-1} \, \|f\| = O(\rho^{m-1}). 
\end{align*}
Similarly, one obtains $\|\delta_2\| = O(\rho^{m-1})$. Using Cauchy--Schwarz inequality, 
\begin{eqnarray*}
|\delta_3| \le 
\frac{1}{r^{n+2m-1}} \|((TQ_{v,w})^{n+2m-1}(f))\| \, \|f\|
 \le C_{\varepsilon} \left(\frac{|\lambda_2|+\varepsilon}{r}\right)^{n+2m-1} \, \|f\|^{2} = O(\rho^{n+2m-1}).
\end{eqnarray*}
This establishes \eqref{eq:delta_bds}. 
% Thus $\frac{1}{\lambda^{n+1+2m}}\int
% ((TQ_{v})^{n+2m}f)[x_{n+m}]f(x_{n+m})dx_{n+m} =
% O((|\lambda_{2}|+\varepsilon)^{n+2m}/\lambda^{n+2m}) = O(\rho^{n+2m})$
% (for some $0\leq \rho<1$).
\\[2ex]
Let $K = (\langle f,v\rangle  \langle
     f,w\rangle + \delta_3)$. Using $p_{[0:n+1]}(x_{0},\ldots,x_{n+1}) =
r^{-(n+1)} \, v(x_{0}) q({\bf x}_{[0:n+1]}) w(x_{n+1})$, 
%we have 
\begin{eqnarray*}
&& h_{[0:n+1]}^{(m-1)}(x_{0},\ldots,x_{n+1}) -
   p_{[0:n+1]}(x_{0},\ldots,x_{n+1}) \\
 &=&  p_{[0:n+1]}(x_{0},\ldots,x_{n+1})\left(\frac{ \langle
     f,v\rangle  \langle f,w\rangle}{K}-1\right)+
     \mathcal{R}_{1}({\bf x}_{[0:n+1]})+\mathcal{R}_{2}({\bf
     x}_{[0:n+1]})+
     \mathcal{R}_{3}({\bf x}_{[0:n+1]})
\end{eqnarray*}
where, 
% with $K = (\langle f,v\rangle  \langle
%      f,w\rangle + O(\rho^{n+2m-1}))^{-1}$, %\textcolor{red}{(continue)}
\begin{eqnarray*}
\mathcal{R}_{1}({\bf x}_{[0:n+1]}) &=& \frac{\langle f,w\rangle}{K} \, \frac{1}{r^{n+1}} \, q({\bf x}_{[0:n+1]}) \, 
      v(x_{0}) \, \delta_2(x_{n+1}), \\
\mathcal{R}_{2}({\bf x}_{[0:n+1]})&=&
\frac{\langle f,v\rangle}{K} \, \frac{1}{r^{n+1}} \, q({\bf x}_{[0:n+1]}) \,  w(x_{n+1}) \, \delta_1(x_0),
     \\
 \mathcal{R}_{3}({\bf x}_{[0:n+1]}) &=& \frac{1}{K} \, \frac{1}{r^{n+1}} \, q({\bf x}_{[0:n+1]}) \, 
     \delta_1(x_0) \, 
     \delta_2(x_{n+1}).
\end{eqnarray*} 
Since 
\begin{eqnarray*}
\int  p_{[0:n+1]}(x_{0},\ldots,x_{n+1})\left|\frac{ \langle
     f,v\rangle  \langle f,w\rangle}{K}-1\right|\prod_{i=0}^{n+1} \mu(dx_{i}) = \frac{|\delta_3|}{K}, %= O(\rho^{n+2m-1}), 
 \end{eqnarray*}  
we have 
\begin{align}\label{eq:tv_bd_intermediate}
\| h_{[0:n+1]}^{(m-1)} - p_{[0:n+1]}\|_{\rm TV} 
\le \frac{|\delta_3|}{K}  + \int \left( \sum_{j=1}^3 \left |\mathcal{R}_{j}({\bf x}_{[0:n+1]}) \right| \right)\prod_{i=0}^{n+1} \mu(dx_{i}). 
\end{align}
We first show how to control the integral of $|\m R_1|$, that of $|\m R_2|$ follows similarly. Writing $q({\bf x}_{[0:n+1]}) = q({\bf x}_{[0:n]}) \times R(x_n, x_{n+1})$, we have, by Jensen and Cauchy--Schwarz, 
\begin{eqnarray*}
  \int |\mathcal{R}_{1}({\bf x}_{[0:n+1]})|\mu(dx_{n+1})
  &\leq & \frac{|\langle f,w\rangle|}{K} \, \frac{1}{r^{n+1}} \, q({\bf
     x}_{[0:n]}) \, v(x_{0}) \, \int R(x_{n},x_{n+1}) \, |\delta_2(x_{n+1})|  \mu(dx_{n+1})  \\   
  &\leq& \frac{|\langle f,w\rangle|}{K} \, \frac{1}{r^{n+1}} \, q({\bf
     x}_{[0:n]}) \, v(x_{0}) \, \|\widetilde{R}_{x_{n}}\| \, \|\delta_2\|, %\\
   % &\leq&  C \frac{1}{r^{n+1}}q({\bf
   %   x}_{[0,n]})v(x_{0}) \, \rho^{m-1},
\end{eqnarray*}  
where we define $\widetilde{R}_x(\cdot) = R(x, \cdot)$. 
% where $C_1 = K |\langle f,w\rangle| R_{\max} C_\varepsilon \|f\|$. Here, we 
% bounded $\|\widetilde{R}_{x_n}\|$ by $R_{\max}$, and $\|(T^{*}Q_{w,v})^{m-1}(f)\|$ by $\|(T^{*}Q_{w,v})^{m-1}\|_{\rm op}  \, \|f\|$, and then invoked Lemma \ref{lemma:pow_it_main}. 
This implies 
\begin{eqnarray*}
  \int |\mathcal{R}_{1}({\bf x}_{[0:n+1]})| \, \mu(dx_{n+1}) \prod_{i=0}^{n}\mu(dx_{i})
  &\leq& \frac{|\langle f,w\rangle| \, \|\delta_2\|}{K} \, \frac{1}{r^{n+1}} \, \int \|\widetilde{R}_{x_{n}}\| \, q({\bf
  x}_{[0:n]})v(x_{0}) \prod_{i=0}^{n}\mu(dx_{i})\\
  &=& \frac{|\langle f,w\rangle| \, \|\delta_2\|}{K} \, \frac{1}{r^{n+1}} \, \int \|\widetilde{R}_{x_{n}}\| \, T^{n}v[x_{n}]\mu(dx_{n})\\
  &= & \frac{|\langle f,w\rangle| \, \|\delta_2\|}{K} \, \frac{1}{r} \, \int \|\widetilde{R}_{x_{n}}\| \, v(x_{n}) \, \mu(dx_{n}) \\
  &\le & \frac{|\langle f,w\rangle| \, \|R(\cdot, \cdot)\| \, \|v\|}{r K} \, \|\delta_2\| = O(\rho^{m-1}). 
\end{eqnarray*}
In the above, going from the first to the second line, we invoked \eqref{eq:iter_op}; going from the second to the third line, we used $T^n v = r^n v$; and going from the third to the fourth line, we invoked Cauchy--Schwarz inequality and the fact that $\int \|\widetilde{R}_{x_{n}}\|^2 \mu(d x_n) = \|R(\cdot, \cdot)\|^2$. Here, we used that $\|R(\cdot, \cdot)\| < \infty$ by the HS condition. At the final step, we used \eqref{eq:delta_bds}.
% {\color{red}I am just wondering if the following would give a more universal bound.
% \begin{eqnarray*}
%   \int |\mathcal{R}_{1}({\bf x}_{[0:n+1]})|dx_{n+1}\prod_{i=0}^{n}\mu(dx_{i})
%   &\leq& C_1\rho^{m-1} \frac{1}{r^{n+1}}\int q({\bf x}_{[0:n]})\|\widetilde{R}_{x_{n}}\|v(x_{0}) \prod_{i=0}^{n}\mu(dx_{i})\\
%   &=& C_1\rho^{m-1}\frac{1}{r^{n+1}}\int T^{n}v[x_{n}]\|\widetilde{R}_{x_{n}}\|\mu(dx_{n})\\
%   &=& C_1\rho^{m-1}\frac{1}{r}\int v(x_{n}) \|\widetilde{R}_{x_{n}}\|dx_{n} = O\left(\rho^{m-1}\right)
% \end{eqnarray*}
%

By a similar argument one arrives at 
\begin{eqnarray*}
  \int |\mathcal{R}_{2}({\bf x}_{[0:n+1]})| \, \prod_{i=0}^{n+1} \mu(dx_i) 
  \leq  \frac{|\langle f,v\rangle| \, \|R(\cdot, \cdot)\| \, \|w\|}{r K} \, \|\delta_1\| = O(\rho^{m-1}).
\end{eqnarray*}  
Finally, we deal with the integral of $|\m R_3|$, which is a lower order term. Writing $q({\bf x}_{[0:n+1]}) = R(x_0, x_1) \times q({\bf x}_{[1:n]}) \times R(x_n, x_{n+1})$, and using similar arguments as above, one has 
\begin{eqnarray*}
  \int |\mathcal{R}_{3}({\bf x}_{[0:n+1]})| \, \mu(dx_{0}) \mu(dx_{n+1})  \le  \frac{1}{K} \frac{1}{r^{n+1}} \, q({\bf
         x}_{[1:n]}) \|R_{x_1}\| \, \|\delta_1\| \, \|\widetilde{R}_{x_n}\| \, \|\delta_2\|.
\end{eqnarray*}
Let $\ell(x) = \|R_x\|$ and $\widetilde{\ell}(x) = \|\widetilde{R}_x\|$. This leads to 
\begin{align*}
\int |\mathcal{R}_{3}({\bf x}_{[0:n+1]})| \, \mu(dx_{0}) \mu(dx_{n+1}) \prod_{i=1}^n \mu(dx_i) 
& \le \frac{\|\delta_1\| \|\delta_2\|}{K} \, \frac{1}{r^{n+1}} \, \int T^{n-1}(\ell)[x_n] \, \widetilde{\ell}(x_n) \mu(dx_n) \\
& \le \frac{\|\widetilde{\ell}\|}{r^2 K} \, \frac{\|T^{n-1}(\ell)\|}{r^{n-1}} \, \|\delta_1\| \|\delta_2\|. 
\end{align*}
Using Lemma \ref{lemma:pow_it_main}, $\frac{\|T^{n-1}(\ell)\|}{r^{n-1}}$ is an $O(1)$ term, and the overall order of the above integral for $|\m R_3|$ is $O(\rho^{2(m-1)})$ using \eqref{eq:delta_bds}. Combining all bounds inside \eqref{eq:tv_bd_intermediate} and using \eqref{eq:delta_bds}, the leading contribution comes from the integrals for $|\m R_1|$ and $|\m R_2|$, and the overall order of the TV is $O(\rho^{m-1})$. 

% Returning to $\mathcal{R}_{1}({\bf x}_{[0,n+1]})$ we have
% \begin{eqnarray*}
%   \int |\mathcal{R}_{1}({\bf x}_{[0,n+1]})|dx_{n+1}\prod_{i=0}^{n}dx_{i}
%   &\leq& C\rho^{m} \frac{1}{\lambda^{n+1}}\int q({\bf
%   x}_{[0,n]})v(x_{0}) \prod_{i=0}^{n}dx_{i}\\
%   &\leq& C\rho^{m}\frac{1}{\lambda^{n+1}}\int T^{n}v[x_{n}]dx_{n}\\
%   &=& C\rho^{m}\frac{1}{\lambda}\int v(x_{n})dx_{n} = O\left(\rho^{m}\right).
% \end{eqnarray*}
% Similarly
% $ \int |\mathcal{R}_{2}({\bf
%   x}_{[0,n+1]})|\prod_{i=0}^{n+1}dx_{i}=O(\rho^{m})$ and
%  $\int |\mathcal{R}_{3}({\bf
%    x}_{[0,n+1]})|\prod_{i=0}^{n+1}dx_{i}=O(\rho^{2m})$.
%  Finally,
%  \begin{eqnarray*}
% \int  p_{[0,\ldots,n+1]}(x_{0},\ldots,x_{n+1})\left|\frac{ \langle
%      f,v\rangle  \langle f,w\rangle}{ \langle f,v\rangle  \langle
%      f,w\rangle + O(\rho^{n+2m})}-1\right|\prod_{i=0}^{n+1}dx_{i}  =O(\rho^{n+2m}).
%  \end{eqnarray*}  

\subsection{Proof of Lemma \ref{lemma:pow_it_main}}\label{subsec:pf_pow_it_main}
Let $\mbox{id}$ denote the identity operator which maps $f$ to itself, so that we can write $Q_{v,w} = \mbox{id} - P_{v,w}$.

We first show that $P_{v,w}$ (and hence $Q_{v,w}$) is a projection operator. To that end, for any $f \in L^2(\m X, \mu)$, 
\begin{align*}
& P_{v,w}^2(f) = P_{v,w}(P_{v,w}(f)) = P_{v,w}(\langle f, w\rangle v) \\
& = \left \langle \langle f, w\rangle v, w\right \rangle v =  \langle f, w \rangle \, \langle v, w\rangle \, v = P_{v,w}(f),
\end{align*}
since $\langle v, w\rangle =1$. Since $f$ is arbitrary, we have $P_{v,w}^2 = P_{v,w}$. Also, $Q_{v,w}^2 = (\mbox{id} - P_{v,w})^2 = \mbox{id} - 2 P_{v,w} + P_{v,w}^2 = Q_{v,w}$. 

Next, we show that the projection operators $P_{v,w}$ and $Q_{v,w}$ commute with $T$. To that end, for any $f \in L^2(\m X, \mu)$,
\begin{align*}
T P_{v,w}(f) = T(\langle f, w\rangle v) = \langle f, w\rangle T(v) = r \langle f, w \rangle v,
\end{align*}
and 
\begin{align*}
P_{v,w} T(f) = \langle T(f), w \rangle v = \langle f, T^*(w)\rangle v = \langle f, r w \rangle v = r \langle f, w \rangle v. 
\end{align*}
This shows $T P_{v,w} = P_{v,w} T$. One can similarly show $T Q_{v,w} = Q_{v,w} T$. 

We first establish \eqref{eq:pow_non_sa}. Using linearity of $T^k$, write $T^k(f) = T^k(P_{v,w}(f)) + T^k(Q_{v,w}(f))$. The first term equals $T^k(\langle f, w\rangle v) = \langle f, w\rangle \, T^k(v) = \langle f, w\rangle \, r^k v$. It thus remains to show that $T^k Q_{v,w} = (T Q_{v,w})^k$. We prove this by induction. The assertion is clearly true for $k = 1$. Now, write 
\begin{align*}
T^k Q_{v,w} = T^{k-1} T Q_{v,w} = T^{k-1} T Q_{v,w} Q_{v,w} = T^{k-1} Q_{v,w} T Q_{v,w} = (TQ_{v,w})^{k-1} TQ_{v,w} = (TQ_{v,w})^k
\end{align*}
to prove it for any $k > 1$. In the above display, we use $Q_{v,w} = Q_{v,w}^2$ in the second step, $T Q_{v,w} = Q_{v,w} T$ in the third step, and our inductive hypothesis $T^{k-1}Q_{v,w} = (TQ_{v,w})^{k-1}$ in the fourth step. 

We next establish \eqref{eq:op_norm_bd}. 
%We first show it for the case $k = 1$. 
Consider the operator $T Q_{v,w}$. We show below that its spectral radius $r(T Q_{v,w}) \le |\lambda_2|$, where recall $|\lambda_2| :\,= \sup \{|\lambda| \,:\, \lambda \in \sigma(T), \lambda \ne r\}$. Then, from \eqref{eq:sp_inf}, we have that for any $\varepsilon > 0$, there exists a norm $|\cdot|_\varepsilon$ on $L^2(\m X, \mu)$ equivalent to $\|\cdot\|$ (that is, there exist constants $a_\varepsilon, b_\varepsilon > 0$ such that $a_\varepsilon |f|_\varepsilon \le \|f\|_2 \le b_\varepsilon |f|_\varepsilon$ for all $f \in L^2(\m X, \mu)$) 
%{\color{blue} and that rather than such that?}
%such that 
with the induced norm $|T Q_{v,w}|_{{\rm op},\varepsilon} < (|\lambda_2| + \varepsilon)$. Fix $f \ne 0$. One then has, for any $k \ge 1$, 
\begin{align*}
\|(T Q_{v,w})^k f\| \le b_\varepsilon |(T Q_{v,w})^k f|_{\varepsilon} \le b_\varepsilon \, |(T Q_{v,w})^k|_{{\rm op},\varepsilon} \, |f|_\varepsilon \le \frac{b_\varepsilon}{a_\varepsilon} \left(|T Q_{v,w} |_\varepsilon\right)^k \, \|f\| \le C_\varepsilon \, (|\lambda_2| + \varepsilon)^k \, \|f\|, 
\end{align*}
where $C_\varepsilon = \frac{b_\varepsilon}{a_\varepsilon}$. This implies $\|(T Q_{v,w})^k\|_{\rm op} \le C_\varepsilon \, (|\lambda_2| + \varepsilon)^k$, giving \eqref{eq:op_norm_bd}. In the above display, going from the third to the fourth step, we use sub-multiplicativity of the induced norm, which follows from definition. 

It remains to establish $r(T Q_{v,w}) \le |\lambda_2|$. 
First, we show that $r$ cannot be an eigenvalue of $T Q_{v,w}$. Suppose in the contrary that there exists $f \ne 0$ such that $T Q_{v,w} f = r f$. Using the commutative property, this gives $Q_{v,w} T f = r f$. Apply $Q_{v,w}$ both sides and use $Q_{v,w}^2 = Q_{v,w}$ to get $Q_{v,w} T f = r Q_{v,w} f$, and hence $T (Q_{v,w} f) = r Q_{v,w} f$. Thus, either $Q_{v,w} f = 0$ or it is an eigenvector corresponding to $r$. If $Q_{v,w} f = 0$, substituting in the equation $T Q_{v,w} f = r f$, we obtain $r f = 0$, which is a contradiction, since $r > 0$ and $f \ne 0$ by assumption. Thus, assume $Q_{v,w} f \ne 0$ and it is an eigenfunction of $T$ corresponding to $r$. Since we know that the spectral radius $r$ is a simple eigenvalue, this means there exists a constant $c \ne 0$ such that $Q_{v,w} f = c v$. Applying $Q_{v,w}$ on both sides, we get $Q_{v,w} f =  c Q v = 0$, since $Q v = v - \langle v, w\rangle v = 0$. This again leads to a contradiction since $c \ne 0$ and $v > 0$ a.e. 

%{\color{blue}Would it make sense to start the paragraph by saying that: Next we show that all eigenvalues of $TQ_{v,w}$ are eigenvalues of $T$, which then proves by using the above that 
%$r(TQ_{v,w})\leq \lambda_2$}
Next we show that all eigenvalues of $TQ_{v,w}$ are eigenvalues of $T$, which in conjunction with the above paragraph then proves that 
$r(TQ_{v,w})\leq |\lambda_2|$. 
To that end, if $\lambda \ne r$ is an eigenvalue of $T Q_{v,w}$, then there exists $f \ne 0$ such that $T Q_{v,w} f = \lambda f$. This implies $Q_{v,w} T f = \lambda f$. Applying $Q_{v,w}$ on both sides, $Q_{v,w} T f = \lambda Q_{v,w} f$, and hence $T Q_{v,w} f = \lambda Q_{v,w} f$. Now, $Q_{v,w} f \ne 0$, since otherwise one would get $f = 0$ from the equation $T Q_{v,w} f = \lambda f$. This implies 
$\lambda$ is an eigenvalue of $T$ with eigenfunction $Q_{v,w} f$. %Together with the derivation in the previous paragraph, this shows $r(T Q_{v,w}) \le |\lambda_2|$. 

Finally, we establish \eqref{eq:non_sa_geom_cgence}. If $\langle f, w \rangle \ne 0$, we have from \eqref{eq:pow_non_sa} that 
\begin{align*}
\left \|\frac{T^k(f)}{r^k \langle f, w \rangle} - v \right \| = 
|\langle f, w \rangle|^{-1}  \, r^{-k} \, \|(T Q_{v,w})^k(f)\| 
\le |\langle f, w \rangle|^{-1}  \, r^{-k} \, \|f\| \, \|(T Q_{v,w})^k\|_{\rm op}.  
\end{align*} 
The desired inequality then follows from the operator norm bound in \eqref{eq:op_norm_bd}.

\section{Proof of Theorem \ref{theorem:stationarityd}}\label{sec:cased}

To prove the result we focus mainly on the case $d=2$
(to simplify notation), this gives the main idea of the proof. We give a brief description of the pertinent parts for the case $d=3$ and $d>3$.

\subsection{$2$-neighbourhood case}

Analysis of the $2$-neighbourhood case hinges on the shift invariant nature of the interaction kernel. 
It is worth pointing that out that on the surface the 
existence of a unique 
stationarity based on the conditional specification appears "obvious". The proof below shows that there are some subtle points that need to be considered. 

We recall that 
in the $2$-neighbourhood case the interaction kernel has the form
\begin{eqnarray*}
R(x_1,x_2,x_3,x_4) = c_{x_1}^{1/2}c_{x_2}^{1/2}
  A(x_{1},x_{2})^{1/2}[B(x_{1},x_{3})A(x_{2},x_{3})B(x_{2},x_{4})]
  A(x_{3},x_{4})^{1/2}c_{x_3}^{1/2}c_{x_4}^{1/2},  
\end{eqnarray*}
where $A$ are the one-lag interactions and $B$ the two lag interactions.
To analysis the corresponding distribution we require the following
background calculations. This requires the following definition: for any $s \in \mb Z$ and $n \in \mb N$, define
\begin{eqnarray*}
q_{[s+1:s+2n]}(x_1, \ldots, x_{2n})= q_{[1:2n]}(x_{1},\ldots,x_{2n}) = \prod_{t=1}^{n-1}R(x_{2t-1}, x_{2t},x_{2t+1},x_{2t+2}).
\end{eqnarray*}

\paragraph{Preliminary calculations}
Using the above we have 
\begin{eqnarray*}
  &&R(x_0,x_1,x_2,x_3) R(x_2,x_3,x_4,x_5)\\
  &=& c_{x_0}^{1/2}c_{x_1}^{1/2}
  A(x_{0},x_{1})^{1/2}[B(x_{0},x_{2})A(x_{1},x_{2})B(x_{1},x_{3})]A(x_{2},x_{3})^{1/2} c_{x_2}^{1/2}c_{x_3}^{1/2}\\
                    && \times    
    c_{x_2}^{1/2}c_{x_3}^{1/2}                A(x_{2},x_{3})^{1/2}[B(x_{2},x_{4})A(x_{3},x_{4})B(x_{3},x_{5})]A(x_{4},x_{5})^{1/2} c_{x_4}^{1/2}c_{x_5}^{1/2}\\
  &=&
      \underbrace{c_{x_0}^{1/2}A(x_{0},x_{1})^{1/2}B(x_{0},x_{2}) A(x_{1},x_{2})^{1/2}c_{x_2}^{1/2}}_{=D(x_{0},x_{1},x_{2})c_{x_0}^{1/2}c_{x_2}^{1/2}}
   R(x_{1},x_{2},x_{3},x_{4}) \underbrace{c_{x_3}^{1/2}A(x_{3},x_{4})^{1/2}B(x_{3},x_{5}) A(x_{4},x_{5})^{1/2}c_{x_5}^{1/2}}_{=D(x_{5},x_{4},x_{3})c_{x_3}^{1/2}c_{x_5}^{1/2}}.                                        \end{eqnarray*}
   
Let
$E(x_{0},x_{1},x_{2}) = D(x_{0},x_{1},x_{2})c_{x_1}^{1/2}c_{x_2}^{1/2}$ and $\widetilde{E}(x_{3},x_{4},x_{5}) = E(x_{5},x_{4},x_{3})$.    
Then the above implies 
\begin{eqnarray*}
  R(x_0,x_1,x_2,x_3) R(x_2,x_3,x_4,x_5)  &=& E(x_{0},x_{1},x_{2})
                                             R(x_{1},x_{2},x_{3},x_{4}) \widetilde{E}(x_{3},x_{4},x_{5}). 
\end{eqnarray*}
Recursing on the above we have  
\begin{eqnarray*}
  &&  R(x_0,x_1,x_2,x_3) R(x_2,x_3,x_4,x_5) R(x_4,x_5,x_6,x_7)\\
     &=&
  E(x_{0},x_1,x_2) R(x_1,x_2,x_3,x_4) R(x_3,x_4,x_5,x_6)
  \widetilde{E}(x_5,x_6,x_7)
\end{eqnarray*}
and the general expression
%\begin{eqnarray*}
% && R(x_0,x_1,x_2,x_3) R(x_2,x_3,x_4,x_5) R(x_4,x_5,x_6,x_7)\ldots R(x_{2n-4},x_{2n-3},x_{2n-2},x_{2n-1})\\
%  &=&D(x_{0},x_1,x_2) R(x_1,x_2,x_3,x_4) R(x_3,x_4,x_5,x_6)\ldots
%  R(x_{2n-5},x_{2n-4},x_{2n-3},x_{2n-2}) D(x_{2n-3},x_{2n-2},x_{2n-1}).
%\end{eqnarray*}
%Hence, borrowing notation from the paper
\begin{eqnarray}
\label{eq:qqq}  
  q_{[0:2n+1]}(x_{0},\ldots,x_{2n+1}) = E(x_{0},x_1,x_2)  q_{[1:2n]}(x_{1},\ldots,x_{2n})\widetilde{E}(x_{2n-1},x_{2n},x_{2n+1})
\end{eqnarray}  

\paragraph{Proof of Stationarity (case $d=2$)}
Employing Theorem \ref{theorem:one_stationarity} (under compactness of $T$, viewed as an operator on $L^2(\m X_2, \mu_2)$), there exists a unique vector stationary process $\{\xi_t\}_{t \in \mb Z}$ with $\xi_t = (Y_t, Z_t) \in \m X_2$, where the joint density of $(\xi_{s},\ldots,\xi_{n+s-1})$ is given by
\begin{eqnarray*}
 p^{\xi}_{[1:n]}(y_{1},z_{1},\ldots,y_{n},z_{n}) = r^{-(n-1)} \, v(y_{1},z_{1}) \left[
 \prod_{t=1}^{n-1}R(y_{t},z_{t},y_{t+1},z_{t+1})\right]
 w(y_{n},z_{n}).
\end{eqnarray*}
Construct a new stochastic process $\{X_t\}_{t \in \mb Z}$ on $\m X^{\mb Z}$ as $X_{2t-1} = Y_t$ and $X_{2t} = Z_t$. As a consequence of the vector stationarity of $\xi_t$, it follows that for any {\it even} $s$ and $n \in \mb N$, the joint density $p_{[s+1:s+2n]}$ of $(X_{s+1}, \ldots, X_{s+2n})$ is given by 
\begin{eqnarray}\label{eq:p_sp1_sp2n}
 p_{[s+1:s+2n]}(x_{s+1},\ldots,x_{s+2n}) \propto
  v(x_{s+1},x_{s+2}) \, q_{[s+1:s+2n]}(x_{s+1},\ldots,x_{s+2n}) 
  \, w(x_{s+2n-1},x_{s+2n}). 
\end{eqnarray}
Clearly, $p_{[s+1:s+2n]} = p_{[1:2n]}$ for any even $s$ and $n \in \mb N$. 
We show below that the above holds for any $s \in \mb Z$, not just even $s$, to conclude that the joint density of $(X_{s+1}, \ldots, X_{s+2n})$ is $p_{[1:2n]}$. We sketch a detailed argument for the joint density $p_{[2:2m-1]}$ of $(X_2, \ldots, X_{2m-1})$ (i.e. with $s = 1$ and $n = (m-1)$ for $m > 2$), and the general argument (for arbitrary $s$) follows similarly exploiting the shift invariance of $q$. Our strategy is to arrive at two different expressions for $p_{[2:2m-1]}$, one by writing 
\begin{align}\label{eq:path_1}
p_{[2:2m-1]}(x_2, \ldots, x_{2m-1}) = \int p_{[-1:2m+2]}(x_{-1}, \ldots, x_{2m+2}) \, \prod_{i \in \m I} \mu(dx_i),
\end{align}
with $\m I = \{-1, 0, 1, 2m, 2m+1, 2m+2\}$, and the second by writing 
\begin{align}\label{eq:path_2}
p_{[2:2m-1]}(x_2, \ldots, x_{2m-1}) = \int p_{[1:2m]}(x_1, \ldots, x_{2m}) \mu(dx_1) \mu(dx_{2m}). 
\end{align}
Since $\{X_t\}$ is a valid stochastic process by construction, the expressions in the right hand sides of the previous two displays must be equal, and the desired result somewhat surprisingly follows from this identity.

We proceed to compute the respective quantities, starting with the right hand side of \eqref{eq:path_1}. 
Using \eqref{eq:qqq} inside \eqref{eq:p_sp1_sp2n}, we have 
\begin{eqnarray*}
 && p_{[-1:2m+2]}(x_{-1},\ldots,x_{2m+2})
  \propto v(x_{-1},x_{0})q_{[-1:2m+2]}(x_{-1},\ldots,x_{2m+2})
      w(x_{2n+1},x_{2m+2}) \\
  &=& v(x_{-1},x_{0})E(x_{-1},x_0,x_1) q_{[1:2m]}(x_{1},\ldots,x_{2m})\widetilde{E}(x_{2m-1},x_{2m},x_{2m+1}) w(x_{2m},x_{2m+1}).
\end{eqnarray*}
Integrating over $x_{-1}$ and $x_{2m+2}$, we obtain 
\begin{eqnarray}\label{eq:p_0_2mp1}
&&  p_{[0:2m+1]}(x_{0},\ldots,x_{2m+1})
  \propto 
  v_{1}(x_{0},x_{1})q_{[1:2m]}(x_{1},\ldots,x_{2m})
      w_{1}(x_{2m},x_{2m+1}), 
\end{eqnarray}
where
\begin{eqnarray}
  \label{eq:vw}
  v_{1}(x_{0},x_{1}) &=& \int E(x_{-1},x_0,x_1)v(x_{-1},x_0)\mu(d x_{-1})\nonumber\\
  \textrm{ and }
   w_{1}(x_{2m},x_{2m+1}) &=& \int \widetilde{E}(x_{2m},x_{2m+1},x_{2m+2})w(x_{2m+1},x_{2m+2})\mu(d x_{2m+2}).
\end{eqnarray}
Now, we integrate over $x_0, x_1$ and $x_{2m}, x_{2m+1}$ in \eqref{eq:p_0_2mp1}. Writing $q_{[1:2m]}(x_1, \ldots, x_{2m}) = R(x_1, x_2) q_{[2:2m-1]}(x_2, \ldots, x_{2m-1}) R(x_{2m-1}, x_{2m})$ and invoking \eqref{eq:iter_op}, one obtains 
\begin{eqnarray}
\label{eq:vw1}
  p_{[2:2m-1]}(x_{2},\ldots,x_{2m-1})
  \propto
  (Tv_{1})[x_{2},x_{3}]
  q_{[2:2m-1]}(x_{2},\ldots,x_{2m-1})
      (T^{*}w_{1})[x_{2m-2},x_{2m-1}].
\end{eqnarray}
Next, we compute the right hand side of \eqref{eq:path_2}. 
Starting from $p_{[1:2m]}$ and repeating the calculation leading to \eqref{eq:p_0_2mp1}, we get 
\begin{eqnarray}
\label{eq:vw2}
p_{[2:2m-1]}(x_{2},\ldots,x_{2m-1}) &\propto&      
  v_{1}(x_{2},x_{3})
  q_{[2:2m-1]}(x_{2},\ldots,x_{2m-1})
      w_{1}(x_{2m-2},x_{2m-1}).
\end{eqnarray}
Comparing (\ref{eq:vw1}) and (\ref{eq:vw2}) gives  
\begin{eqnarray*}
  Tv_{1} \propto v_{1} \textrm{ and }T^{*}w_{1} \propto w_{1}.
\end{eqnarray*}
This implies that $v_1$ and $w_1$ must be eigenfunctions of $T$ and $T^*$
respectively. However,  $v_1$ and $w_1$ are strictly positive, and we
know from Theorem \ref{thm:KR_without_int_main} 
and Remark \ref{rem:only_positive}
that the positive eigenfunctions are unique, thus $w_{1}=w$ and $v_1=v$ (upto scaling). 
This gives 
\begin{eqnarray*}
p_{[2:2m-1]}(x_{2},\ldots,x_{2m-1}) &\propto&      
  v(x_{2},x_{3})
  q_{[2:2m-1]}(x_{2},\ldots,x_{2m-1})
      w(x_{2m-2},x_{2m-1}),
\end{eqnarray*}
proving our claim. Consequently, \emph{for all} $s \in \mb Z$ and $n$ the joint density of $X_{s},X_{s+1},\ldots,X_{s+2n-1}$ is 
\begin{eqnarray*}
 p_{[1:2n]}(x_{1},\ldots,x_{2n}) = r^{-(n-1)}
  v(x_{1},x_{2})q_{[0:2n-1]}(x_{1},\ldots,x_{2n-1}) w(x_{2n-1},x_{2n}).
\end{eqnarray*}
By using the same arguments in 
Theorem \ref{theorem:one_stationarity} this immediately implies that $\{X_{t};t\in \mathbb{Z}\}$ is a strictly stationary process. 
% Since the distributions of $X_{2t-1},X_{2t}$ and
%  $X_{2t},X_{2t+1}$ are the same, the marginal distributions of $X_{2t}$ and $X_{2t+1}$ are the same.
%By a similar argument we have that
%  the distributions of $X_{r+1},\ldots,X_{r+2n}$ are shift invariant for
%  all $r$ and $n$, thus any subset $X_{r+t_1},\ldots,X_{r+t_k}$ is shift invariant. This 
%  immediately implies that the stochastic
%  process is
%strictly stationarity, thus proving the result for the case $d=2$. 
%The case for $d>2$, is similar we briefly outline the pertinent parts of the proof below.

\medskip 

It is interesting to note that equation (\ref{eq:vw}) implies that
\begin{eqnarray}
  \label{eq:vw_modif}
 \int E(x_0,x_1,x_2)v(x_0,x_1)\mu(d x_0)\propto v(x_1,x_2)\nonumber\\
  \int \widetilde{E}(x_{2n-1},x_{2n},x_{2n+1})w(x_{2n},x_{2n+1})\mu(d
  x_{2n+1}) \propto w(x_{2n-1},x_{2n}).
\end{eqnarray} 
Further using the above calculations we obtain the joint distribution of $2n+1$ consecutive
random variables $X_{s},X_{s+1},\ldots,X_{s+2n}$
\begin{eqnarray}
  p^{}_{[0:2n+1]}(x_{0},\ldots,x_{2n})
  &\propto& v(x_{0},x_{1})E(x_{0},x_{1},x_{2})q_{[1:2n]}(x_{1},\ldots,x_{2n})
      w(x_{2n-1},x_{2n}).  
\end{eqnarray}

\subsection{$3$-neighbourhood case}

We now generalize the above arguments 
to the $3$-neighbourhood case. Broadly, the arguments are similar to the case $d=2$. However, there are some 
differences between the case $d=2$ and $d>2$ which are 
essentially captured in the case $d=3$. The main difference are that there are 
two possible "internal" interaction kernels in 
the product
\begin{eqnarray*}
R(x_0,x_{1},x_{2},x_3,x_4,x_5) R(x_3,x_{4},x_{5},x_6,x_7,x_8),
\end{eqnarray*}
whereas in the case $d=2$ there is only one. Further, 
in the case $d=2$, there were certain symmetries that we exploited that are lost for the case
$d=3$. We make this precise below. 

For the case $d=3$, the interaction kernel  has the form
\begin{align*}
  R(x_1,x_2,x_3,x_4,x_5,x_6) =&
   A(x_{1},x_{2})^{1/2}A(x_2,x_3)^{1/2}B(x_1,x_3)^{1/2}c_{x_1}^{1/2}c_{x_2}^{1/2}c_{x_3}^{1/2} \\
 &  [C(x_1,x_4)B(x_{2},x_4)C(x_{2},x_5)^{1/2}A(x_3,x_4)C(x_{2},x_5)^{1/2}B(x_{3},x_5)C(x_3,x_6)]\\
 &   A(x_{4},x_{5})^{1/2}A(x_5,x_6)^{1/2}B(x_4,x_6)^{1/2}
 c_{x_4}^{1/2}c_{x_5}^{1/2}c_{x_6}^{1/2}.
\end{align*}
Analogous to $D$ in the $2$-neighbourhood case we define 
\begin{eqnarray*}
  D_{1}(x_{1},x_{2},x_{3},x_{4}) &=&
                                     A(x_{1},x_{2})^{1/2}A(x_{3},x_{4})^{1/2}B(x_{1},x_{3})^{1/2}
                                     B(x_2,x_4)^{1/2}C(x_1,x_4)
                                     \\
 D_{2}(x_{2},x_{3},x_{4},x_{5}) &=& A(x_2,x_{3})^{1/2}B(x_{2},x_{4})^{1/2}C(x_2,x_5)B(x_{3},x_5)^{1/2} \\
 D_{3}(x_3,x_4,x_5,x_6) &=& 
C(x_3,x_6)A(x_3,x_4)^{1/2}B(x_3,x_5)^{1/2}.
\end{eqnarray*}
We observe that
$D_{1}(x_{1},x_{2},x_3,x_4)$ removes all interactions in $R(x_{1},\ldots,x_{6})$ which contain $x_{1}$
 \emph{and} all (half) interactions from either $2$ and $3$ to  $4$.
$D_{2}(x_{2},x_3,x_4,x_5)$ removes interactions
that contain $x_{2}$ and all (half) interactions from $3$ to $5$. $D_{3}(x_{3},x_4,x_5,x_6)$ removes all interactions from $x_3$ to $x_{4}$ to $x_6$.
An
illustration is given in Figure \ref{fig:3}.
\begin{figure}[h!]
\begin{center}
  \includegraphics[scale =0.5]{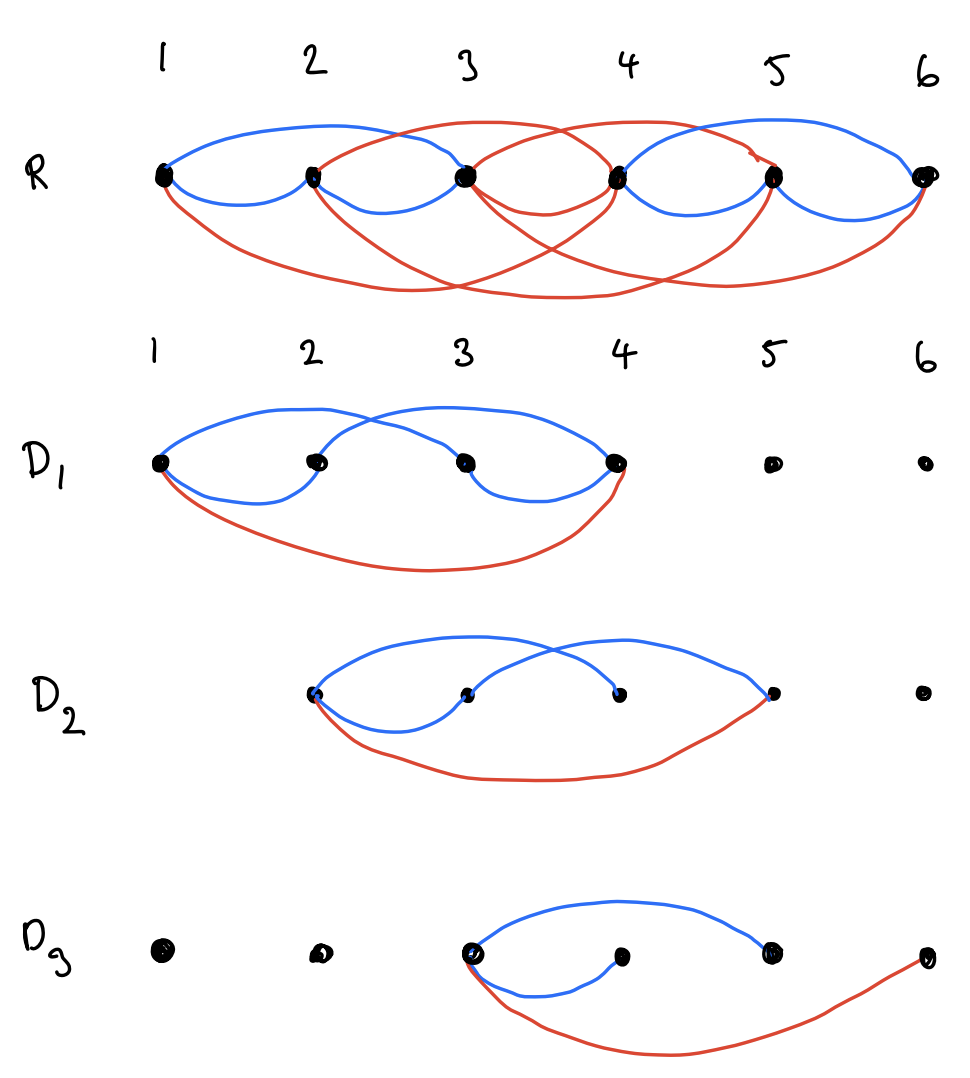}
\end{center}
\caption{Top: $R$, Middle: $D_1$ and $D_{2}$ Bottom: $D_3$. \label{fig:3}}
\end{figure}
Using $D_{1},D_{2}$ and $D_{3}$ we define 
%$F_{1},F_{2}$ and $F_{3}$ as 
%\begin{eqnarray*}
%F_{3}(x_{1},x_{3},x_{4}) &=& %D_{3}(x_{4},x_{3},x_{2},x_{1}),
%\quad F_{2}(x_{1},x_{3},x_{4}) = D_{2}(x_{4},x_{3},x_{2},x_{1}),
%\\
%F_{1}(x_{1},x_{3},x_{4}) &=& 
%D_{1}(x_{4},x_{3},x_{2},x_{1}).
%\end{eqnarray*}
%Further, let
\begin{eqnarray*}
E_{1}(x_{1},x_{2},x_{3},x_{4}) &=& 
D_{1}(x_1,x_2,x_3,x_4)c_{x_{1}}^{1/2}c_{x_4}^{1/2} \\
E_{2}(x_{1},x_{2},x_{3},x_{4},x_{5}) &=& 
D_{1}(x_1,x_2,x_3,x_4)D_{2}(x_2,x_3,x_4,x_5)
c_{x_{1}}^{1/2}c_{x_{2}}^{1/2}c_{x_4}^{1/2}c_{x_5}^{1/2} \\
\widetilde{E}_{2}(x_{1},x_{2},x_{3},x_{4},x_{5}) &=&
E_{2}(x_{5},x_{4},x_{3},x_{2},x_{1}) \\
\widetilde{E}_{1}(x_{1},x_{2},x_{3},x_{4}) &=& 
E_{1}(x_{4},x_{3},x_{2},x_{1}).
\end{eqnarray*}
Then, by using the above notation
\begin{eqnarray*}
  &=&R(x_0,x_{1},x_{2},x_3,x_4,x_5) R(x_3,x_{4},x_{5},x_6,x_7,x_8)   \\
  &=& E_{2}(x_{0},x_{1},x_{2},x_{3},x_{4}) 
     R(x_{2},x_{3},x_4,x_5,x_6,x_7)  \widetilde{E}_{1}(x_{5},x_{6},x_{7},x_{8}) \\
  &=&E_{1}(x_{0},x_{1},x_{2},x_{3}) 
R(x_{1},x_{2},x_3,x_4,x_5,x_6)
     \widetilde{E}_{2}(x_{4},x_{5},x_{6},x_{7},x_8).
\end{eqnarray*}  
More generally, we have
\begin{eqnarray*}
q_{[0,3n+2]}(x_{0},\ldots,x_{3n+2}) 
   &=&
  E_{1}(x_{0},x_{1},x_{2},x_{3})  q_{[1,3n]}(x_{1},\ldots,x_{3n}) 
  \widetilde{E}_{2}(x_{3n-2},x_{3n-1},x_{3n},x_{3n+1},x_{3n+2})\\
q_{[-1,3n+1]}(x_{-1},\ldots,x_{3n+1})
 &=&      E_{2}(x_{-1},x_{0},x_{1},x_{2},x_3)
  q_{[1,3n]}(x_{1},\ldots,x_{3n}) 
  \widetilde{E}_{1}(x_{3n-2},x_{3n-1},x_{3n},x_{3n+1}),
\end{eqnarray*}
where 
$q_{[0,3n+2]}(x_{0},\ldots,x_{3n+2}) = \prod_{t=1}^{n}R(x_{3t-3}, x_{3t-2}, x_{3t-1}, x_{3t},x_{3t+1},x_{3t+2})$.
We use the above identities in the stationarity proof. 

\paragraph{Proof of stationarity (case $d=3$)}
As in the $2$-neighbourhood case, there 
exists a unique stationary process
$\{\xi_{t} = (Y_{t},Z_{t},U_{t})\}_{t}$ with the 
the stated
conditional specification. For each $n$ the joint distribution of
$\{\xi_{t}\}_{t=1}^{n}$ is
\begin{eqnarray*}
  p_{[1:n]}^{\xi}(y_{1},z_{1},u_{1},\ldots,y_n,z_n,u_{n}) = v(y_{1},z_{1},u_{1})q_{[1:3n]}(y_{1},\ldots,u_{n})w(y_{n},z_{n},u_{n}).
\end{eqnarray*}
We relabel the entries of the process
$X_{3t} = Y_{t}$, $X_{3t+1}=Z_{t}$ and 
$X_{3t+2} = U_{t}$. Thus by rewriting the above, the joint distribution of  
$\{X_{t}\}_{t=1}^{3n}$ is
\begin{eqnarray*}
  p_{[1:3n]}(x_{1},\ldots,x_{3n}) = r^{-n+1}v(x_{1},x_{2},x_{3})q_{[1:3n]}(x_{1},\ldots,x_{3n})w(x_{3n-2},x_{3n-1},x_{3n}).
\end{eqnarray*}

By using the same ideas as in the $2$-neighbourhood case the
joint distribution of $X_{2},\ldots,X_{3n+1}$ is  
\begin{eqnarray*}
  p_{[2:3n+1]}(x_{2},\ldots,x_{3n+1}) =
  v_{1}(x_{2},x_{3},x_{4})q_{[2:3n+1]}(x_{2},\ldots,x_{3n+1})
  w_1(x_{3n-1},x_{3n},x_{3n+1}),
\end{eqnarray*}
where
\begin{eqnarray*}
&&  v_{1}(x_{2},x_{3},x_{4}) \propto \int
  E_{1}(x_{1},x_{2},x_{3},x_{4})v(x_{1},x_{2},x_{3})\mu(dx_1) \\
&&  w_{1}(x_{3n-1},x_{3n},x_{3n+1}) \propto \int
\widetilde{E}_{2}(x_{3n-1},x_{3n},x_{3n+1},x_{3n+2},x_{3n+3})\times\\
 &&  w(x_{3n+1},x_{3n+2},x_{3n+3})\mu(dx_{3n+3}) \mu(dx_{3n+2}).
\end{eqnarray*} 
But on the other hand, by integrating out
$x_{-1},x_{0},x_{1},x_{3n+2},x_{3n+3},x_{3n+4}$ we can show that 
\begin{eqnarray*}
  p_{[2:3n+1]}(x_{2},\ldots,x_{3n+1}) =
  [Tv_{1}](x_{2},x_{3},x_{4})q_{[2:3n+1]}(x_{2},\ldots,x_{3n+1})
  [T^{*}w_1](x_{3n-1},x_{3n},x_{3n+1}).
\end{eqnarray*}
Consequently, 
\begin{eqnarray*}
  Tv_{1} \propto v_{1} \textrm{ and }T^{*}w_{1} \propto w_{1}.
\end{eqnarray*}
Now by using the same argument as in the case $d=2$, this implies that 
$v_1=v$ and $w_1=w$, hence 
the joint density of 
$X_{2},\ldots,X_{3n+1}$ is 
\begin{eqnarray*}
  p_{[2:3n+1]}(x_{2},\ldots,x_{3n+1}) = r^{-n+1}v(x_{2},x_{3},x_{4})q_{[1:3n]}(x_{2},\ldots,x_{3n+1})w(x_{3n-1},x_{3n},x_{3n+1}).
\end{eqnarray*}
By a similar argument
the joint distribution of $X_{3},\ldots,X_{3n+2}$ is 
\begin{eqnarray*}
  p_{[3:3n+2]}(x_{3},\ldots,x_{3n+2}) \propto
  v_{2}(x_{3},x_{4},x_{5})q_{[2:3n+2]}(x_{3},\ldots,x_{3n+2})
  w_2(x_{3n},x_{3n+1},x_{3n+2}),
\end{eqnarray*}
where
\begin{eqnarray*}
  &&  v_{2}(x_{3},x_{4},x_{5}) \propto \int
     E_{2}(x_{1},x_{2},x_{3},x_{4},x_5)
    v(x_{1},x_{2},x_{3})
     \mu(dx_1)\mu(dx_{2}) \\
&&  w_{2}(x_{3n},x_{3n+1},x_{3n+2}) \propto \int
   \widetilde{E}_{1}(x_{3n},x_{3n+1},x_{3n+2},x_{3n+3})w(x_{3n+1},x_{3n+2},x_{3n+3}) \mu(dx_{3n+3}).
  \end{eqnarray*} 
But also 
\begin{eqnarray*}
  p_{[3:3n+2]}(x_{3},\ldots,x_{3n+2}) \propto
  [Tv_{2}](x_{3},x_{4},x_{5})q_{[2:3n+2]}(x_{3},\ldots,x_{3n+2})
  [T^{*}w_2](x_{3n},x_{3n+1},x_{3n+2}),
\end{eqnarray*}
Hence $Tv_2 \propto v_2=v$ and $T^*w_2\propto w_2=w$. 
Hence the joint density of 
$X_{3},\ldots,X_{3n+2}$ is 
\begin{eqnarray*}
  p_{[3:3n+2]}(x_{3},\ldots,x_{3n+2}) = r^{-n+1}v(x_{3},x_{4},x_{5})q_{[1:3n]}(x_{3},\ldots,x_{3n+2})w(x_{3n},x_{3n+1},x_{3n+3}).
\end{eqnarray*}
Thus for all $s\in \mathbb{Z}$ we have that the joint density of $X_{s},\ldots,X_{3n+s-1}$ is 
\begin{eqnarray*}
  p_{[1:3n]}(x_{1},\ldots,x_{3n}) \propto v(x_{1},x_{2},x_{3})q_{[1:3n]}(x_{1},\ldots,x_{3n})w(x_{3n-2},x_{3n-1},x_{3n}).
\end{eqnarray*}  
From this, and the arguments in 
Theorem \ref{theorem:one_stationarity} 
stationarity of the time series immediately follows. 

It is interesting to note that the joint distribution of 
$3n+1$ consecutive time series is
\begin{eqnarray*}
  p_{[0:3n]}(x_{0},\ldots,x_{3n}) \propto v(x_{0},x_{1},x_{2})E_{1}(x_{0},x_{1},x_{2},x_3)
  q_{[1:3n]}(x_{1},\ldots,x_{3n})w(x_{3n-2},x_{3n-1},x_{3n})
\end{eqnarray*}
and the joint distribution of 
$3n+2$ consecutive time series is
\begin{eqnarray*}
  p_{[-1:3n]}(x_{-1},\ldots,x_{3n}) &\propto& v(x_{-1},x_{0},x_{1})E_{2}(x_{-1},x_{0},x_{1},x_2,x_3)
q_{[1:3n]}(x_{1},\ldots,x_{3n})w(x_{3n-2},x_{3n-1},x_{3n}).
\end{eqnarray*}

\subsection{The general $d$-neighbourhood case}

We now generalize the above arguments 
to the $d$-neighbourhood case. Based on the interaction kernel 
$R(x_{[1:d]},x_{[d+1:2d]})$ 
we define $D_{1},D_{2},\ldots,D_{d}$ as follows.  $D_{j}$ contains all interactions in $R$
which contain $x_{j}$ \emph{ and } all 'half' interactions between
$x_{j+1},\ldots,x_{d}$ to $x_{d+j}$. With this we define
\begin{eqnarray*}
  E_{j}(x_{[1:d+j]})  &=& \prod_{i=1}^{j}D_{i}({\bf x}_{[i:i+d]})c_{x_i}^{1/2}c_{x_{i+d}}^{1/2}
\end{eqnarray*}  
and $\widetilde{E}_{j}(x_{1},\ldots,x_{d+j}) = 
E_{j}(x_{d+j},\ldots,x_{1})$.
Using the above, and similar arguments to the 
case $d=3$ we have 
\begin{eqnarray*}
  q_{[0:d(n+1)-1]}({\bf x}_{[0:d(n+1)-1]}) =
E_{j}({\bf x}_{[0:d+j-1]}) q_{[j-1:dn+j-1]}({\bf x}_{[j-1:dn+j-1]}) \widetilde{E}_{d-j}({\bf x}_{[dn-(d-j):d(n+1)-1]}).
\end{eqnarray*}
Using the same argument as in the case $d=2$ and $d=3$, 
there exists a stochastic process 
$\{\xi_{t} = (X_{1,t},\ldots,X_{d,t})\}_{t}$ with the stated conditional specification and joint density
\begin{eqnarray*}
p_{[1:n]}^{\xi}(\xi_{1},\ldots,\xi_{n}) \propto  v(\xi_1) q_{[1:n]}(\xi_{1},\ldots,\xi_{n}) w(\xi_{n}).
\end{eqnarray*} 
We relabel the entries $X_{dt+1}=X_{1,t}$, 
$X_{dt+2}=X_{2,t},\ldots,X_{dt+d}=X_{d,t}$. Thus the joint density of $X_{0},X_{1},\ldots,X_{dn-1}$ is 
\begin{eqnarray*}
p_{[0:dn-1]}(x_{[0:nd-1]}) =  r^{-n+1}v({\bf x}_{[0:d-1]}) q_{[0:dn-1]}({\bf x}_{[0:dn-1]}) w({\bf x}_{[(n-1)d:nd-1]}).
\end{eqnarray*}
And using the same argument as in the case 
$d=3$, the joint density of 
$X_{j},X_{j+1},\ldots,X_{dn+j-1}$ (for $j=1,\ldots,d-1$)
is 
\begin{eqnarray*}
p_{[j:dn+j-1]}({\bf x}_{[j:dn+j-1]}) \propto  
v_{j}({\bf x}_{[j:d+j-1]}) q_{[j:dn+j-1]}({\bf x}_{[j:dn+j-1]}) 
w_{j}({\bf x}_{[(n-1)d+j:dn+j-1]})
\end{eqnarray*}  
with $v_{0} = v$ and $w_{0} = w$ and for $j\in \{1,\ldots,d-1\}$
\begin{eqnarray*}
  v_{j}({\bf x}_{[j:d+j-1]})&=&
\int E_{j}({\bf x}_{[0:d+j-1]})v(x_{[0:d-1]})
\mu(d x_{0})\ldots \mu(d x_{j-1})\\
  w_{j}({\bf x}_{[(n-1)d+j:nd+j-1]}) &=& \int
                               \widetilde{E}_{d-j}({\bf x}_{[dn-(d-j):d(n+1)-1]})\mu(d
                               x_{dn+j})\ldots\mu(d x_{d(n+1)-1}).
\end{eqnarray*}  
But we also have that $T v_{j}\propto v_{j}$ and $T^{*}w_{j}\propto
w_{j}$, this immediately implies that $v_{j}=v$ and $w_{j}=w$ and consequently $\{X_{t}\}_{t}$ is a stationary time series. 

A bi-product of the proof is that the joint distribution of time consecutive random variables has a simple expression. 
In particular, 
the joint distribution of $dn$ consecutive random variables is 
\begin{eqnarray*}
  v({\bf x}_{[1:d]}) q_{[1:dn]}({\bf x}_{[1:dn]}) w({\bf x}_{[(n-1)d+1:nd]}),
\end{eqnarray*}
whereas for $j\in \{1,\ldots d-1\}$ the joint distribution of $dn+j$ consecutive random variables is 
\begin{eqnarray*}
  v({\bf x}_{[0:d-1]})E_{j}({\bf x}_{[0:d+j-1]}) 
  q_{[j:dn+j-1]}({\bf x}_{[j:dn+j-1]}) w({\bf x}_{[(n-1)d+j+1:dn+j]}).
\end{eqnarray*}

\section{Additional derivations}\label{sec:addnl_derivs}
\subsection{Univariate Gaussian with reflective boundary: lack of internal consistency and stationarity}\label{subsec:univ_g_reflect}
Let $|\phi| < 1$. Suppose the joint density $p^{(n)}$ of $(X_1, \ldots, X_n)$ on $\mb R^n$ is constructed based on the following conditional specification (with a reflective boundary condition): 
\begin{align*}
X_t \mid X_{-t} & \sim N\left( \frac{\phi}{1+\phi^2} (X_{t-1} + X_{t+1}), \frac{1}{1+\phi^2} \right), \quad 1 < t < n, \\
X_1 \mid X_{-1} & \sim N\left( \frac{\phi}{1+\phi^2} X_2, \frac{1}{1+\phi^2} \right), \\
X_n \mid X_{-n} & \sim N\left( \frac{\phi}{1+\phi^2} X_{n-1}, \frac{1}{1+\phi^2} \right),
\end{align*}
where $X_{-t} = (X_s \,: \, s \ne t \in [n])$. 
The conditionals (barring the boundaries) correspond to a Gaussian AR(1) process; see Example \ref{ex:AR1}. Using Theorem \ref{theorem:distributionedge}, it follows that 
\begin{align*}
p^{(n)}(x_1, \ldots, x_n) = [c^{(n)}]^{-1} \exp \left[-\frac{1}{2} \left\{
    \sum_{t=1}^{n}(1+\phi^{2})x_{t}^{2}-2\phi\sum_{t=2}^{n}x_{t}x_{t-1}\right\}\right]
\end{align*}
It is immediate that $p^{(n)}$ corresponds to a $n$-variate Gaussian distribution with mean 0 and precision (or inverse covariance) matrix $Q^{(n)}$ with $Q^{(n)}_{ii} = (1 + \phi^2)$; $Q^{(n)}_{ij} = - \phi$ for $|i -j| = 1$ and $0$ for $|i - j| > 1$. 

Observe that we can write 
\begin{align*}
p^{(n+1)}(x_1, \ldots, x_{n+1}) = \frac{c^{(n)}}{c^{(n+1)}} \, p^{(n)}(x_1, \ldots, x_n) \, \exp \left[-\frac{1}{2} \big\{(1 + \phi^2) x_{n+1} - 2 \phi x_n x_{n+1} \big\}\right]. 
\end{align*}
Using $\int_{-\infty}^{\infty} e^{-\frac{1}{2} (q x^2 - 2 bx)} dx = \sqrt{2 \pi} \, q^{-1/2} \, e^{b^2/(2q)}$ for $q > 0$, we obtain
\begin{align*}
\int p^{(n+1)}(x_1, \ldots, x_{n+1}) dx_{n+1} & = \frac{c^{(n)}}{c^{(n+1)}} \, p^{(n)}(x_1, \ldots, x_n) \, \sqrt{2 \pi} \, (1 +\phi^2)^{-1/2} \, e^{\frac{\phi^2 x_n^2}{2(1+\phi^2)}} \\
& \ne p^{(n)}(x_1, \ldots, x_n). 
\end{align*}
Moreover, since $Q^{(n)}$ is a tridiagonal Toeplitz matrix, using a formula for inverse of such matrices (c.f. \cite[Theorem 2.8]{meurant1992review}), it follows that the diagonal entries $(\Sigma^{(n)})_{tt}$ of the covariance matrix $\Sigma^{(n)} :\,= (Q^{(n)})^{-1}$ vary across $t$. This implies the $n$ univariate marginal distributions arising from $p^{(n)}$ are different, or in other words, lack stationarity.

\subsection{Unique positive eigenfunction in Remark \ref{rem:only_positive}}\label{subsec:only_positive}
We show that $v$ is (upto scaling) the only a.e. positive eigenfunction of $T$. Suppose that $f$ is an a.e. positive eigenfunction of $T$ with $f/\|f\| \ne v$. Let $\lambda$ be the corresponding eigenvalue so that $T f = \lambda f$. Since the eigenspace corresponding to $r$ is spanned by $v$, it is clear that $|\lambda| < r$. Moreover, if $f > 0$ a.e., then $Tf > 0$ a.e., and hence it must be the case that $\lambda$ is (real and) positive. Since $f > 0$ a.e., we have $\langle f, w\rangle > 0$, and hence $f$ satisfies the condition for \eqref{eq:non_sa_geom_cgence} to hold in Lemma \ref{lemma:pow_it_main}. Then, $T^k(f)/(r^k \langle f, w\rangle)$ converges to $v$ as $k \to \infty$. On the other hand, $T^k(f) = \lambda^k f$ for any $k \ge 1$, implying $T^k(f)/(r^k \langle f, w\rangle) \to 0$. This leads to a contradiction. 

\subsection{Eigenfunction calculation for Example \ref{ex:AR1}}\label{subsec:eigencalc}
We detail the steps for arriving at the dominant eigenfunction for the interaction kernel $R$ defined in \eqref{eq:kernel_AR1} with $|\phi| \ne 1$. Consider the equation
\begin{align*}
\int_{-\infty}^{\infty} R(x, y) v(x) dx = r \, v(y), \quad y \in \mb R.
\end{align*}
We demonstrate a candidate for $v$ below. By the Krein--Rutman theorem (Theorem \ref{thm:KR_without_int}), this is the unique positive eigenfunction upto scaling, since we know the integral operator corresponding to $R$ is Hilbert--Schmidt, and hence compact, for $|\phi| \ne 1$. 

As a candidate, set 
$$
v(x) \,\propto\, e^{-\frac{1}{2} (qx^2 - 2bx)}
$$
proportional to the density function of a $N(b/q,1/q)$ distribution. Substituting in the previous equation, and after some algebra, we obtain a system of equations $(1+\phi^2) - \phi^2/q = q$ and $b \phi/q = -b$. This implies $b = 0$ and solving the quadratic for $q$, $q^2 - (1+\phi^2)q + \phi^2 = 0$, gives $q = 1$ or $q = \phi^2$, corresponding to the cases $|\phi| < 1$ and $|\phi| > 1$, respectively. 

\subsection{The case $|\phi| = 1$ in Example \ref{ex:AR1}}\label{subsec:not_compact}
We show that the integral operator in Example \ref{ex:AR1} is not compact when $|\phi| = 1$. 
\begin{lemma}
The integral operator $T:L_{2}(\mathbb{R})\rightarrow L_{2}(\mathbb{R})$ given by 
\begin{eqnarray*}
T(f)[x]=\int_{\mathbb{R}} R(x,y)f(y)dy
\end{eqnarray*}
where 
\begin{eqnarray*}
R(x,y) = \exp\left(-\frac{1}{2}x^{2}+xy-\frac{1}{2}y^{2} \right) = 
\exp\left(-\frac{1}{2}(x-y)^{2} \right) 
\end{eqnarray*}
is not compact on $L_{2}(\mathbb{R})$. 
\end{lemma}
\begin{proof}
We recall that an operator is compact in $L_{2}(\mathbb{R})$ if for every
bounded sequence $\{f_{j}\}_{j}$ the sequence 
$\{T(f_{j})\}_{j}$
contains a convergent subsequence. 

To show that $T$ is not compact
 we construct a sequence of functions in $L_{2}(\mathbb{R})$ whose mapping $\{T(f_{j})\}_{j}$
cannot contain a convergent subsequence.
We define the indicator function
\begin{eqnarray*}
I(x) = \left\{
\begin{array}{cc}
1 & x\in [0,1] \\
0 & x\notin [0,1]
\end{array}
\right.
\end{eqnarray*}
Using $I(x)$ we define the sequence of functions
$f_{j}(x) = I(x+j)$. Clearly, $\|f_{j}\|_2=1$ hence 
$\{f_{j}\}_j$ is a bounded sequence in $L_{2}(\mathbb{R})$. Let  
\begin{eqnarray*}
\Phi(x) = \int_{-\infty}^{x}\exp(-y^{2}/2)dy.
\end{eqnarray*}
Using this it is easily seen that
\begin{eqnarray*}
T(f_j)[x] = \Phi(x+j+1) -\Phi(x+j)
\end{eqnarray*}
For any $0<n<m<\infty$, by using a change of variables it can be shown that
\begin{eqnarray*}
\|T(f_m) - T(f_{n})\| = \|T(f_{n-m}) - T(f_{0})\|
\end{eqnarray*}
$\{\|T(f_{n}) - T(f_{0})\|\}_{n}$ is an increasing sequence with 
\begin{eqnarray*}
\|T(f_{n}) - T(f_{0})\|\rightarrow 2\|Tf_0\| \textrm{ as }n\rightarrow \infty.
\end{eqnarray*}
Thus there cannot exist a subsequence $\{T(f_{n_j})\}_j$ which is Cauchy convergent, which proves that $T$ is not a compact operator. 
The principle reason the above holds is that 
$\mathbb{R}$ is not compact, if it were then one would obtain an accumulation point.

\end{proof}

\section{Examples}\label{sec:examples}

\subsection{Trivariate conditional beta}

\paragraph{Example 1} We now give two examples of a trivariate
conditional beta. Both with the same process-level conditional graph
but with different individual graphs. 

The sample space of $x_{t}^{(a)},x_{t}^{(b)},x_{t}^{(c)}\in (0,1)$ and
the conditional specification is
\begin{align*}
   \log p(x_{t}^{(a)}|\mathcal{H}_{(a,t)}) 
   &\propto
  [\alpha^{(a)} - 1 + \phi_{1,2,1}^{(a,a)}s_{2}(x_{t-1}^{(a)})+
    \phi_{1,1,2}^{(a,a)}s_{2}(x_{t+1}^{(a)})+\phi_{1,1,2}^{(a,b)}s_{2}(x_{t+1}^{(b)})]s_{1}(x_{t}^{(a)})
   \nonumber\\
  &  + [\beta^{(a)} - 1 + \phi_{1,1,2}^{(a,a)}s_{1}(x_{t-1}^{(a)})+
    \phi_{1,2,1}^{(a,a)}s_{1}(x_{t+1}^{(a)})+\phi_{1,2,1}^{(a,b)}s_{1}(x_{t+1}^{(b)})]s_{2}(x_{t}^{(a)})
              \end{align*}
 \begin{align*}         
      \log p(x_{t}^{(b)}|\mathcal{H}_{(b,t)}) 
   &\propto
  [\alpha^{(b)} - 1 + \phi_{1,2,1}^{(b,b)}s_{2}(x_{t-1}^{(b)})+
     \phi_{1,1,2}^{(b,b)}s_{2}(x_{t+1}^{(b)})+\phi_{1,2,1}^{(a,b)}s_{2}(x_{t-1}^{(a)})+
     \phi_{1,2,1}^{(c,b)}s_{2}(x_{t-1}^{(c)})]s_{1}(x_{t}^{(b)})
    \\
  &  + [\beta^{(b)} - 1 + \phi_{1,1,2}^{(b,b)}s_{1}(x_{t-1}^{(b)})+
    \phi_{1,2,1}^{(b,b)}s_{1}(x_{t+1}^{(b)})+\phi_{1,1,2}^{(a,b)}s_{1}(x_{t-1}^{(a)})+
    \phi_{1,1,2}^{(c,b)}s_{1}(x_{t-1}^{(c)})]s_{2}(x_{t}^{(b)}).
  \end{align*}
\begin{align*}
   \log p(x_{t}^{(c)}|\mathcal{H}_{(c,t)}) 
   &\propto
  [\alpha^{(c)} - 1 + \phi_{1,2,1}^{(c,c)}s_{2}(x_{t-1}^{(c)})+
    \phi_{1,1,2}^{(c,c)}s_{2}(x_{t+1}^{(c)})+\phi_{1,1,2}^{(c,b)}s_{2}(x_{t+1}^{(b)})]s_{1}(x_{t}^{(c)})
   \nonumber\\
  &  + [\beta^{(c)} - 1 + \phi_{1,1,2}^{(c,c)}s_{1}(x_{t-1}^{(c)})+
    \phi_{1,2,1}^{(c,c)}s_{1}(x_{t+1}^{(c)})+\phi_{1,2,1}^{(c,b)}s_{1}(x_{t+1}^{(b)})]s_{2}(x_{t}^{(c)})
              \end{align*}

In this example, ${\boldsymbol \theta} =
(\alpha^{(a)}-1,\beta^{(a)}-1, \alpha^{(b)}-1,\beta^{(b)}-1, \alpha^{(c)}-1,\beta^{(c)}-1)$,
${\bf s}(x) =
(s_{1}(x^{(a)}),s_{2}(x^{(a)}),s_{1}(x^{(b)}),s_{2}(x^{(b)}), s_{1}(x^{(c)}),s_{2}(x^{(c)}))^{\top}$,
the block matrices are  $\Psi_{0}=0$,
\begin{eqnarray*}
 \Psi_{1}^{} =
  \left(
  \begin{array}{ccc}
    \Psi_{1}^{(a,a)}  & \Psi_{1}^{(a,b)}  & 0 \\
    0 & \Psi_{1}^{(b,b)}  & 0 \\
    0 & \Psi_{1}^{(c,b)} & \Psi_{1}^{(c,c)}
  \end{array}  
  \right) \textrm{ and }
  \Psi_{-1}^{} =
  \left(
  \begin{array}{ccc}
    (\Psi_{1}^{(a,a)})^{\top}  & 0  & 0\\
    ( \Psi_{1}^{(a,b)})^{\top} & (\Psi_{1}^{(b,b)})^{\top} & (\Psi_{1}^{(c,b)})^{\top} \\
       0 & 0 &(\Psi_{1}^{(c,c)})^{\top}   \\
  \end{array}  
  \right) 
\end{eqnarray*}
where
\begin{eqnarray*}
  \Psi_{1}^{(a,a)} &=&
  \left(
  \begin{array}{cc}
    0  & \phi_{1,1,2}^{(a,a)} \\
   \phi_{1,2,1}^{(a,a)}   & 0 \\ 
  \end{array}  
  \right),
   \Psi_{1}^{(b,b)} =\left(
  \begin{array}{cc}
    0  & \phi_{1,1,2}^{(b,b)} \\
   \phi_{1,2,1}^{(b,b)}   & 0 \\ 
  \end{array}  
  \right),
   \Psi_{1}^{(c,c)} =\left(
  \begin{array}{cc}
    0  & \phi_{1,1,2}^{(c,c)} \\
   \phi_{1,2,1}^{(c,c)}   & 0 \\ 
  \end{array}  
  \right) \\
   \Psi_{1}^{(a,b)} &=&\left(
  \begin{array}{cc}
    0  & \phi_{1,1,2}^{(a,b)} \\
   \phi_{1,2,1}^{(a,b)}   & 0 \\ 
  \end{array}  
  \right),
  \Psi_{1}^{(c,b)} =\left(
  \begin{array}{cc}
    0  & \phi_{1,1,2}^{(c,b)} \\
   \phi_{1,2,1}^{(c,b)}   & 0 \\ 
  \end{array}  
  \right)
\end{eqnarray*}
If all the entries in $\Psi_{1}$ are negative and
$\alpha^{(a)}, \alpha^{(b)}, \alpha^{(c)}, \beta^{(a)}, \beta^{(b)},
\beta^{(c)}>0$, then the HS-condition is satisfied.

The process-wide conditional independence graph and the individual
level conditional independence graph is given below. 
\begin{center}
  \includegraphics[scale =0.4]{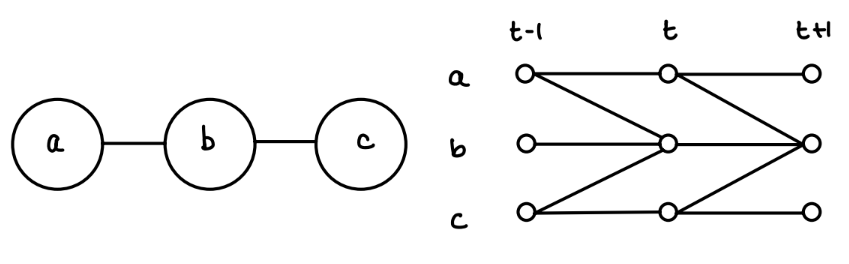}
\end{center}

A realisation using the method described in Section
\ref{sec:prob_cestgm} is given below.
\begin{center}
  \includegraphics[scale =0.3]{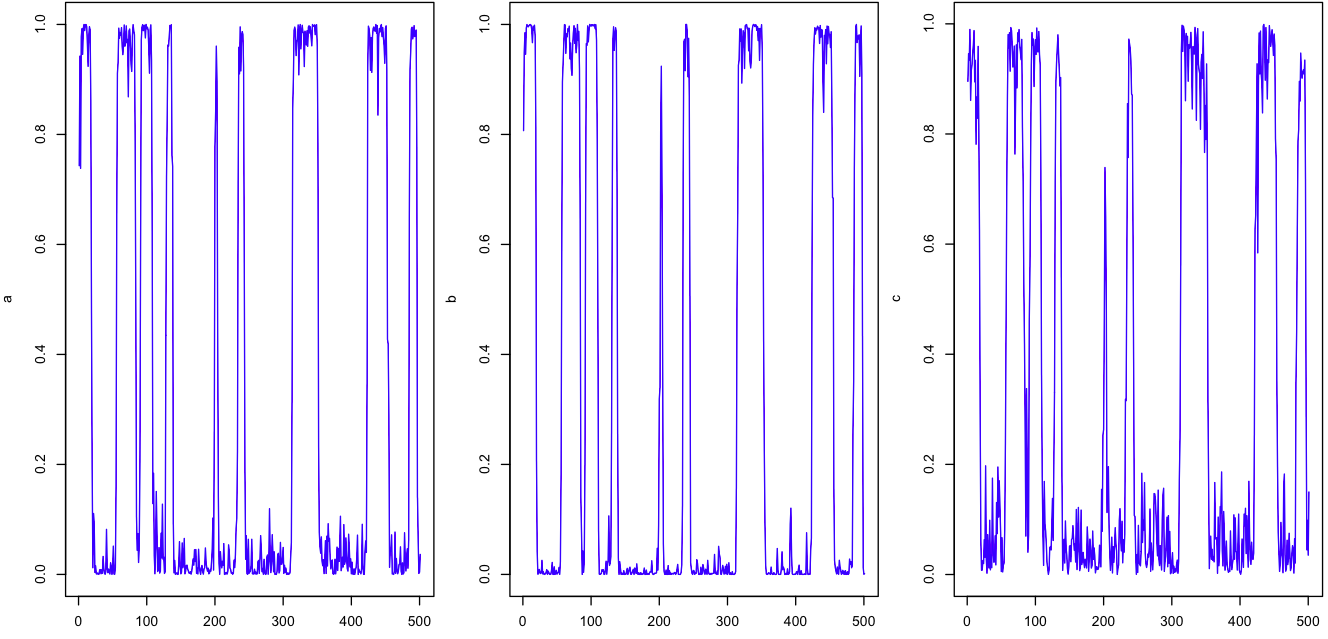}
 \end{center}

\paragraph{Example 2} Another trivariate conditional beta:

The sample space of $x_{t}^{(a)},x_{t}^{(b)},x_{t}^{(c)}\in (0,1)$ and
the conditional specification is
\begin{align*}
   \log p(x_{t}^{(a)}|\mathcal{H}_{(a,t)}) 
   &\propto
  [\alpha^{(a)} - 1 + \phi_{1,2,1}^{(a,a)}s_{2}(x_{t-1}^{(a)})+
    \phi_{1,1,2}^{(a,a)}s_{2}(x_{t+1}^{(a)})+\phi_{0,1,2}^{(a,b)}s_{2}(x_{t}^{(b)})]s_{1}(x_{t}^{(a)})
   \nonumber\\
  &  + [\beta^{(a)} - 1 + \phi_{1,1,2}^{(a,a)}s_{1}(x_{t-1}^{(a)})+
    \phi_{1,2,1}^{(a,a)}s_{1}(x_{t+1}^{(a)})+\phi_{0,2,1}^{(a,b)}s_{1}(x_{t}^{(b)})]s_{2}(x_{t}^{(a)})
              \end{align*}
  \begin{align*}         
      \log p(x_{t}^{(b)}|\mathcal{H}_{(b,t)}) 
   &\propto
  [\alpha^{(b)} - 1 + \phi_{1,2,1}^{(b,b)}s_{2}(x_{t-1}^{(b)})+
     \phi_{1,1,2}^{(b,b)}s_{2}(x_{t+1}^{(b)})+\phi_{0,2,1}^{(a,b)}s_{2}(x_{t}^{(a)})+
     \phi_{0,2,1}^{(c,b)}s_{2}(x_{t}^{(c)})]s_{1}(x_{t}^{(b)})
    \\
  &  + [\beta^{(b)} - 1 + \phi_{1,1,2}^{(b,b)}s_{1}(x_{t-1}^{(b)})+
    \phi_{1,2,1}^{(b,b)}s_{1}(x_{t+1}^{(b)})+\phi_{0,1,2}^{(a,b)}s_{1}(x_{t}^{(a)})+
    \phi_{0,1,2}^{(c,b)}s_{1}(x_{t}^{(c)})]s_{2}(x_{t}^{(b)}).
  \end{align*}
\begin{align*}
   \log p(x_{t}^{(c)}|\mathcal{H}_{(c,t)}) 
   &\propto
  [\alpha^{(c)} - 1 + \phi_{1,2,1}^{(c,c)}s_{2}(x_{t-1}^{(c)})+
    \phi_{1,1,2}^{(c,c)}s_{2}(x_{t+1}^{(c)})+\phi_{0,1,2}^{(c,b)}s_{2}(x_{t}^{(b)})]s_{1}(x_{t}^{(c)})
   \nonumber\\
  &  + [\beta^{(c)} - 1 + \phi_{1,1,2}^{(c,c)}s_{1}(x_{t-1}^{(c)})+
    \phi_{1,2,1}^{(c,c)}s_{1}(x_{t+1}^{(c)})+\phi_{0,2,1}^{(c,b)}s_{1}(x_{t}^{(b)})]s_{2}(x_{t}^{(c)})
              \end{align*}

In this example, ${\boldsymbol \theta} =
(\alpha^{(a)}-1,\beta^{(a)}-1, \alpha^{(b)}-1,\beta^{(b)}-1, \alpha^{(c)}-1,\beta^{(c)}-1)$,
${\bf s}(x) =
(s_{1}(x^{(a)}),s_{2}(x^{(a)}),s_{1}(x^{(b)}),s_{2}(x^{(b)}), s_{1}(x^{(c)}),s_{2}(x^{(c)}))^{\top}$,
the block matrices are 
\begin{eqnarray*}
\Psi_{0} =  \left(
  \begin{array}{ccc}
    0  & \Psi_{0}^{(a,b)}  & 0 \\
   (\Psi_{0}^{(a,b)})^{\top}  & 0  &  (\Psi_{0}^{(c,b)})^{\top}\\
    0 & \Psi_{0}^{(c,b)} & 0
  \end{array}  
  \right),
 \Psi_{1}^{} =
  \left(
  \begin{array}{ccc}
    \Psi_{1}^{(a,a)}  & 0  & 0 \\
    0 & \Psi_{1}^{(b,b)}  & 0 \\
    0 & 0 & \Psi_{1}^{(c,c)}
  \end{array}  
  \right) \textrm{ and }
  \Psi_{-1}^{} = \Psi_{1}^{\top}
\end{eqnarray*}
where
\begin{eqnarray*}
  \Psi_{1}^{(a,a)} &=&
  \left(
  \begin{array}{cc}
    0  & \phi_{1,1,2}^{(a,a)} \\
   \phi_{1,2,1}^{(a,a)}   & 0 \\ 
  \end{array}  
  \right),
   \Psi_{1}^{(b,b)} =\left(
  \begin{array}{cc}
    0  & \phi_{1,1,2}^{(b,b)} \\
   \phi_{1,2,1}^{(b,b)}   & 0 \\ 
  \end{array}  
  \right),
   \Psi_{1}^{(c,c)} =\left(
  \begin{array}{cc}
    0  & \phi_{1,1,2}^{(c,c)} \\
   \phi_{1,2,1}^{(c,c)}   & 0 \\ 
  \end{array}  
  \right) \\
   \Psi_{0}^{(a,b)} &=&\left(
  \begin{array}{cc}
    0  & \phi_{0,1,2}^{(a,b)} \\
   \phi_{0,2,1}^{(a,b)}   & 0 \\ 
  \end{array}  
  \right),
  \Psi_{0}^{(c,b)} =\left(
  \begin{array}{cc}
    0  & \phi_{0,1,2}^{(c,b)} \\
   \phi_{0,2,1}^{(c,b)}   & 0 \\ 
  \end{array}  
  \right)
\end{eqnarray*}
If all the entries in $\Psi_0$ and $\Psi_{1}$ are negative and
$\alpha^{(a)}, \alpha^{(b)}, \alpha^{(c)}, \beta^{(a)}, \beta^{(b)},
\beta^{(c)}>0$, then the HS-condition is satisfied.

The process-wide conditional independence graph and the individual
level conditional independence graph is given below.

\begin{center}
  \includegraphics[scale =0.4]{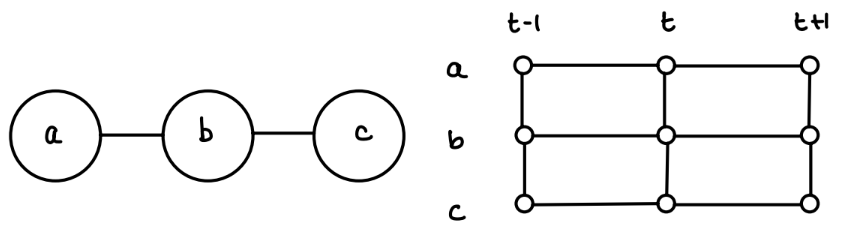}
 \end{center} 

 \subsection{Bivariate conditional beta and conditional binomial}

 The sample space of $x_{t}^{(a)},x_{t}^{(b)}\in (0,1)$,
 $x_{t}^{(c)}\in \{0,1,\ldots,n\}$. We model
 the conditional distributions of $x_{t}^{(a)}$ and $x_{t}^{(b)}$ as a beta distribution 
\begin{align*}
   \log p(x_{t}^{(a)}|\mathcal{H}_{(a,t)}) 
   &\propto
  [\alpha^{(a)} - 1 + \phi_{1,2,1}^{(a,a)}s_{2}(x_{t-1}^{(a)})+
    \phi_{1,1,2}^{(a,a)}s_{2}(x_{t+1}^{(a)})+\phi_{1,1,2}^{(a,b)}s_{2}(x_{t+1}^{(b)})]s_{1}(x_{t}^{(a)})
   \nonumber\\
  &  + [\beta^{(a)} - 1 + \phi_{1,1,2}^{(a,a)}s_{1}(x_{t-1}^{(a)})+
    \phi_{1,2,1}^{(a,a)}s_{1}(x_{t+1}^{(a)})+\phi_{1,2,1}^{(a,b)}s_{1}(x_{t+1}^{(b)})]s_{2}(x_{t}^{(a)})
              \end{align*}
  \begin{align*}         
      \log p(x_{t}^{(b)}|\mathcal{H}_{(b,t)}) 
   &\propto
  [\alpha^{(b)} - 1 + \phi_{1,2,1}^{(b,b)}s_{2}(x_{t-1}^{(b)})+
     \phi_{1,1,2}^{(b,b)}s_{2}(x_{t+1}^{(b)})+\phi_{1,2,1}^{(a,b)}s_{2}(x_{t-1}^{(a)})+
     \phi_{1,2,1}^{(c,b)}x_{t-1}^{(c)}]s_{1}(x_{t}^{(b)})
    \\
  &  + [\beta^{(b)} - 1 + \phi_{1,1,2}^{(b,b)}s_{1}(x_{t-1}^{(b)})+
    \phi_{1,2,1}^{(b,b)}s_{1}(x_{t+1}^{(b)})+\phi_{1,1,2}^{(a,b)}s_{1}(x_{t-1}^{(a)})+
    \phi_{1,1,2}^{(c,b)}x_{t-1}^{(c)}]s_{2}(x_{t}^{(b)}).
  \end{align*}
  and the conditional distribution of $x_{t}^{(c)}$ as a Binomial
\begin{align*}
   \log p(x_{t}^{(c)}|\mathcal{H}_{(c,t)}) 
   &\propto [\theta+\phi_{1}x_{t-1}^{(c)}+\phi_{1}x_{t+1}^{(c)} +
     \phi_{1,2,1}^{(c,b)}s_{1}(x_{t+1}^{(b)}) + \phi_{1,1,2}^{(c,b)}s_{2}(x_{t+1}^{(b)})]x_{t}^{(c)}
+\log {n\choose x_{t}^{(c)}}.
              \end{align*}

In this example, ${\boldsymbol \theta} =
(\alpha^{(a)}-1,\beta^{(a)}-1, \alpha^{(b)}-1,\beta^{(b)}-1, \theta)$,
\\*
${\bf s}(x) =
(s_{1}(x^{(a)}),s_{2}(x^{(a)}),s_{1}(x^{(b)}),s_{2}(x^{(b)}), x^{(c)})^{\top}$,
the block matrices are  $\Psi_{0}=0$,
\begin{eqnarray*}
 \Psi_{1}^{} =
  \left(
  \begin{array}{ccc}
    \Psi_{1}^{(a,a)}  & \Psi_{1}^{(a,b)}  & 0 \\
    0 & \Psi_{1}^{(b,b)}  & 0 \\
    0 & \Psi_{1}^{(c,b)} & \Psi_{1}^{(c,c)}
  \end{array}  
  \right) \textrm{ and }
  \Psi_{-1}^{} =
  \left(
  \begin{array}{ccc}
    (\Psi_{1}^{(a,a)})^{\top}  & 0  & 0\\
    ( \Psi_{1}^{(a,b)})^{\top} & (\Psi_{1}^{(b,b)})^{\top} & (\Psi_{1}^{(c,b)})^{\top} \\
       0 & 0 &(\Psi_{1}^{(c,c)})^{\top}   \\
  \end{array}  
  \right) 
\end{eqnarray*}
where
\begin{eqnarray*}
  \Psi_{1}^{(a,a)} &=&
  \left(
  \begin{array}{cc}
    0  & \phi_{1,1,2}^{(a,a)} \\
   \phi_{1,2,1}^{(a,a)}   & 0 \\ 
  \end{array}  
  \right),
   \Psi_{1}^{(b,b)} =\left(
  \begin{array}{cc}
    0  & \phi_{1,1,2}^{(b,b)} \\
   \phi_{1,2,1}^{(b,b)}   & 0 \\ 
  \end{array}  
  \right),
   \Psi_{1}^{(c,c)} = \phi_{1} \\
   \Psi_{1}^{(a,b)} &=&\left(
  \begin{array}{cc}
    0  & \phi_{1,1,2}^{(a,b)} \\
   \phi_{1,2,1}^{(a,b)}   & 0 \\ 
  \end{array}  
  \right),
  \Psi_{1}^{(c,b)} =\left(
  \begin{array}{cc}
     \phi_{1,2,1}^{(c,b)}  & \phi_{1,1,2}^{(c,b)} \\
  \end{array}  
  \right)
\end{eqnarray*}
If all the entries in $\Psi_{1}^{(a,a)}$, $\Psi_{1}^{(b,b)}$,
$\Psi_{1}^{(a,b)}$ are negative, all the entries in
$\Psi_{1}^{(c,b)}$ are positive, $\Phi_{1}^{(c)}$ and $\theta$ can take any real
values and 
$\alpha^{(a)}, \alpha^{(b)}, \beta^{(a)}, \beta^{(b)}>0$, then the HS-condition is satisfied.

The process-wide and individual level graphs are the same as that in the
trivariate conditonal beta in example 1, above. 
A realisation using the method described in Section
\ref{sec:prob_cestgm} is given below.

\begin{center}
  \includegraphics[scale =0.3]{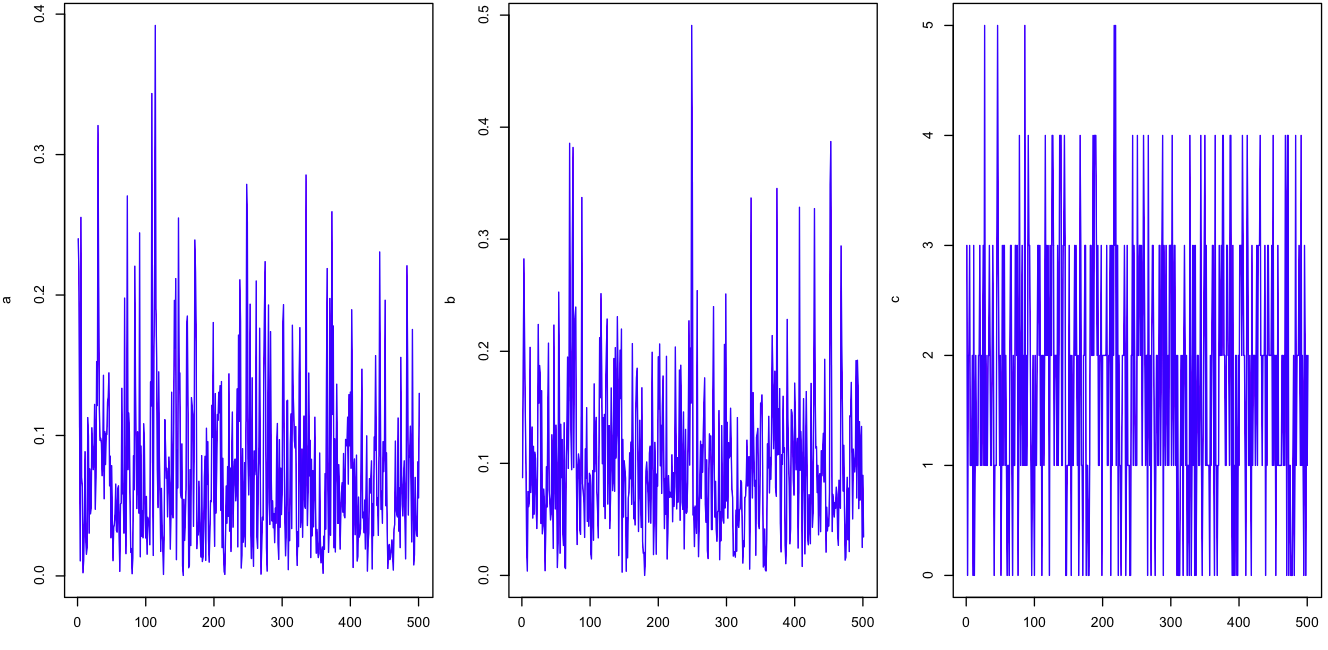}
 \end{center}

\end{document}